\documentclass[11pt,a4paper]{amsart}
\usepackage[pdftex]{graphicx}
\usepackage[pdftex]{xcolor}
\usepackage{enumitem}
\usepackage{graphicx}
\usepackage{hyperref}
\usepackage{mathtools,  stmaryrd}
\usepackage{xparse} \DeclarePairedDelimiterX{\Iintv}[1]{\llbracket}{\rrbracket}{\iintvargs{#1}}
\NewDocumentCommand{\iintvargs}{>{\SplitArgument{1}{,}}m}
{\iintvargsaux#1} 
\NewDocumentCommand{\iintvargsaux}{mm} {#1\mkern1.5mu,\mkern1.5mu#2}

\makeatletter
\newtheorem*{rep@theorem}{\rep@title}
\newcommand{\newreptheorem}[2]{
\newenvironment{rep#1}[1]{
 \def\rep@title{#2~\ref{##1}}
 \begin{rep@theorem}}
 {\end{rep@theorem}}}
\makeatother

\definecolor{RedOrange}{cmyk} {0, 0.77, 0.87, 0}
\definecolor{RoyalPurple}{cmyk} {0.84, 0.53, 0, 0}
\definecolor{YellowGreen}{cmyk} {0.44, 0, 0.74, 0}
\definecolor{Fuchsia}{cmyk} {0.47, 0.91, 0, 0.08}
\definecolor{Blue}{cmyk} {0.84, 0.53, 0, 0}
\definecolor{BlueViolet}{cmyk} {0.84, 0.53, 0, 0}
\definecolor{Black}{cmyk} {0.75, 0.68, 0.67, 0.9}
\usepackage{xcolor}
\usepackage{verbatim}
\usepackage{amsmath}
\usepackage{amssymb, mathrsfs}
\usepackage{amsbsy}

\usepackage{amscd}
\usepackage{amsthm}

\usepackage[T1]{fontenc}
\usepackage[utf8]{inputenc}

\usepackage{color}

\newcommand{\R}{\mathbb{R}}

\newcommand{\B}{\mathbb{B}}

\newcommand{\N}{\mathbb{N}}

\newcommand{\E}{\mathbb{E}}
\newcommand{\Z}{\mathbb{Z}}

\renewcommand{\P}{\mathbb{P}}
\newcommand{\prob}{\mathbb{P}}

\newcommand{\Cyl}{\mathrm{Cyl}}

\newcommand{\lin}{\left[\kern-0.15em\left[}
\newcommand{\rin} {\right]\kern-0.15em\right]}
\newcommand{\linf}{[\kern-0.15em [}
\newcommand{\rinf} {]\kern-0.15em ]}
\newcommand{\ilin}{\left]\kern-0.15em\left]}
\newcommand{\irin} {\right[\kern-0.15em\right[}

\def\al#1{\begin{align*}#1\end{align*}}
\def\aln#1{\begin{align}#1\end{align}}

\renewcommand{\hat}{\widehat}
\renewcommand{\tilde}{\widetilde}

\newtheorem{lem}{Lemma}[section]

\newtheorem{prop}[lem]{Proposition}
\newtheorem{thm}[lem]{Theorem}

\newtheorem{cor}[lem]{Corollary}

\newtheorem{Def}[lem]{Definition}

\newtheorem {rem}[lem] {Remark}

\newtheorem {assum} {Assumption}

\newcounter{assu}
\usepackage{color}
\definecolor{lilas}{RGB}{182, 102, 210}

\newcommand{\CO}{\color{black}}
\numberwithin{equation}{section}

\newcommand{\K}{K}

\author{Wai-Kit Lam}
\address[Wai-Kit Lam]{Department of Mathematics, National Taiwan University, No. 1, Sec. 4, Roosevelt Rd., Taipei 10617, Taiwan}
\email{waikitlam(at)ntu.edu.tw}

\author{Shuta Nakajima} 
\address[Shuta Nakajima]
{Keio University, Kanagawa, Japan}
\email{njima(at)keio.jp}


\title[Moderate deviations in first-passage percolation for bounded weights]{Moderate deviations in first-passage percolation for bounded weights}
\date{\today}

\begin{document}
\begin{abstract}
We investigate the moderate and large deviations in first-passage percolation (FPP) with bounded weights on $\mathbb{Z}^d$ for $d \geq 2$. Write $T(\mathbf{x}, \mathbf{y})$ for the first-passage time and denote by $\mu(\mathbf{u})$ the time constant in direction $\mathbf{u}$. In this paper, we establish that, if one assumes that the sublinear error term $T(\mathbf{0}, N\mathbf{u}) - N\mu(\mathbf{u})$ is of order $N^\chi$, then under some unverified (but widely believed) assumptions, for $\chi < a < 1$, \begin{align*}
  &\mathbb{P}\bigl(T(\mathbf{0}, N\mathbf{u}) > N\mu(\mathbf{u}) + N^a\bigr) = \exp{\Big(-\,N^{\frac{d(1+o(1))}{1-\chi}(a-\chi)}\Big)},\\
  &\mathbb{P}\bigl(T(\mathbf{0}, N\mathbf{u}) < N\mu(\mathbf{u}) - N^a\bigr) = \exp{\Big(-\,N^{\frac{1+o(1)}{1-\chi}(a-\chi)}\Big)}, \end{align*} with accompanying estimates in the borderline case $a=1$. Moreover, the exponents $\frac{d}{1-\chi}$ and $\frac{1}{1-\chi}$ also appear in the asymptotic behavior near $0$ of the rate functions for upper and lower tail large deviations. Notably, some of our estimates are established rigorously without relying on any unverified assumptions. Our main results highlight the interplay between fluctuations and the decay rates of large deviations, and bridge the gap between these two regimes.

  A key ingredient of our proof is an improved concentration via multi-scale analysis for several moderate deviation estimates, a phenomenon that has previously appeared in the contexts of two-dimensional last-passage percolation and two-dimensional rotationally invariant FPP.

\end{abstract}

\maketitle

\section{Model and main results}

\subsection{Introduction}

 First-passage percolation (FPP), first introduced in \cite{hammersley1965}, is a fundamental model in stochastic geometry that describes the random propagation of interfaces through disordered media. Its study is not only of deep intrinsic mathematical interest but also has numerous applications in modeling nonequilibrium growth, fluid flow in porous media, and scaling phenomena in statistical mechanics. Furthermore,  FPP in two dimensions is expected to fall into the ($1+1$) Kardar–Parisi–Zhang (KPZ) universality class, and in particular it serves as a crucial link between probability theory and statistical physics. Physicists even believe that for each  dimension, there is a universal structure. See \cite{barabasi1995fractal} for the detailed discussions.

The FPP model is set on the lattice $\mathbb{Z}^d$, where $d\geq 2$. To each nearest-neighbor edge $e\in E(\mathbb{Z}^d)$ we assign a nonnegative random weight $\tau_e$, representing the time needed to traverse that edge. The collection $(\tau_e)$ is assumed to be independent and identically distributed (i.i.d.), and nondegenerate. Throughout this paper, we will always assume that
\begin{equation*}
\mathbb{P}(\tau_e=0) < p_c(d),\quad \exists C>0,\,\mathbb{P}(\tau_e\leq C)=1,
\end{equation*}
where $p_c(d)$ is the critical probability for Bernoulli bond percolation on $\mathbb{Z}^d$.

For vertices $\mathbf{x},\mathbf{y}\in\mathbb{Z}^d$, the \emph{first-passage time} between $\mathbf{x}$ and $\mathbf{y}$ is defined by
\[
T(\mathbf{x},\mathbf{y}) := \inf_{\gamma:\, \mathbf{x}\to\mathbf{y}} \sum_{e\in\gamma} \tau_e =:\inf_{\gamma:\, \mathbf{x}\to\mathbf{y}} T(\gamma),
\]
where the infimum is over all lattice paths connecting $\mathbf{x}$ and $\mathbf{y}$. For convenience, we will extend the domain of $T$ to $\R^d\times \R^d$: if $\mathbf{x}, \mathbf{y} \in \R^d$, we define $T(\mathbf{x}, \mathbf{y}) = T(\lfloor \mathbf{x}\rfloor, \lfloor \mathbf{y}\rfloor)$, where $\lfloor\mathbf{x}\rfloor$ is the unique point in $\Z^d$ such that $\mathbf{x} \in \lfloor \mathbf{x}\rfloor + [0,1)^d$ ($\lfloor \mathbf{y}\rfloor $ is defined similarly).  By Kingman's subadditive ergodic theorem  \cite{kingman1973}, for each direction $\mathbf{x} \in \mathbb{Z}^d$, there exists a deterministic constant $\mu(\mathbf{x})$, called the time constant, such that
\aln{
\lim_{N\to\infty} \frac{T(\mathbf{0},N\mathbf{x})}{N} = \mu(\mathbf{x}) \quad \text{almost surely and in } L^1.\notag
}
The function $\mu$ can be uniquely extended to the whole $\mathbb{R}^d$ in a natural way, such that $\mu$ is a norm on $\R^d$; see \cite[Section~2.1]{50years} for the details. We define the \emph{limit shape} $\B_d:=\{\mathbf{x}\in \R^d:\mu(\mathbf{x})\leq 1\}$.

A central challenge in FPP is to understand the sublinear correction term
\[
T(\mathbf{0},N\mathbf{x}) - N\mu(\mathbf{x}),
\]
which is conjectured to fluctuate on the order of $N^\chi$ for some $\chi>0$.  The exponent $\chi=\chi(d,\mathbf{x})$, if it exists, is called the \emph{fluctuation exponent}. It is expected that $\chi$ is ``universal'': the value of $\chi$ should be the same for a large class of edge-weight distributions and independent of the direction $\mathbf{x}$.

Motivated by this perspective, we would like to rigorously analyze the upper- and lower-tail \emph{moderate deviations} in FPP. Assume that the fluctuation exponent $\chi$ exists in an appropriate sense. For an exponent \(a\) satisfying
$
\chi < a < 1,
$ 
we aim to determine the asymptotic behaviors of
\begin{equation}
\label{eq: MDP}
\mathbb{P}\Bigl(T(\mathbf{0},N\mathbf{x}) \;>\; N\,\mu(\mathbf{x}) \;+\; N^a\Bigr)\quad \text{and} \quad \mathbb{P}\Bigl(T(\mathbf{0},N\mathbf{x}) \;<\; N\,\mu(\mathbf{x}) \;-\; N^a\Bigr)
\end{equation}
as \(N\to\infty\). We refer to the borderline case \(a=1\) as the \emph{large deviation} regime.

The main objective of this paper is to establish moderate deviation estimates, and also study the behavior of the rate function near zero, under suitable assumptions on the fluctuation exponent \(\chi\) and on the geometry of the limit shape $\B_d$. We emphasize that some of these assumptions have not been fully verified, but they are widely believed to hold for a broad class of distributions.  These assumptions include both concentration conditions (Assumptions~\ref{assum: initial concentration},  \ref{cond: full concentration}, \ref{ass: lower tail concentration}, and \ref{ass: lower bound for lower tail deviations}) and curvature requirements on the limit shape (Assumptions~\ref{assum: finite curvature} and \ref{ass: positive curvature}). In particular, they ensure that the  fluctuations of \(T(\mathbf{0},N\mathbf{x})\) around $N\mu(\mathbf{x})$ can be controlled up to the moderate deviation scale \(N^a\) for \(a\in(\chi,1)\). Assuming these, the main results provide a rigorous  derivation of the decay rates for general dimensions, bridging the gap between fluctuations and the classical large deviation framework.

    \subsection{Notations}
We now introduce some notation used in the subsequent sections.

\textbf{Norms.} For a vector $\mathbf{x}$, $|\mathbf{x}|$ denotes its Euclidean norm, and $|\mathbf{x}|_\infty$ denotes its $\ell^\infty$-norm.

    \textbf{Tangent hyperplane and orthonormal basis with respect to the limit shape.}
Recall the limit shape $\B_d:=\{\mathbf{x}\in \R^d:\mu(\mathbf{x})\leq 1\}$. Given a vector $\mathbf{u}\in \mathbb{R}^d\setminus \{\mathbf{0}\}$, denote by $H_{\mathbf{u}}$ a tangent hyperplane to the boundary of $\mu(\mathbf{u})\mathbb{B}_d$, with ties broken by a predetermined rule. Moreover, we fix an orthonormal basis $\{\tilde{\mathbf{u}}_1,\dots,\tilde{\mathbf{u}}_d\}$ satisfying
\begin{equation}\label{def: utilde}
\mathbf{u} + \mathrm{span}\{\tilde{\mathbf{u}}_2,\dots,\tilde{\mathbf{u}}_d\} = H_{\mathbf{u}} \quad \text{and} \quad \tilde{\mathbf{u}}_1\cdot\mathbf{u} > 0.
\end{equation}
The choice of $\{\tilde{\mathbf{u}}_1,\dots,\tilde{\mathbf{u}}_d\}$ is not unique, and we will always fix a predetermined rule to construct such an orthonormal basis for each $\mathbf{u}$. When $\mathbf{u} = \mathbf{e}_1$, we will simply set $\tilde{\mathbf{u}}_i = \mathbf{e}_i$, where $(\mathbf{e}_i)$ is the standard basis for $\mathbb{R}^d$. We adopt this convention throughout the paper.

\textbf{Tilted cylinders.} Fix $\mathbf{u}\in \R^d\setminus \{\mathbf{0}\}$. Given $\mathbf{v}\in \R^d\setminus \{\mathbf{0}\}$, $I\subset \R$, and $h>0$, we define a \emph{tilted cylinder} in direction $\mathbf{v}$  and height $h$ as
 \al{
 {\rm Cyl}_{\mathbf{v}}(I,h):=  {\rm Cyl}^{\mathbf{u}}_{\mathbf{v}}(I,h) := \Big\{ y_1 {\mathbf{v}}+  \sum_{i=2}^d y_i \tilde{\mathbf{u}}_i :~\,y_1\in I,\,|y_i|\leq h\Big\}.
 }
 We write ${\rm Cyl}_{\mathbf{x},\mathbf{v}}(I,h):=\mathbf{x}+  {\rm Cyl}_{\mathbf{v}}(I,h)$  for the shift by $\mathbf{x}\in  \Z^d$.

\begin{figure}[t]
    \centering
    \includegraphics[width=0.35\linewidth]{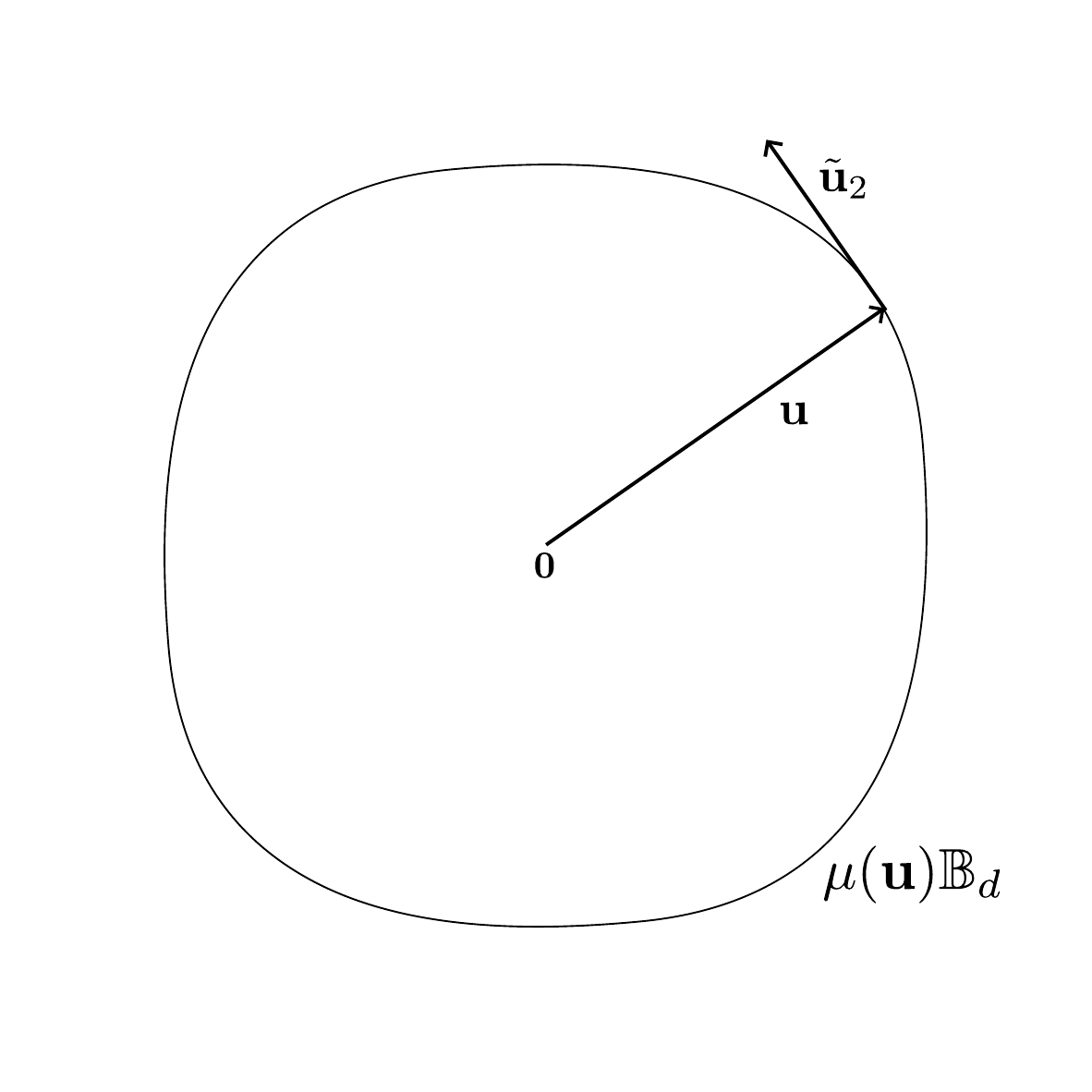}    \includegraphics[width=0.6\linewidth]{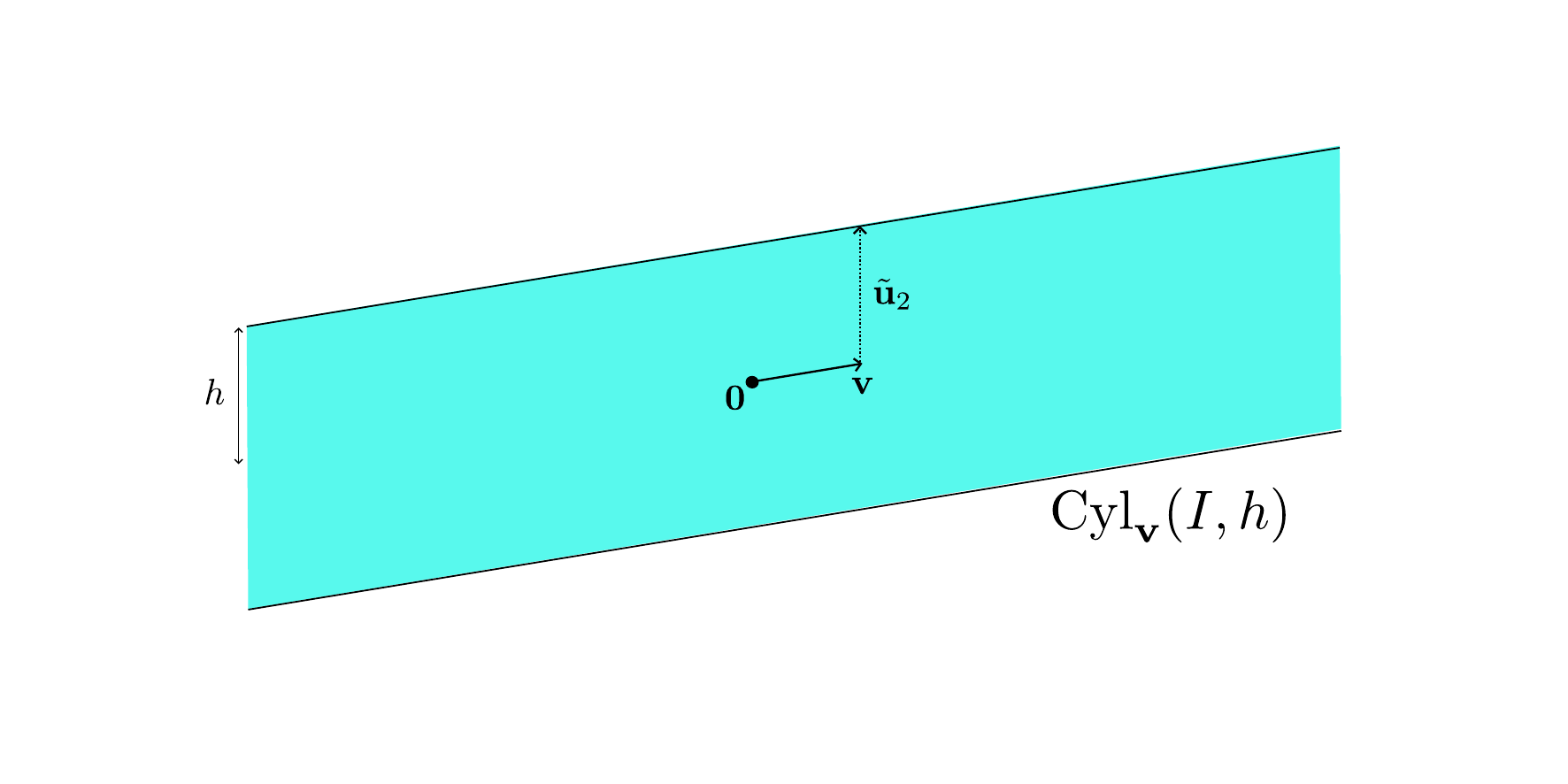} 
    \caption{A depiction of the limit shape (left) and a tilted cylinder (right). Here, the boundary of the limit shape is smooth, and at any direction $\mathbf{u}$ we can find a unique tangent plane $H_\mathbf{u}$ that touches $\mathbf{u}$ at the boundary. In this case, $\tilde{\mathbf{u}}_2$ is a unit vector in $H_\mathbf{u}$. As for the cylinder, note that the directions other than $\mathbf{v}$ are spanned by the $\tilde{\mathbf{u}}_j$, which might not be orthogonal to $\mathbf{v}$, and, in particular, a tilted cylinder may not be a rotated rectangle.}
    \label{fig:ballcylinder}
\end{figure}
\textbf{Geodesics.}
We denote the set of all minimizing paths between $A,B\subset \R^d$, called \emph{geodesics}, by
$$\mathbb{G}(A,B):=\{\gamma:~\gamma \text{ is a path from $\lfloor \mathbf{x}\rfloor$ to $\lfloor \mathbf{y}\rfloor$},~\mathbf{x}\in A,\,\mathbf{y}\in B,\,T(\gamma)=T(\lfloor \mathbf{x}\rfloor, \lfloor \mathbf{y}\rfloor)\}.$$
{Note that this set is almost surely nonempty under our current assumptions on the edge weights.}

\textbf{Restricted first-passage time.}
 If $A\subset \R^d$, we write
 \[
 T_A(\mathbf{x}, \mathbf{y}) := \inf\left\{\sum_{e\in \gamma} \tau_e: \text{$\gamma$ is a path in $A$ from $\mathbf{x}$ to $\mathbf{y}$}\right\}
 \]
 for the first-passage time restricted to $A$. If further $C, D\subset \mathbb{R}^d$, we write $T_A(C, D) = \inf_{\mathbf{x} \in C, \mathbf{y}\in D} T_A(\mathbf{x}, \mathbf{y})$.

  \textbf{Integer intervals.}
 Given $a,b\in \R$ with $a<b$, we define $\Iintv{a,b}:=[a,b]\cap\Z$. Furthermore, given $L\in \N$, we set $\Iintv{a,b}_L := [a+L^{-1},b-L^{-1}]\cap L^{-1}\Z$.

\subsection{Main results}
\label{sec: results}
In this section, we will state the assumptions and the main results. We will then discuss the meanings of the assumptions in Section~\ref{sec: assumption_meaning}. After that, we will discuss why we should expect the main results by using a heuristic argument in Section~\ref{sec: heuristic}. 

\vspace{8pt}
    
We fix $\mathbf{u}\in S^{d-1}:=\{\mathbf{x}\in \R^d:~|\mathbf{x}|=1\}$ from now on throughout this paper. 

\subsubsection{Upper bound results for upper tail moderate deviations}

We first state an upper bound for the left term in \eqref{eq: MDP}. We will consider the following assumptions.

  \begin{assum}[Upper Tail Concentration]\label{assum: initial concentration}
 There exist $\overline{\chi}\in (0,1)$, $\theta_0>0$, and $c_0>0$  such that, if  $K$ is large enough, then for any $a\in [\overline{\chi},1]$ and   for any $\mathbf{v}\in H_{\mathbf{u}}$ with $|K\mathbf{v}-K\mathbf{u}|<K^{\frac{\overline{\chi}+1+c_0}{2}}$,\footnote{The exponents  $\frac{\overline{\chi}+1+c_0}{2}$ and $\frac{a+1}{2}$ arise from the scaling relation (see Sections~\ref{sec: assumption_meaning} and~\ref{sec: related_work}).

}
    \al{
    \mathbb{P} \Big(T_{{\rm Cyl}_{\mathbf{v}}\Big([-2K,2K],K^{\frac{a+1}{2}}\Big)}(\mathbf{0},K\mathbf{v})> K\mu(\mathbf{v})+ K^{a}\Big)\leq  e^{-  K^{\theta_0 (a-\overline{\chi})}}.
    }
    \end{assum}
    
    \begin{assum}[Directional Finite Curvature] \label{assum: finite curvature}
         There is a constant $C > 0$ such that,  for any $\mathbf{u}^{\perp}\in \langle \tilde{\mathbf{u}}_2,\ldots, \tilde{\mathbf{u}}_d\rangle,$
	\[
    \mu(\mathbf{u} + \mathbf{u}^{\perp}) - \mu(\mathbf{u}) \leq C|\mathbf{u}^{\perp}|^2.
	\]
\end{assum}
\begin{thm}\label{thm: key prop for upper bound}
 Assume Assumptions~\ref{assum: initial concentration} and \ref{assum: finite curvature} with $\overline{\chi}\in (0,1)$.
 \begin{enumerate}
     \item For any $\varepsilon>0$, there exists $\delta > 0$ such that if  $N\in\N$ is large enough, then  for any $a\in [\overline{\chi} +\varepsilon,1]$ and for any $\mathbf{v}\in H_{\mathbf{u}}$ with $|N\mathbf{v}-N\mathbf{u}|<N^{\frac{\overline{\chi}+1+\delta}{2}}$,
\al{
\mathbb{P} \Big(T_{{\rm Cyl}_{\mathbf{v}}\Big(\R,N^{\frac{a+1}{2}}\Big) }(\mathbf{0}, N\mathbf{v})> N\mu(\mathbf{v})+ N^{a}\Big)\leq  e^{- N^{\frac{d(1-\varepsilon)}{1-\overline{\chi}} (a-\overline{\chi})}}.
}
In particular, we have
\al{
\mathbb{P} \Big(T(\mathbf{0}, N\mathbf{u})> N\mu(\mathbf{u})+ N^{a}\Big)\leq  e^{- N^{\frac{d(1-\varepsilon)}{1-\overline{\chi}} (a-\overline{\chi})}}.
}
\item For any $\varepsilon>0$, if $\zeta>0$ is small enough, then 
\al{
\limsup_{N\to\infty}  N^{-d} \log \P\Bigl(T(\mathbf{0}, N\mathbf{u}) > N\,\mu(\mathbf{u}) + \zeta N\Bigr) \leq - \zeta^{(1+\varepsilon)\frac{d}{1-\overline{\chi}}} .
}
  \end{enumerate}
\end{thm}

\subsubsection{Refinement of concentration condition for upper tail}
{We can show that Assumption~\ref{assum: initial concentration} can be satisfied for a particular $\overline{\chi}$.} Define the upper fluctuation exponent as
 \[
\overline{\chi}_u :=\overline{\chi}_u(\mathbf{u}) := \lim_{\delta \to 0} \lim_{n\to\infty} \inf\left\{\chi>0:~\begin{array}{c}
\forall \mathbf{x}\in \R^d\text{ with }\,|\mathbf{x}|\geq n,\, |  \mathbf{u} - (\mathbf{x}/|\mathbf{x}|)| <\delta, \text{ one has}\\
\E[(T_{{\rm Cyl}_{ \mathbf{x}}(\mathbb{R}, \lvert\mathbf{ x} \rvert^{(1+\chi)/2})}(\mathbf{0},\mathbf{x})-\mu(\mathbf{x}))_+]\leq \lvert \mathbf{x}\rvert^{\chi}
\end{array}\right\}.
\]
We will prove the following using a similar (but simpler) technique to the proof of Theorem~\ref{thm: key prop for upper bound}.
\begin{thm}\label{thm: refinement of condition}
 Assume  Assumption~\ref{assum: finite curvature}. 
    For any $\varepsilon>0$, Assumption~\ref{assum: initial concentration} holds with $\overline{\chi}= \overline{\chi}_u+\varepsilon$.
    \end{thm}
    This combined with Theorem~\ref{thm: key prop for upper bound} yields the following corollary.
    \begin{cor}\label{cor: chi and chiu}
         Assume  Assumption~\ref{assum: finite curvature}.   Then, the results of Theorem~\ref{thm: key prop for upper bound} hold with $\overline{\chi}=\overline{\chi}_u$.
    \end{cor}
    Moreover, we have the following without any curvature assumption:
   \begin{thm}\label{thm: bound for chiu}
       For any $\mathbf{v} \in S^{d-1}$, 
       $$\overline{\chi}_u(\mathbf{v})\leq \frac{3}{5}.$$
   \end{thm}
Although the bound \( 3/5 \) is likely far from optimal, it provides a meaningful estimate for moderate deviations from Theorem~\ref{thm: key prop for upper bound} and Corollary~\ref{cor: chi and chiu}. This is because there is at least one vector \( \mathbf{u} \) that satisfies Assumption~\ref{assum: finite curvature}; see Section~\ref{sec: assumption_meaning}.
\subsubsection{Lower bound results for upper tail moderate deviations}
Next, we will state the lower bound for the left term in \eqref{eq: MDP}. We will need Assumption~\ref{assum: finite curvature}, together with the following assumptions.
 
\begin{assum}[Fluctuation Exponent and Full Concentration]\label{cond: full concentration}
\begin{enumerate}
    \item 
 The following limit exists and satisfies:
    \[
     \chi := \lim_{|\mathbf{x}|\to\infty} \frac{\log\mathrm{Var}(T(\mathbf{0},\mathbf{x}))}{2\log|\mathbf{x}|}\geq  \limsup_{|\mathbf{x}|\to\infty} \frac{\log(\E T(\mathbf{0},\mathbf{x}) -\mu(\mathbf{x}))}{\log|\mathbf{x}|}\in [0,1).
     \]
          \item  For any $\epsilon>0$, there exists $\theta>0$ such that for any $\mathbf{x}\in \Z^d \setminus\{\mathbf{0}\}$,
          \begin{align*}
              \mathbb P\Big( |T(\mathbf{0},\mathbf{x})-\E T(\mathbf{0},\mathbf{x})| > |\mathbf{x}|^{\chi+\epsilon}  \Big)&\leq e^{- |\mathbf{x}|^\theta}.
          \end{align*}
          \end{enumerate}
\end{assum}
\begin{assum}[Directional Positive Curvature]\label{ass: positive curvature}
      There is a constant $C > 0$ such that  for any $\mathbf{u}^{\perp}\in \mathrm{span}\{ \tilde{\mathbf{u}}_2,\ldots, \tilde{\mathbf{u}}_d\},$
	\[
    \mu(\mathbf{u} + \mathbf{u}^{\perp}) - \mu(\mathbf{u}) \geq |\mathbf{u}^{\perp}|^2/C.
	\]
\end{assum}

\begin{thm}\label{thm: key prop for lower bound}
 Assume Assumptions~\ref{assum: finite curvature}, \ref{cond: full concentration}, and \ref{ass: positive curvature} with ${\chi}\in [0,1)$.
 \begin{enumerate}
     \item For any $\varepsilon>0$, and for any $a\in [\chi +\varepsilon,1]$, if  $N\in\N$ is large  enough, then
\al{
\mathbb{P} \Big(T(\mathbf{0}, N\mathbf{u})> N \mu(\mathbf{u})+ N^{a}\Big)\geq  e^{- N^{\frac{d(1+\varepsilon)}{1-{\chi}} (a-\chi)}}.
}
\item  For any $\varepsilon>0$, if $\zeta>0$ is small enough, then we have
\al{
\liminf_{N\to\infty}  N^{-d} \log \P\Bigl(T(\mathbf{0}, N\mathbf{u}) > N\,\mu(\mathbf{u}) + \zeta N\Bigr) \geq - \zeta^{(1-\varepsilon)\frac{d}{1-\chi}} .
}
  \end{enumerate}
\end{thm}
\subsubsection{Results for lower tail moderate deviations}
{In addition to the upper tail moderate deviations, we also provide estimates for the lower tail, under certain assumptions. We will consider the following assumption and use it to give an upper bound for the lower tail moderate deviations.}
\begin{assum}[Lower Tail Concentration]\label{ass: lower tail concentration}
    Let $\overline{\chi}\in (0,1)$ and $\mathbf{u}\in S^{d-1}$. There exists $\theta_0>0$ such that, if  $K$  is large enough, then for any $a\in [\overline{\chi},1]$ and $\mathbf{v} \in H_{\mathbf{u}}$, 
\aln{\label{Assum:initial concentration for lower tail}
\P\Bigl(T(\mathbf{0}, K\mathbf{v}) < K\,\mu(\mathbf{u}) -  K^{a}\Bigr)\leq e^{-K^{\theta_0(a-\overline{\chi})}}.
}
\end{assum}
\begin{thm}\label{thm: upper bound for  lower tail MD} 
Suppose Assumption~\ref{ass: lower tail concentration} with $\overline{\chi}\in (0,1)$.
 \begin{enumerate}
     \item 
    For any $\varepsilon>0$, if $N$ is large enough, then for any $a\in [\overline{\chi} + \varepsilon,1]$ and $\mathbf{v}\in  H_{\mathbf{{u}}}$,
\[
\P\Bigl(T(\mathbf{0}, N\mathbf{v}) < N\,\mu(\mathbf{u}) - N^{a}\Bigr) \le e^{-N^{\frac{1-\varepsilon}{1-\overline{\chi}}(a-\overline{\chi})}}.
\]
\item For any $\varepsilon>0$, if $\zeta>0$ is small enough, then
\[
\limsup_{N\to\infty}  N^{-1} \log \P\Bigl(T(\mathbf{0}, N\mathbf{u}) < N\,\mu(\mathbf{u}) - \zeta N\Bigr) \le - \zeta^{(1+\varepsilon)\frac{1}{1-\overline{\chi}}} .
\] 
\end{enumerate}

\end{thm}
{Last, we consider the lower bound for the lower tail moderate deviations.}
\begin{assum}[Lower Bound for Lower Tail Deviation]\label{ass: lower bound for lower tail deviations}
 There exists $\underline{\chi}\in [0,1)$ such that for any $\delta>0$ and for $K$ large enough depending on $\delta$,
\[
\P\Bigl(T(\mathbf{0}, K\mathbf{u}) < K\,\mu(\mathbf{u}) - 2 K^{\underline{\chi}}\Bigr)\ge e^{-K^{\delta}}.
\] 
\end{assum}
\begin{thm}\label{thm: lower bound for  lower tail MD} 
Assume Assumption~\ref{ass: lower bound for lower tail deviations} with $\underline{\chi}\in [0,1)$.    Let $1\geq a>\underline{\chi}\geq 0$.
\begin{enumerate}
    \item For any $\varepsilon>0$, if $N$ is large enough, then
\[
\P\Bigl(T(\mathbf{0}, N\mathbf{u}) < N\,\mu(\mathbf{u}) - N^{a}\Bigr) \ge e^{-N^{\frac{1+\varepsilon}{1-\underline{\chi}}(a-\underline{\chi})}}.
\]
\item For any $\varepsilon>0$, if $\zeta>0$ is small enough, 
\[
\liminf_{N\to\infty}  N^{-1} \log \P\Bigl(T(\mathbf{0}, N\mathbf{u}) < N\,\mu(\mathbf{u}) - \zeta N\Bigr) \ge - \zeta^{(1-\varepsilon)\frac{1}{1-\underline{\chi}}} .
\] 
\end{enumerate}

\end{thm}

\subsection{Discussion on the assumptions} \label{sec: assumption_meaning}
In this section, we aim to explain the meaning behind the {assumptions in Section~\ref{sec: results}} and why they are reasonable to assume.

\subsubsection*{Upper Tail Concentration (Assumption~\ref{assum: initial concentration})}
This assumption is an upper tail concentration for passage time restricted in a cylinder. {From Corollary~\ref{cor: chi and chiu} and Theorem~\ref{thm: bound for chiu}, we know that Assumption~\ref{assum: initial concentration} is satisfied for some $\overline{\chi} \leq 3/5$, under Assumption~\ref{assum: finite curvature} (note that Assumption~\ref{assum: finite curvature} is satisfied by at least one direction; see the discussion below). However, it is conceivable that Assumption~\ref{assum: initial concentration} holds for $\overline{\chi} = \chi$, the fluctuation exponent.} It is believed that every geodesic from $\mathbf{0}$ to $K\mathbf{v}$ should stay within ${\rm Cyl}_{\mathbf{v}}\Big(\R,K^{\frac{\chi+1+\varepsilon}{2}}\Big)$ for all large $K$ with high probability. (In fact, this has been proved by Chatterjee \cite{chatterjee2013} and Auffinger--Damron \cite{Auffinger2014} under a strong notion of fluctuation exponent and the so-called wandering exponent; see also the discussion on the wandering exponent in Section~\ref{sec: related_work}) For this reason, if $\overline{\chi} \geq \chi$, the restricted passage time should just be the same as the ordinary passage time, and Assumption~\ref{assum: initial concentration} essentially says that it is unlikely for $T(0, K\mathbf{v})$ to deviate from $K\mu(\mathbf{v})$ at an order higher than $K^{\overline{\chi}}$. Using this interpretation, the concentration bound is essentially
\[
\mathbb{P}(T(\mathbf{0}, K\mathbf{v}) > K \mu(\mathbf{v}) + tK^{\overline{\chi}}) \leq e^{-t^{\theta_0}},
\]
with $t = K^{a - \overline{\chi}}$, which should hold if $\overline{\chi} \geq \chi$. We will see that this assumption serves as the initial condition for the induction step.
\subsubsection*{Fluctuation Exponents and Full Concentration (Assumption~\ref{cond: full concentration})}
This condition has two parts. First, it postulates the existence of a fluctuation exponent $\chi\in[0,1)$ such that $\mathrm{Var}(T(\mathbf{0},\mathbf{x})) \sim |\mathbf{x}|^{2\chi}$ and $\E T(\mathbf{0},\mathbf{x}) - \mu(\mathbf{x}) \lesssim |\mathbf{x}|^{\chi}$. The first one is another usual interpretation of the fluctuation exponent, while the second one corresponds to the so-called \emph{nonrandom fluctuations}. There is no guarantee that the exponent for nonrandom fluctuations, if exists, is bounded above by the fluctuation exponent, though it is believed to be the case in FPP. The reader may refer to Section~\ref{sec: related_work} for a further discussion (see also \cite{auffingerdamronhanson2015}). The second part is a two-sided concentration inequality that ensures the passage time is concentrated around the time constant. It appears to be stronger than Assumption~\ref{assum: initial concentration} (since the one in Assumption~\ref{cond: full concentration} is two-sided). Note that Assumption~\ref{cond: full concentration} is satisfied for $\chi = \chi_a$ defined in \cite{chatterjee2013}. Furthermore, one can show that Assumption~\ref{cond: full concentration} implies $\chi = \overline{\chi}_u$.

\subsubsection*{Lower Tail Concentration (Assumption~\ref{ass: lower tail concentration})}
This assumption describes a concentration inequality for the lower tail of the passage time. It states that the probability of the passage time being significantly smaller than its expected value is exponentially small. The exponent $\overline{\chi}$ in the decay rate suggests that deviations below the mean occur at a rate dictated by the fluctuation exponent. This is a fundamental assumption used to rule out extreme deviations on the lower side, ensuring that typical passage times do not fall substantially below their expected value.

It is known from \cite{kesten1993} that this assumption holds with {$\overline{\chi} = 1/2$} in certain settings, providing a justification for its reasonableness. This assumption plays a crucial role in ensuring that the passage time fluctuations remain within a controlled range and do not lead to unexpected short paths that would violate the standard fluctuation theory.

\subsubsection*{Lower Bound for Lower Tail Deviation (Assumption~\ref{ass: lower bound for lower tail deviations})}
Unlike the previous assumption, which provides an upper bound on the probability of lower tail deviations, this assumption establishes a lower bound. It ensures that the probability of a substantial downward deviation in passage time is at least stretched-exponential in $K$. The existence of such a bound is crucial for understanding the distribution's lower tail behavior and highlights that, while rare, significant downward deviations are still possible with a small but non-negligible probability. Currently, there is no known case verifying this assumption with $\underline{\chi}>0$ rigorously (see \cite{newmanpiza1995} for lower bound for a relevant exponent).

\subsubsection*{Directional Curvature (Assumptions~\ref{assum: finite curvature} and \ref{ass: positive curvature})}
To understand these assumptions, recall the limit shape 
$
\mathbb{B}_d := \{\mathbf{x}\in \mathbb{R}^d : \mu(\mathbf{x}) \leq 1\}. 
$ Since $\mathbb{B}_d$ is the unit ball of the norm $\mu$, in particular, $\mathbb{B}_d$ is convex, compact, and has nonempty interior. Geometrically, Assumption~\ref{assum: finite curvature} ensures that one can put a small paraboloid that in contained entirely in $\mathbb{B}_d$ at $\mathbf{u}$ (so that the boundary of $\mathbb{B}_d$ is regular at $\mathbf{u}$), while Assumption~\ref{ass: positive curvature} says that one can put a paraboloid at $\mathbf{u}$ that ``inscribes'' $\mathbb{B}_d$ (so that the boundary of $\mathbb{B}_d$ is curved at $\mathbf{u}$). These curvature assumptions were first introduced by Newman \cite{newmancurvature} to study the geometry of geodesics (we remark that the definition in \cite{newmancurvature} is slightly different from ours).

Note that by convexity and compactness, one can show that there is always a direction that satisfies Assumption~\ref{assum: finite curvature}. However, there is no known example such that both Assumptions~\ref{assum: finite curvature} and \ref{ass: positive curvature} are satisfied at the same direction (although it is believed that this should hold for all directions for a large class of distributions, according to simulations \cite{almdeijfen2015}).

\subsection{A heuristic argument for the main results}
\label{sec: heuristic}
To build intuition, suppose that \(T(\mathbf{0}, N\mathbf{x}) - N\mu(\mathbf{x})\) typically fluctuates on the order of \(N^\chi\) and  that the rate function for upper tail large deviations is given by
\begin{equation*}
I(\zeta):= I_{\mathbf{x}}(\zeta) \;:=\; \lim_{N \to \infty} -N^{-d} \,\log\, \mathbb{P}\Bigl(T(\mathbf{0}, N \mathbf{x}) \;>\; \bigl(\mu(\mathbf{x}) + \zeta\bigr)\,N\Bigr),
\quad \zeta>0.
\end{equation*}
If the weight distribution of  $\tau_e$ has a continuous density with support  $[0,b]$ for some $b>0$, it is known that \(I(\zeta)\) exists \cite{BasuGangulySly}. Fix $t > 0$ and consider the event
\[
\{T(\mathbf{0}, N\mathbf{x}) > N\,\mu(\mathbf{x}) + t\,N^\chi\}.
\]
We will see that the behavior of the rate function \(I(\zeta)\) for small $\zeta$ governs probabilities of deviations on the \(N^\chi\)-scale. 
 First we rewrite the event as
\[
\bigl\{T(\mathbf{0}, N\mathbf{x}) > N\mu(\mathbf{x}) + N \cdot (t N^{-(1-\chi)})\bigr\}.
\]
If \(I(\zeta)\) behaves like \(\zeta^\alpha\) for some \(\alpha > 0\) as \(\zeta \downarrow 0\), then
\aln{\label{eq: heuristic derivation MDP}
&\mathbb{P}\bigl(T(\mathbf{0}, N\mathbf{x}) > N\mu(\mathbf{x}) + tN\cdot N^{-(1-\chi)}\bigr) \notag\\
&\approx \exp\Bigl(-N^d I(tN^{-(1-\chi)})\bigr)\Bigr) \notag\\
&= \exp\Bigl(-\Omega\bigl(N^d (\,tN^{-(1-\chi)})^\alpha\bigr)\Bigr)= \exp\Bigl(-\Omega\bigl(t^\alpha\,N^{d - \alpha(1-\chi)}\bigr)\Bigr).
}
For this to remain bounded away from both \(0\) and \(1\), one must have $$\text{ \(d = \alpha(1-\chi)\), \quad i.e.,\quad \(\alpha = \frac{d}{1-\chi}\).}$$
From this scaling argument with $t=N^{a-\chi}$, we also deduce
\[
\mathbb{P}\bigl(T(\mathbf{0}, N\mathbf{x}) > N\,\mu(\mathbf{x}) + N^a\bigr)
=
\exp\Bigl(-N^{\frac{d(a-\chi)}{\,1-\chi\,}+o(1)}\Bigr)
\quad
\text{for } \chi < a < 1.
\]

Similarly, we can consider a scaling argument for the lower tail moderate deviations as follows. Suppose that for some $\beta > 0$, we have
\[
\mathbb{P}\bigl(T(\mathbf{0}, N\mathbf{x}) < N\mu(\mathbf{x}) - \zeta N\bigr) \approx \exp\bigl(-\Omega(\zeta^\beta) N\bigr),
\]
as $\zeta \downarrow 0$. Then, one expects
\begin{align}
\mathbb{P}\bigl(T(\mathbf{0}, N\mathbf{x}) < N\mu(\mathbf{x}) - t N^{\chi}\bigr)
&= \mathbb{P}\bigl(T(\mathbf{0}, N\mathbf{x}) < N\mu(\mathbf{x}) + tN \cdot N^{-(1-\chi)}\bigr) \notag\\
&\approx \exp\Bigl(-N \,\Omega\bigl( (tN^{-(1-\chi)})^\beta\bigr)\Bigr) \notag\\
\label{eq: lower_tail_heuristic}
&= \exp\Bigl(-\Omega\bigl(t^\beta N^{1 - \beta(1-\chi)}\bigr)\Bigr).
\end{align}
For this probability to remain bounded away from both $0$ and $1$, we require
\[
1 = \beta(1-\chi), \quad \text{i.e.,} \quad \beta = \frac{1}{1-\chi}.
\]
By applying this scaling argument with $t = N^{a-\chi}$, we also deduce that
\[
\mathbb{P}\bigl(T(\mathbf{0}, N\mathbf{x}) < N\mu(\mathbf{x}) - N^a\bigr)
=
\exp\Bigl(-N^{\frac{a-\chi}{1-\chi} + o(1)}\Bigr),
\quad \text{for } \chi < a < 1.
\]
These heuristics suggest the exponents governing moderate deviations are fully determined by  the rate functions $I(\zeta)$ and its lower-tail counterpart. We note that, although the preceding heuristic offers valuable intuition, it lacks sufficient rigor and is not employed in the actual proof (see Section~\ref{sec: idea for upper bound}).

\subsection{A conjecture on the limiting distribution} The above heuristic argument also gives us some insight on the tail behavior of the limiting distribution of $T(\mathbf{0}, N\mathbf{x})$ if it exists. Suppose that \(T(\mathbf{0}, N\mathbf{x}) - N\,\mu(\mathbf{x})\) typically fluctuates on the order of \(N^\chi\), and the normalized passage time
\[
\frac{T(\mathbf{0},N\mathbf{x})-N\mu(\mathbf{x})}{\E\Bigl[\Bigl|T(\mathbf{0},N\mathbf{x})-N\mu(\mathbf{x})\Bigr|\Bigr]}
\]
converges in law to some random variable $X$ as $N\to\infty$. Our heuristic considerations, in particular \eqref{eq: heuristic derivation MDP}, suggest that the upper tail of \(X\) behaves as
\begin{equation}\label{eq:tail_conj}
\mathbb{P}(X>t)=\exp\Bigl(-t^{\frac{d}{\,1-\chi}+o(1)}\Bigr),\qquad t\to\infty.
\end{equation}
On the other hand, \eqref{eq: lower_tail_heuristic} leads to
\[
\mathbb{P}(X<-t)=\exp\Bigl(-t^{\frac{1}{1-\chi}+o(1)}\Bigr),\qquad t\to\infty.
\]
Thus, the tail behavior of \(X\) is asymmetric, with the upper tail decaying as \(\exp\bigl(-t^{\frac{d}{1-\chi}+o(1)}\bigr)\) and the lower tail decaying as \(\exp\bigl(-t^{\frac{1}{1-\chi}+o(1)}\bigr)\), up to some \(o(1)\) corrections in the exponents.

When $d = 2$, it is widely believed that fluctuation exponent is given by \(\chi=\tfrac13\). In particular, we obtain that as $t\to\infty$,
\[
\mathbb{P}(X>t)=\exp\Bigl(-t^{3+o(1)}\Bigr)
\quad\text{and}\quad
\mathbb{P}(X<-t)=\exp\Bigl(-t^{3/2+o(1)}\Bigr).
\]
For comparison, recall that the Tracy--Widom \(F_2\) distribution, which describes the fluctuations of the largest eigenvalue in the Gaussian unitary ensembles (GUE), satisfies (see, e.g., \cite{TracyWidom1994})
\al{
1-F_2(t) \sim \exp\Bigl(-t^{3/2 + o(1)}\Bigr), \quad  F_2(-t) \sim \exp\Bigl(-t^{3+o(1)}\Bigr) \quad \text{as \(t\to \infty\).}
}

Notably, the tail behaviors are reversed between the two settings. In the Tracy--Widom $F_2$ framework, one considers maximization problems, such as the largest eigenvalue of a random matrix or last-passage percolation. In contrast, FPP involves a minimization problem. As a result, the roles of the tails are interchanged.

\begin{rem}
    Note that when the edge weights are unbounded, the speed for the upper tail large deviations can be different from $N^d$; see \cite{CoscoNakajima}. In this case, if we carry out the same heuristic argument, we will see that the tail behavior of the limiting distribution could be different from \eqref{eq:tail_conj}, and this suggests that the limiting distribution could possibly be different from the Tracy--Widom law. {\CO It is in fact expected in the physics literature that when $d=2$, there is a phase transition (from Tracy--Widom and $\chi = 1/3$ to different behaviors) when the tails become heavier than power law with exponent $5$. See \cite{Gueudre_2015} for related research in physics. } Since the situation appears to be more complicated in the unbounded case, we do not intend to conjecture anything.
\end{rem}

\subsection{Related work}
\label{sec: related_work}
In this section, we provide a brief historical overview of research in FPP and discuss related work relevant to our main results.
\subsubsection*{Fluctuation estimates}
As mentioned before, one central problem in FPP is to understand the fluctuations $T(\mathbf{0},N\mathbf{u}) - N\mu(\mathbf{u})$. In order to analyze it, one usually divides it into two parts:
\[
T(\mathbf{0},N\mathbf{u}) - N\mu(\mathbf{u}) = (T(\mathbf{0},N\mathbf{u}) - \E T(\mathbf{0},N\mathbf{u})) + (\E T(\mathbf{0},N\mathbf{u}) - N\mu(\mathbf{u})).
\]
The first term, $T(\mathbf{0},N\mathbf{u}) - \E T(\mathbf{0},N\mathbf{u})$, is the so-called random fluctuations. It can be understood by studying the variance of $T(\mathbf{0},N\mathbf{u})$. The second term, $\E T(\mathbf{0},N\mathbf{u}) - N\mu(\mathbf{u})$, is coined ``nonrandom fluctuations''. It is expected that
\[
T(\mathbf{0},N\mathbf{u}) - \E T(\mathbf{0},N\mathbf{u}) \sim N^\chi, \quad \E T(\mathbf{0},N\mathbf{u}) - N\mu(\mathbf{u}) \sim N^{\gamma}
\]
for some dimensional-dependent exponents $\chi, \gamma \geq 0$. It is also expected that $\chi = \gamma$ for any $d\geq 2$, $\chi = 1/3$ when $d = 2$, and $\chi$ is nonincreasing in $d$. However, none of these is verified rigorously. The relationship between these two exponents has been investigated in \cite{auffingerdamronhanson2015}.

Kesten~\cite{kesten1993} was the first who gave an upper bound for the variance. He showed that $\mathrm{Var}(T(\mathbf{0}, N\mathbf{u})) \le CN$ for some constant \(C\).  Subsequent work by Benjamini--Kalai--Schramm~\cite{benjamini2003}, Benaim--Rossignol~\cite{BenaimRossignol2008}, and Damron--Hanson--Sosoe \cite{Damron2015} refined Kesten's result: under a near optimal moment assumption, the variance grows sublinearly, namely,
\[
\mathrm{Var}\bigl(T(\mathbf{0}, N\mathbf{u})\bigr) \le \frac{CN}{\log N}.
\]
For lower bound, Newman--Piza~\cite{newmanpiza1995} proved that when $d = 2$, for a large class of distributions, one has
$\mathrm{Var}\bigl(T(\mathbf{0}, N\mathbf{u})\bigr) \ge c\log N$, 
which rules out fluctuations that are too small at large scales. Under an additional directional positive curvature assumption, they further improved the bound to $cN^{1/4}$. {Finally, we would like to mention that recently Bates--Chatterjee \cite{bateschatterjee} and Damron--Hanson--Houdr\'e--Xu \cite{DHHX} give a lower bound $\sqrt{\log{N}}$ for the fluctuations of the passage time in two dimensions, which is stronger than the variance lower bound by Newman and Piza.}

As for nonrandom fluctuations, it was first shown by Kesten \cite{kesten1993} that $\mathbb{E} T(\mathbf{0}, N\mathbf{u}) - N\mu(\mathbf{u}) \leq CN^{5/6}\log{N}$. Alexander~\cite{Alexander} improved the upper bound to $CN^{1/2}\log{N}$, and it was further improved to $C(N\log{N})^{1/2}$ by Tessera \cite{Tessera}. Damron--Kubota \cite{damronkubota} showed the same bound under a low moment assumption. On the other hand, the second author of this paper demonstrated that the nonrandom fluctuations diverge in a suitable sense in any dimension \(d\ge 2\): there exists a sequence $\mathbf{x}_n\in \Z^d$ such that $|\mathbf{x}_n|\to \infty$ and $\E|T(\mathbf{0},\mathbf{x}_n)-\mu(\mathbf{x}_n)|\to \infty$~\cite{nakajima2019divergence}.

\subsubsection*{Wandering exponent and the scaling relation}
The wandering exponent $\xi$ is closely related to the fluctuation exponent. Roughly speaking, it governs the maximal distance from any geodesic in $\mathbb{G}(\mathbf{x}, \mathbf{y})$ to the line segment from $\mathbf{x}$ to $\mathbf{y}$.  It is expected the wandering and fluctuation exponents satisfy  the \emph{scaling relation}, 
\al{
\chi = 2\xi - 1,\text{ or equivalently }\xi=\frac{\chi+1}{2}.
} It was originally discovered in the Kardar--Parisi--Zhang (KPZ) equation \cite{kpz1986}. Later, Chatterjee \cite{chatterjee2013} proved the scaling relation rigorously (and the proof was simplified by \cite{Auffinger2014}), assuming the existence of these exponents in a strong sense.

\subsubsection*{Large deviations in FPP}

The study of large deviations in FPP began with Kesten~\cite{aspects}. In the lower tail regime, a subadditive argument shows that for sufficiently small $\zeta>0$, the limit
\aln{\label{eq: lower tail LDP}
J(\zeta) := \lim_{N\to\infty}\frac{1}{N}\log \P\Big(T(\mathbf{0},N\mathbf{e}_1)<N\mu(\mathbf{e}_1)- \zeta N\Big)
}
    exists and is strictly negative (see also \cite{verges2024largedeviationprinciplespeed} for large deviations in terms of geodesics, which is closely related to lower tail large deviations). In contrast, for the upper tail, under the boundedness of the weight distribution, Kesten established that
\[
-\infty <\varliminf_{N\to\infty}\frac{1}{N^d}\log \P\Big(T(\mathbf{0},N\mathbf{e}_1)>N\mu(\mathbf{e}_1)+\zeta N\Big)
\le \varlimsup_{N\to\infty}\frac{1}{N^d}\log \P\Big(T(\mathbf{0},N\mathbf{e}_1)>N\mu(\mathbf{e}_1)+\zeta N\Big)<0.
\]
These estimates reveal that the lower and upper tails are strongly asymmetric (see \cite{ChowZhang} for heuristic insights). When the two upper tail limits coincide, the common value is identified as the rate function.

Subsequent contributions have extended this framework in several directions. Chow--Zhang \cite{ChowZhang} examined box-to-box passage times, while Cranston--Gauthier--Mountford \cite{CranstonMountfordGore} derived necessary and sufficient conditions---specifically, a volumic $N^d$ scaling---for the validity of the upper tail estimates.  Moreover, Basu--Ganguly--Sly \cite{BasuGangulySly} proved the existence of the rate function for the upper tail under additional assumptions on the weight distribution. Cosco--Nakajima \cite{CoscoNakajima} further analyzed the rate function in the setting of light-tailed distributions (including the exponential distribution, which is also known as the Eden growth model), where the rate function is characterized by the so-called $p$-capacity.

\subsubsection*{Rotationally Invariant FPP and tail estimates.}
Recent work by Basu--Sidoravicius--Sly \cite{basu2023} investigates two-dimensional rotationally invariant FPP models, establishing novel tail bounds and scaling relations. In particular, they demonstrate an “improved concentration” phenomenon for the tail of the first-passage time. A variant of this behavior will emerge later in our proof (see Propositions~\ref{prop: key prop for upper bound} and \ref{prop: key prop for upper bound2}). Furthermore, we iterate this improved concentration to obtain the correct exponent in the moderate deviation regime.

\subsubsection*{Connections to last-passage percolation (LPP) and the KPZ universality class}
FPP is closely related to LPP, with the key distinction that LPP focuses on maximizing (directed) paths, whereas FPP considers minimizing paths. In the (1+1)-dimensional LPP with exponential or geometric distributions, integrable structures provide explicit passage-time distributions \cite{Johansson2000}. Under appropriate rescaling, these fluctuations converge to the Tracy--Widom law for GUE \cite{TracyWidom1994}.

These exactly solvable LPP models belong to the Kardar--Parisi--Zhang (KPZ) universality class \cite{kpz1986, Corwin2012, ferrari_integrable_probability, MatetskiQuastelRemenik2021}, characterized by universal exponents common to KPZ-type growth processes \cite{Borodin2015, Quastel2011}. However, FPP appears to lack integrable structures, leaving open the rigorous derivation of asymptotic distributions or exact exponents. Nevertheless, exact results from LPP serve as valuable guides for conjectures and heuristics in FPP, reinforcing the belief that FPP also belongs to the KPZ class, sharing its universal scaling and distributional properties. Ongoing research in FPP continues to draw upon insights from LPP and related models, uncovering deep connections beyond integrable systems.

We remark that our results share certain similarities with the recent study \cite{GangulyHegde2023} by {\CO Ganguly--Hegde} of two-dimensional  LPP for general weights under some assumptions, which establishes optimal tail bounds. However, there are also fundamental differences. First, their optimal tail estimates are strictly stronger than moderate deviations. Second, in FPP, paths are undirected, which pose additional technical challenges compared to the directed paths in LPP. Third, whereas \cite{GangulyHegde2023} restricts attention to $(1,1)$-directed paths in two dimensions, our framework accommodates higher-dimensional settings and a broader range of directional constraints. Fourth, \cite{GangulyHegde2023} consistently assumes both positive and finite curvature, and for lower tail estimates \cite[Theorem 3]{GangulyHegde2023}, they further impose the full concentration condition, a variant of  Assumption~\ref{cond: full concentration}-(2).  
Finally, their approach to the upper tail of LPP depends on the exponent governing lower tail deviations---analogous to $N^a$ in Theorem~\ref{thm: upper bound for  lower tail MD}  in our case---being sufficiently small, a requirement that we do not impose. {\CO To establish the improved concentration for upper-tail, they rely on a property known as the geodesic watermelon \cite{BasuGangulyHammondHegde2022}, which involves constructing many disjoint geodesics in the plane. However, extending such a construction to higher dimensions is not straightforward. Instead, we introduce a new slab argument with modified scaling schemes.}

\subsection{Organization of the paper}

The remainder of this paper is organized as follows. 

In Section~\ref{sec:upper-bound}, we prove Theorem~\ref{thm: key prop for upper bound}, which establishes an upper bound for the upper-tail moderate deviations, highlighting the key challenges in the proof and employing an induction scheme based on Proposition~\ref{prop: key prop for upper bound}. Sections~\ref{sec:upper-bound-1} and~\ref{sec:upper-bound-2} complete the proofs of Theorem~\ref{thm: key prop for upper bound}~(1) and (2), respectively. Section~\ref{sec:refinement} refines Assumption~\ref{assum: initial concentration} and contains the proofs of Theorems~\ref{thm: refinement of condition} (Section~\ref{sec:refinement-1}) and~\ref{thm: bound for chiu} (Section~\ref{sec:refinement-2}). In Section~\ref{sec:upper bound lower tail}, we establish Theorem~\ref{thm: upper bound for lower tail MD} by a similar but simpler argument to that in Section~\ref{sec:upper-bound}.

We prove Theorem~\ref{thm: key prop for lower bound} in Section~\ref{sec: lower_bound_upper_tail}, while in Section~\ref{sec:final-proof} we prove Theorem~\ref{thm: lower bound for lower tail MD}. These two theorems can be proved by using more straightforward arguments.

\subsubsection*{Bound on Edge Weights}

To simplify the proofs, we impose the following additional condition:
\begin{equation} \label{ass: weight distribution} \tau_e \leq \frac{1}{4d^2} \quad \text{a.s}. \end{equation}
This normalization by  $1/(4d^2)$ is chosen for notational convenience. Since any rescaling of edge weights leads to a corresponding linear rescaling of passage times, this assumption does not change the system’s qualitative behavior. Note that \eqref{ass: weight distribution} implies the following:
\begin{equation}\label{ass: weight distribution2}
     T_A(\mathbf{x},\mathbf{y})\leq {\rm dist}_A(\mathbf{x},\mathbf{y})/(4d)^2 \quad \text{for all $\mathbf{x}, \mathbf{y}\in A$,}
\end{equation}
where ${\rm dist}_A$ is the graph distance restricted to $A.$

    \section{Upper bound for upper tail moderate deviations}\label{sec:upper-bound}
    \subsection{Idea of the proof}\label{sec: idea for upper bound}

Our proof proceeds by an induction scheme that progressively improves the exponent $\theta_0$ appearing in Theorem~\ref{thm: key prop for upper bound}. More precisely, we establish in Proposition~\ref{prop: key prop for upper bound} that, under suitable assumptions, if the probability of a large deviation event initially decays as
\[
\mathbb{P} \Big(T_{\mathrm{Cyl}_{\mathbf{0}, \mathbf{v}}(\R, N^{\frac{a+1}{2}}) }(\mathbf{0}, N\mathbf{v}) > N\mu(\mathbf{v})+ N^{a} \Big) \leq  \exp\left(-N^{\theta (a-\overline{\chi})}\right),
\]
then the exponent $\theta$ can be improved recursively through the function
\[
\overline{f}_{\overline{\chi}}(x) := \frac{x}{1-\overline{\chi}} \wedge \frac{d}{1-\overline{\chi}}.
\]
Since $\overline{f}_{\overline{\chi}}(x) > x$ for all $x \in (0, d)$, this process leads to an eventual strengthening of the deviation estimates.

The core of our approach is an adaptation of \emph{Kesten's slab argument} (see e.g.\ \cite[Section~5]{aspects}), traditionally used to study fluctuations and large deviations in FPP. However, our setting introduces two key difficulties:
\begin{enumerate}
    \item \textbf{Moderate Deviations:} Unlike previous applications, which mainly focus on large deviation regimes, we need to handle moderate deviations at multiple scales simultaneously.
    \item \textbf{Reduction to Face-to-Face Passage Times:} To effectively decompose the passage time across different regions, we need to control passage times between specific geometric subsets, particularly directional slabs and cylinders.
\end{enumerate}

To address the first difficulty, we follow the approach of \cite{CoscoNakajima} together with multi-scale analysis, where it is shown that if the overall passage time is large, then there must exist many localized regions with anomalously high passage times. These regions correspond to \emph{bad vertices}, where we will define in  Definition~\ref{def: bad points}. The second difficulty is handled by restricting our analysis to passage times within appropriately defined tilted directional cylinders.

\subsection{Induction scheme}
Let us define the function $\overline{f}_{\overline{\chi}}(x):=\frac{x}{1-\overline{\chi}}\wedge \frac{d}{1-\overline{\chi}}$. Since $\overline{\chi}\in (0,1),$ we have $\overline{f}_{\overline{\chi}}(x)>x$ for any $x\in (0,d)$. The following proposition plays a key role to execute the induction.
\begin{prop}\label{prop: key prop for upper bound}
 Assume  Assumptions~\ref{assum: initial concentration} and \ref{assum: finite curvature} with $\overline{\chi}\in (0,1)$, $\theta_0>0$, and $c_0>0$.    For any $\varepsilon>0 $, there exists $c_1>0$ such that, if $N\in\N$ is large enough, then for any  $a\in [\overline{\chi}+\varepsilon,1]$ and for any $\mathbf{v}\in H_{\mathbf{u}}$ with $|\mathbf{v}-\mathbf{u}|<N^{\frac{\overline{\chi}+\varepsilon-1+c_1}{2}} $,
 \al{
\mathbb{P} \Big(T_{\mathrm{Cyl}_{\mathbf{0}, \mathbf{v}}([-2N,2N], N^{\frac{a+1}{2}}) }(\mathbf{0}, N\mathbf{v})> N\mu(\mathbf{v})+ N^{a}\Big)\leq  \exp{\left(-N^{(1-\varepsilon)\overline{f}_{\overline{\chi}+\varepsilon}(\theta_0) (a-\overline{\chi}-\varepsilon)}\right)}.
}
\end{prop}
 Assume  Assumptions~\ref{assum: initial concentration} and \ref{assum: finite curvature} with $\overline{\chi}\in (0,1)$,  $\theta_0>0$, and $c_0>0$ until we complete the proof of Proposition~\ref{prop: key prop for upper bound}. Let $\varepsilon>0$. Without loss of generality, we assume  that $\varepsilon$ is small enough so that 
 \aln{\label{cond: vareps condition}
0< \varepsilon<(8d+(1-\overline{\chi})^{-1} +c_0^{-1}+ \theta_0^{-1})^{-2d}.
 }
 We set $M:= L:=\lceil 4\varepsilon^{-4}\rceil$ and $c_1: =\varepsilon^5$. Let $a\in [\overline{\chi}+\varepsilon,1]$ and  $\mathbf{v}\in H_{\mathbf{u}}$ with $|\mathbf{v}-\mathbf{u}|<N^{\frac{\overline{\chi}+\varepsilon-1+c_1}{2}}$. If $a \in [\overline{\chi}+\varepsilon,\overline{\chi}+\varepsilon+\varepsilon^{3/2}]$, then  for $N$ large enough, since $\varepsilon$ is small enough, 
 \al{\mathbb{P} \Big(T_{\mathrm{Cyl}_{\mathbf{0}, \mathbf{v}}([-2N,2N], N^{\frac{a+1}{2}}) }(\mathbf{0}, N\mathbf{v})> N\mu(\mathbf{v})+ N^{a}\Big)&\leq  \exp{(-N^{\theta_0 (a-\overline{\chi})})}\\
 &\leq  \exp{(-N^{\theta_0 \varepsilon})}\\
 &\leq \exp{\Big(-N^{\frac{\theta_0}{1-\overline{\chi}-\varepsilon} \varepsilon^{3/2}}\Big)}\leq  \exp{\Big(-N^{\frac{\theta_0}{1-\overline{\chi}-\varepsilon}(a-\overline{\chi}-\varepsilon)}\Big)},
 }
 which gives us the conclusion of Proposition~\ref{prop: key prop for upper bound}. From now on, we always assume $a\geq \overline{\chi}+\varepsilon+\varepsilon^{3/2}$. 
 Before going to the proof, we prepare some notations and lemmas.  

The following was proved in \cite[Theorem 5.2 and Proposition 5.8]{aspects}, which will also be frequently used in later sections. 
\begin{lem}[Lower tail large deviation estimates]\label{lem: Theorem 5.2 and Proposition 5.8}
    Suppose $\mathbb{P}(\tau_e=0)<p_c(d)$. There exists $c>0$ such that for any $\mathbf{x}\in \R^d$ with $|\mathbf{x}|\geq 1$,
\aln{   \label{eq: Kesten Propositin 5.8}
\mathbb{P}(\exists \gamma:\mathbf{0}\to \mathbf{x},\,T(\gamma)<c|\mathbf{x}|)\leq e^{-c|\mathbf{x}|}.
}
Moreover, for any $\varepsilon>0$, there exists $c'=c'(\varepsilon)>0$ such that for any $\mathbf{x} \in \R^d$  with $|\mathbf{x}|\geq 1$,
\aln{
\mathbb{P}(T(\mathbf{0},\mathbf{x})<\mu(\mathbf{x}) - \varepsilon |\mathbf{x}|)\leq e^{-c'|\mathbf{x}|}.
}
\end{lem}
 \begin{Def}\label{def: bad points}
        Given $b\geq \overline{\chi}+\varepsilon$, $K\geq 1$, and $\mathbf{v} \in \R^d \setminus \{0\}$, we say that $\mathbf{x}\in \R^d$ 
is $(b, K)_{\mathbf{v}}$-bad if there exist $\mathbf{y} \in  \mathrm{Cyl}_{\mathbf{x}, \mathbf{v}}([-K,2K], K^{\frac{b+1}{2}})$ and $\mathbf{y'}\in \mathrm{Cyl}_{\mathbf{y}, \mathbf{v}}(\{K\}, K^{\frac{\overline{\chi}+\varepsilon+1}{2}})$ such that
\al{
    T_{\mathrm{Cyl}_{\mathbf{x}, \mathbf{v}}([-4K, 4K], 4 K^{\frac{b+1}{2}})}(\mathbf{y,y'})\geq  K\mu(\mathbf{v})+ 2K^b.
}
\end{Def}
Roughly speaking, $\mathbf{x}$ is $(b, K)_{\mathbf{v}}$-bad if there is a mesoscopic neighborhood near $\mathbf{x}$ in which the passage time from one side to the other can be large.
\begin{lem}
      Let $\varepsilon > 0$. If $K$ is large enough depending on $\varepsilon$, then for any $\mathbf{v}\in H_{\mathbf{u}}$ with $|\mathbf{v}-\mathbf{u}|<K^{\frac{\overline{\chi}+\varepsilon-1+\varepsilon^5}{2}} $ and $b\in [\overline{\chi}+\varepsilon + L^{-1},1-L^{-1}]$,  we have for any $\mathbf{x}\in\R^d$,
\aln{\label{eq: bad probability}
\mathbb{P}(\textnormal{$\mathbf{x}$ is $(b, K)_{\mathbf{v}}$-bad})\leq K^{4d} e^{- K^{\theta_0 (b-\overline{\chi})}}.
}
\end{lem}

\begin{proof}
By definition, a vertex \(\mathbf{x} \in \mathbb{Z}^d\) is \((b,K)_{\mathbf{v}}\)-bad if there exist  
\(\mathbf{y} \in \mathrm{Cyl}_{\mathbf{x}, \mathbf{v}}([-K,2K], K^{\frac{b+1}{2}})\)  
and  
\(\mathbf{y'} \in \mathrm{Cyl}_{\mathbf{y}, \mathbf{v}}(\{K\}, K^{\frac{\overline{\chi}+\varepsilon+1}{2}})\)  
such that,  
\[
T_{\mathrm{Cyl}_{\mathbf{x}, \mathbf{v}}([-4K,4K],4K^{\frac{b+1}{2}})}(\mathbf{y},\mathbf{y'})\geq K\mu(\mathbf{v})+2K^b.
\]
Write
\begin{align*}
K\mu(\mathbf{v})-\mu({\mathbf{y}}'-{\mathbf{y}}) = \big[K\mu(\mathbf{v})-\mu(K\mathbf{u})\big]+\big[\mu(K\mathbf{u})-\mu({\mathbf{y}}'-{\mathbf{y}})\big].
\end{align*}
Recall that $\mathbf{v} \in H_\mathbf{u}$. Hence, by  Assumption~\ref{assum: finite curvature}, we have 
\[
|K\mu(\mathbf{v})-\mu(K\mathbf{u})| = K|\mu(\mathbf{v}) - \mu(\mathbf{u})| \leq C_1 K|\mathbf{v} - \mathbf{u}|^2 \leq C_1 K^{\overline{\chi} + \varepsilon + \varepsilon^5}
\]
for some constant $C_1 > 0$. For the second term, observe that $\mathbf{y}' - \mathbf{y} \in \Cyl_{\mathbf{0}, \mathbf{v}}(\{K\}, K^{\frac{\overline{\chi}+\varepsilon+1}{2}})\subset \Cyl_{\mathbf{0}, \mathbf{u}}(\{K\}, K^{\frac{\overline{\chi}+\varepsilon+\varepsilon^5+1}{2}}+ K^{\frac{\overline{\chi}+\varepsilon+1}{2}})$, so we can again use Assumption~\ref{assum: finite curvature}, together with triangle inequality, to bound
\[
|\mu(K\mathbf{u})-\mu({\mathbf{y}}'-{\mathbf{y}})| \leq C_2K^{\overline{\chi} + \varepsilon + \varepsilon^5}
\]
for some constant $C_2>0$. Thus, we deduce that
\[
\big|K\mu(\mathbf{v})-\mu({\mathbf{y}}'-{\mathbf{y}})\big|\le C' K^{{\overline{\chi}+\varepsilon+\varepsilon^5}}\le K^b -1 ,
\]
for a suitable constant \(C'>0\) and for \(K\) large enough, since $b\geq \overline{\chi}+\varepsilon+L^{-1}$ and $L\leq 4\varepsilon^{-4}+1$. 

Now, applying Assumption~\ref{assum: initial concentration} with \(a=b\), due to \eqref{ass: weight distribution} and $$\mathrm{Cyl}_{{\mathbf{y}}, \mathbf{v}}([-2K,2K],K^{\frac{b+1}{2}})\subset \mathrm{Cyl}_{\mathbf{x}, \mathbf{v}}([-4K,4K],4K^{\frac{b+1}{2}}),$$ we obtain\footnote{In the string of inequalities, since $\mathbf{y}$ and $\mathbf{y}'$ may not be in $\mathbb{Z}^d$, after shifting, the passage times might have different distributions, but they are close to each other because of the boundedness assumption \eqref{ass: weight distribution}. (This is the reason we include the $+1$, because we need to compensate this minor correction.)}
\al{
&\mathbb{P}\Big(T_{\mathrm{Cyl}_{\mathbf{x}, \mathbf{v}}([-4K,4K],4K^{\frac{b+1}{2}})}(\mathbf{y},\mathbf{y'})\geq K\mu(\mathbf{v})+2K^b\Big) \\
&\leq \mathbb{P}\Big(T_{\mathrm{Cyl}_{{\mathbf{y}}, \mathbf{v}}([-2K,2K],K^{\frac{b+1}{2}})}({\mathbf{y}},{\mathbf{y}}')\geq \mu({\mathbf{y}}'-{\mathbf{y}})+K^b+1\Big)\\
&\leq \mathbb{P}\Big(T_{\mathrm{Cyl}_{\mathbf{0}, \mathbf{v}}([-2K,2K],K^{\frac{b+1}{2}})}(\mathbf{0},{\mathbf{y}}'-{\mathbf{y}})\geq \mu({\mathbf{y}}'-{\mathbf{y}})+K^b\Big)\leq  e^{-K^{\theta_0(b-\overline{\chi})}}.
}
A union bound over all possible pairs of \(({\lfloor \mathbf{y}}\rfloor,{\lfloor \mathbf{y}}'\rfloor)\), whose total number is at most of order \(K^{2d}\), gives
\[
\mathbb{P}(\text{\(\mathbf{x}\) is \((b, K)_{\mathbf{v}}\)-bad})\leq K^{4d} e^{-K^{\theta_0(b-\overline{\chi})}}.
\]
This completes the proof.
\end{proof}

\begin{figure}[t]
    \centering
    \includegraphics[width=0.7\linewidth]{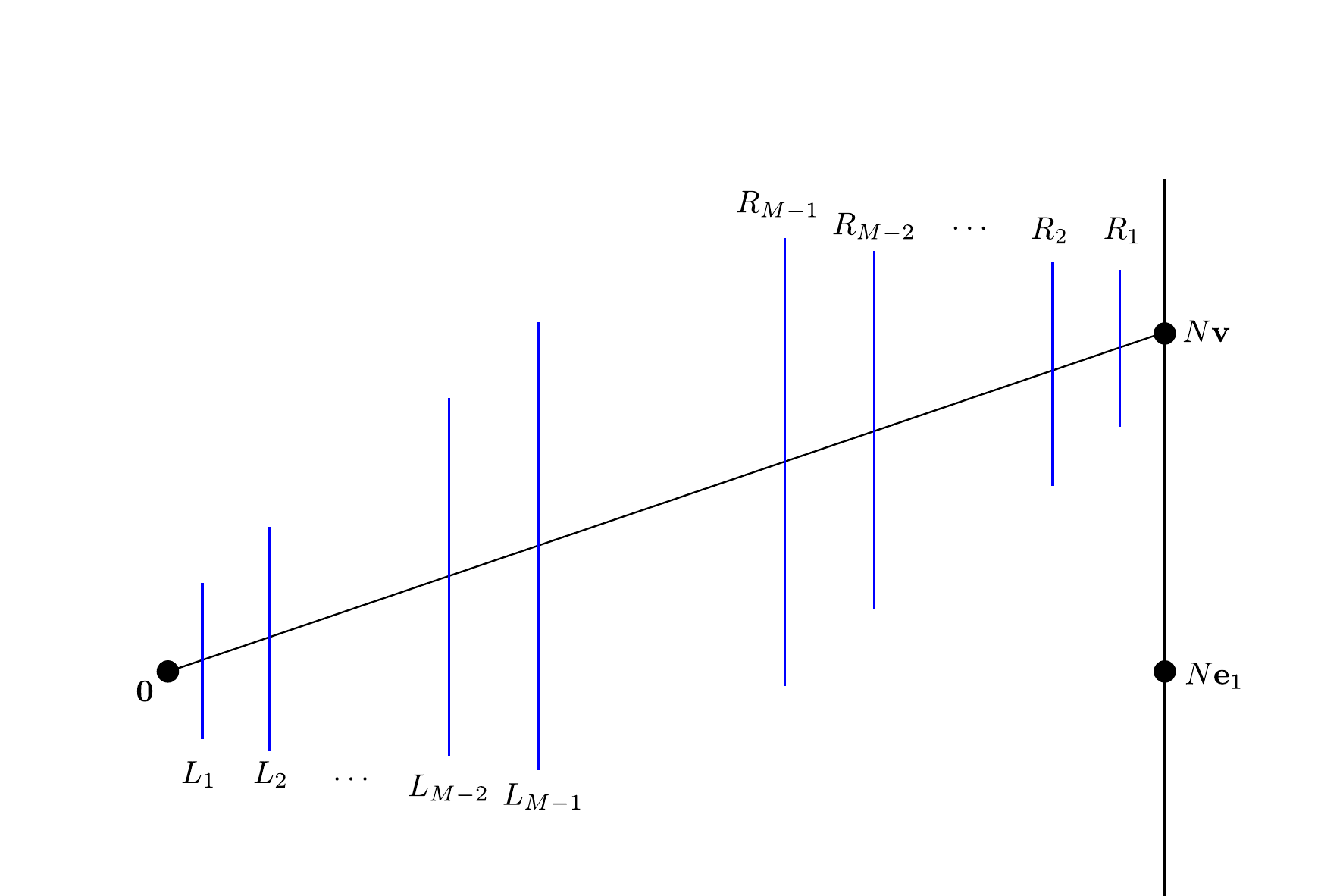}    
    \caption{Depiction of the sets $L_m$ and $R_m$. $L_m$ goes from left to right as $m$ increases, while $R_m$ goes from right to left, but both $L_m$ and $R_m$ are ``moving along'' the direction $\mathbf{v}$ as $m$ changes, and their lengths increase as $m$ increases. The distance between $L_{M-1}$ and $R_{M-1}$ is $\Delta_{M-1} = N - 2N^{\frac{M-1}{M}}$.}
    \label{fig:LmRm}
\end{figure}
For brevity of notation, we introduce
\al{
\mathrm{Cyl}_{ N,a;\mathbf{v}}:= \mathrm{Cyl}_{\mathbf{0}, \mathbf{v}}\left([-2N,2N], N^{\frac{a+1}{2}}\right).
}
Given \(m \in \Iintv{0,M-1}\), we define
\begin{align*}
L_m &:= \{\lfloor \mathbf{x}\rfloor:~ \mathbf{x}\in \mathrm{Cyl}_{\mathbf{0}, \mathbf{v}}(\{N^{ \frac{m}{M}}\}, N^{\frac{a}{2}+\frac{m}{2M}-\varepsilon^2})\}, \\
R_m &:=\{\lfloor \mathbf{x}\rfloor:~ \mathbf{x}\in  \mathrm{Cyl}_{\mathbf{0}, \mathbf{v}}(\{N-N^{ \frac{m}{M}}\}, N^{\frac{a}{2}+\frac{m}{2M}-\varepsilon^2})\}.  
\end{align*}
(See Figure~\ref{fig:LmRm} for a depiction of $L_m$ and $R_m$.) The term $\varepsilon^2$ is technical and its significance will become clear later (see \eqref{eq: next point constraint}). Given $m\in \Iintv{1,M-1}$, set 
\aln{\label{Def: Delta}
\Delta_m:= 
\begin{cases}
   N^{\frac{m+1}{M}}-N^{\frac{m}{M}} & \text{if $m\leq M-2$,}\\
   N- 2N^{\frac{M-1}{M}} & \text{if $m=M-1$.}
\end{cases}
}

\begin{lem}\label{prop: bad slabs}
For all sufficiently large $N$ depending on $\varepsilon$, the following holds: If
\[
T_{\mathrm{Cyl}_{ N,a;\mathbf{v}}}(\mathbf{0}, N\mathbf{v}) > N\mu(\mathbf{v}) +  N^{a},
\]
then there exist $m \in \Iintv{1,M-1}$ with  $m\geq   Ma-2$ and hyperplanes $(\mathcal{L}, \mathcal{R})$ from the following choices:
\begin{enumerate}
    \item $m \leq M-2$, $\mathcal{L} = L_m$, $\mathcal{R} = L_{m+1}$;
    \item $m \leq M-2$, $\mathcal{L} = R_m$, $\mathcal{R} = R_{m+1}$;
    \item $m = M-1$, $\mathcal{L} = L_{M-1}$, $\mathcal{R} = R_{M-1}$,
\end{enumerate}
and there exists $A \subset \mathcal{L}$ with
$|A| \geq (1 - 2^{M(m-2M)}) |\mathcal{L}|$ such that
\[
\left| \left\{ \mathbf{y} \in  \mathcal{R} : T_{\mathrm{Cyl}_{ N,a;\mathbf{v}}}(A,\mathbf{y}) \leq \Delta_{m} \mu(\mathbf{v}) +  (2M)^{-1} N^{a} \right\} \right| \leq  (1 - 2^{M(m+1-2M)}) |\mathcal{R}|.
\]
\end{lem}
\begin{proof}
We prove the contrapositive. Suppose that none of the hyperplanes \((\mathcal{L}, \mathcal{R})\) listed in the lemma satisfies the large-overlap condition, i.e., for every \(m \in\Iintv{1,M-1}\)  with $m\geq  Ma-2$ and for every \(A \subset \mathcal{L}\) with \(|A| \geq (1 - 2^{M(m-2M)}) |\mathcal{L}|\), we have
\begin{align}\label{ass: no good slabs}
    \left| \left\{ \mathbf{y} \in  \mathcal{R} : T_{\mathrm{Cyl}_{ N,a;\mathbf{v}}}(A,\mathbf{y}) \leq \Delta_m \mu(\mathbf{v}) + (2M)^{-1} N^{a} \right\} \right| > (1 - 2^{M(m+1-2M)}) |\mathcal{R}|.
\end{align}

We now show that under this assumption, the total passage time \(T_{\mathrm{Cyl}_{ N,a;\mathbf{v}}}(\mathbf{0}, N\mathbf{v})\) cannot exceed \(N\mu(\mathbf{v}) + N^{a}\).  For \(m \in \Iintv{1,M-1}\), define
\[
L_m' := \left\{\mathbf{x} \in L_m : T_{\mathrm{Cyl}_{ N,a;\mathbf{v}}}(\mathbf{0},\mathbf{x}) \leq N^{ \frac{m}{M}} \mu(\mathbf{v})  + (2M)^{-1} m N^{a} \right\}.
\]
We will show that \(|L_m'| \geq (1 - 2^{M(m-2M)}) |L_m|\) for all \( m \in \Iintv{1, M - 1}\). If $N$ is sufficiently large and if $m < Ma - 1$, then for any $\mathbf{x} \in L_m$, \(T_{\mathrm{Cyl}_{ N,a;\mathbf{v}}}(\mathbf{0}, \mathbf{x}) \leq N^{\frac{m}{M}} < (2M)^{-1} m N^a\), so trivially we have $|L_m'| = |L_m|$.  For induction, suppose that \(|L_m'| \geq (1 - 2^{M(m-2M)}) |L_m|\) for some $m\geq Ma-2$.
 By \eqref{ass: no good slabs}, we have \begin{align}
\label{eq: trivially}
    \left| \left\{ \mathbf{y} \in  L_{m+1} : T_{\mathrm{Cyl}_{ N,a;\mathbf{v}}}(L_m',\mathbf{y}) \leq \Delta_m \mu(\mathbf{v}) + (2M)^{-1} N^{a} \right\} \right| > (1 - 2^{M(m+1-2M)}) |L_{m+1}|.
\end{align}
Next, note that
if $\mathbf{y}\in  L_{m+1}$ satisfies  $T_{\mathrm{Cyl}_{ N,a;\mathbf{v}}}(L_m',\mathbf{y}) \leq \Delta_m \mu(\mathbf{v}) + (2M)^{-1} N^{a}$, then by definition there exists \(\mathbf{x} \in L_m'\) such that
\[
T_{\mathrm{Cyl}_{ N,a;\mathbf{v}}}(\mathbf{x},\mathbf{y}) \leq \Delta_m \mu(\mathbf{v}) + (2M)^{-1} N^{a}.
\]
This implies
\al{
T_{\mathrm{Cyl}_{ N,a;\mathbf{v}}}(\mathbf{0},\mathbf{y}) &\leq T_{\mathrm{Cyl}_{ N,a;\mathbf{v}}}(\mathbf{0},\mathbf{x}) + T_{\mathrm{Cyl}_{ N,a;\mathbf{v}}}(\mathbf{x},\mathbf{y})\\
&\leq N^{ \frac{m}{M}} \mu(\mathbf{v})  + (2M)^{-1} m N^{a} + (N^{ \frac{m+1}{M}} - N^{ \frac{m}{M}})\mu(\mathbf{v}) + (2M)^{-1} N^{a}\\
&\leq  N^{ \frac{m+1}{M}} \mu(\mathbf{v}) + (2M)^{-1} (m+1)N^{a}.
}
That is, $\mathbf{y} \in L_{m+1}'$. Therefore, we obtain \(|L_{m+1}'| \geq (1 - 2^{M(m+1-2M)}) |L_{m+1}|\). By induction, \(|L_m'| \geq (1 - 2^{M(m-2M)}) |L_m|\) for all \( m \in \Iintv{1, M - 1}\).

Similarly, by symmetry, we have \(|R_m'| \geq (1 - 2^{M(m-2M)}) |R_m|\), where
\[
R_m' := \left\{\mathbf{x} \in R_m : T_{\mathrm{Cyl}_{ N,a;\mathbf{v}}}(\mathbf{x},N\mathbf{v}) \leq N^{ \frac{m}{M}} \mu(\mathbf{v})  + (2M)^{-1} m N^{a} \right\}.
\]

Finally, using \eqref{ass: no good slabs} again, we conclude that
\begin{align*}
    \left| \left\{ \mathbf{y} \in  R_{M-1} : T_{\mathrm{Cyl}_{ N,a;\mathbf{v}}}(L_{M-1}',\mathbf{y}) \leq (N - 2N^{\frac{M-1}{M}})\mu(\mathbf{v}) + N^{a} \right\} \right| > (1 - 2^{-M^2}) |R_{M-1}|.
\end{align*}
In particular, there exist \(\mathbf{x} \in L_{M-1}'\) and \(\mathbf{y} \in R_{M-1}'\) such that
\[
T_{\mathrm{Cyl}_{ N,a;\mathbf{v}}}(\mathbf{x},\mathbf{y}) \leq (N - 2N^{\frac{M-1}{M}})\mu(\mathbf{v}) + (2M)^{-1}N^{a}.
\]
Therefore,
\begin{align*}
   T_{\mathrm{Cyl}_{ N,a;\mathbf{v}}}(\mathbf{0},N\mathbf{v}) &\leq T_{\mathrm{Cyl}_{ N,a;\mathbf{v}}}(\mathbf{0},\mathbf{x}) + T_{\mathrm{Cyl}_{ N,a;\mathbf{v}}}(\mathbf{x},\mathbf{y}) + T_{\mathrm{Cyl}_{ N,a;\mathbf{v}}}(\mathbf{y,}N\mathbf{v}) \\
    &\leq N^{ \frac{M-1}{M}} \mu(\mathbf{v}) +(2M)^{-1} (M-1)N^{a} + (N - 2N^{\frac{M-1}{M}})\mu(\mathbf{v}) + (2M)^{-1} N^{a} \\
    &\qquad + N^{ \frac{M-1}{M}} \mu(\mathbf{v}) + (2M)^{-1} (M-1)N^{a} \\
    &\leq N\mu(\mathbf{v}) + N^{a},
\end{align*}
which contradicts the assumption. Hence, we proved the contrapositive.
\end{proof}
 \begin{figure}[t]
    \centering
    \includegraphics[width=0.7\linewidth]{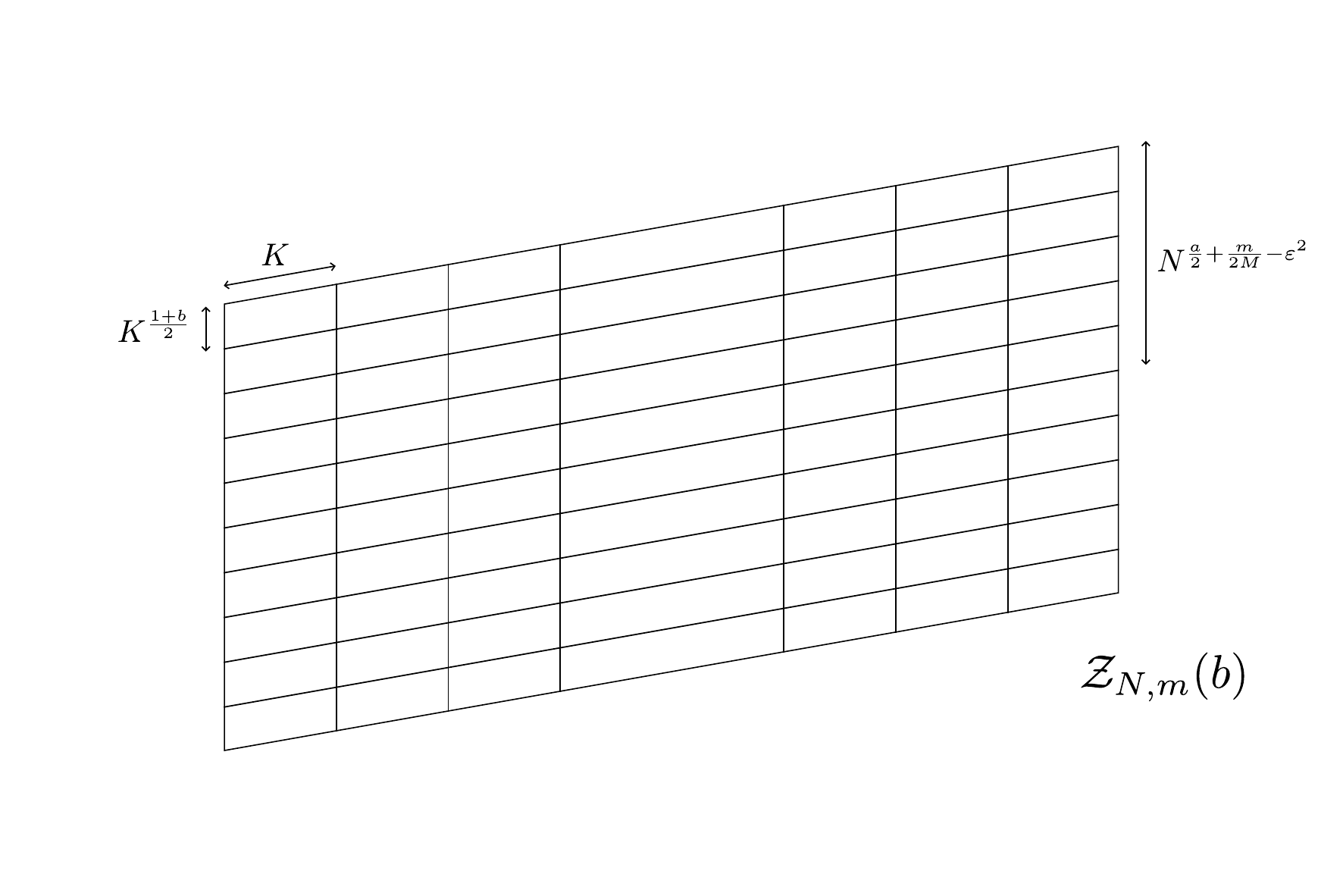}    
    \caption{Depiction of the set $\mathcal{Z}_{N,m}(b)$ for $m\leq M-2$. Each intersection is a point in $\mathcal{Z}_{N,m}(b)$.}
    \label{fig:Znmb}
\end{figure}
Given $a>\overline{\chi}$, $L\in\N$ and $m \in \Iintv{0,M-1}$ with $M\in\N$, we define $J_m := \lfloor \Delta_m N^{-\frac{((m/M)-a)_+}{1-{\overline{\chi}}-\varepsilon} - \varepsilon^3 } \rfloor $ and $K_{m} := \Delta_m/J_m$. Note that  
\aln{\label{eq: condition for K}
N^{\frac{((m/M)-a)_+}{1-{\overline{\chi}}-\varepsilon} + \varepsilon^3 }\leq K_m  \leq 2N^{\frac{((m/M)-a)_+}{1-{\overline{\chi}}-\varepsilon} + \varepsilon^3 }.}
 Given $b\in (0,1)$ and $m\in \Iintv{0,M-2}$, set
\al{
\mathcal{Z}_{N,m}(b)& :=\left\{ y_1 \mathbf{v} + \sum_{i=2}^d y_i \tilde{\mathbf{u}}_i :~
\begin{array}{c}
y_1\in ([N^{ \frac{m}{M}},N^{ \frac{m+1}{M}}]\cup [N-N^{ \frac{m+1}{M}},N-N^{ \frac{m}{M}}])\cap K\Z,\\
y_i \in [-N^{\frac{a}{2}+\frac{m}{2M}-\varepsilon^2},N^{\frac{a}{2}+\frac{m}{2M}-\varepsilon^2}]\cap ( K^{\frac{1+b}{2}} \Z),\,i\geq 2 
\end{array}
\right\},\\
\mathcal{Z}_{N,M-1}(b) &:=\left\{ y_1 \mathbf{v} + \sum_{i=2}^d y_i \tilde{\mathbf{u}}_i  :~
\begin{array}{c}
y_1\in [N^{ \frac{M-1}{M}},N-N^{ \frac{M-1}{M}}]\cap K\Z,\\
y_i \in [-N^{\frac{a}{2}+\frac{M-1}{2M}-\varepsilon^2},N^{\frac{a}{2}+\frac{M-1}{2M}-\varepsilon^2}]\cap ( K^{\frac{1+b}{2}} \Z),\,i\geq 2
\end{array}
\right\}.
}
(See Figure~\ref{fig:Znmb} for a depiction of the set $\mathcal{Z}_{N,m}(b)$.) For simplicity of notation, we introduce
 \al{
\Iintv{s,t}_L:= \left[s + L^{-1},  t -L^{-1}\right] \cap L^{-1} \mathbb{Z}.
 }

The following lemma ensures that if the passage time is large, then we can find bad regions  in the sense that the most face-to-face passage times are large. 
\begin{lem}\label{prop: many bad boxes}
Let $\varepsilon$ satisfy \eqref{cond: vareps condition}, and we set $M:=L :=
\lceil 4\varepsilon^{-4}\rceil\in \N$.  If $N$ large enough, then for any $ a\in [\overline{\chi}+\varepsilon+\varepsilon^{3/2},1]$, $m \in \Iintv{1,M-1}$ with $Ma  \leq  m+ 2$,  and $\mathbf{v}\in H_{\mathbf{u}}$ with $|\mathbf{v}-\mathbf{u}|<N^{\frac{\overline{\chi}+\varepsilon-1+\varepsilon^5}{2}}$, the following holds. Suppose one of the three hold: 
\begin{enumerate}
    \item $m \leq M-2$, $\mathcal L= L_{m}$ and $\mathcal R=L_{m+1}$,
    \item $m \leq M-2$, $\mathcal L= R_{m}$ and $\mathcal R=R_{m+1}$, or
    \item $m =M-1 $, $\mathcal L= L_{M-1}$ and $\mathcal R=R_{M-1}$.
\end{enumerate}
Suppose further that  there exists $\mathcal{A} \subset  \mathcal L$ with  $|\mathcal{A}|\geq (1-2^{M(m-2M)})|\mathcal L|$ such that
\begin{equation}
\label{eq: assumption1}
    \left|\left\{\mathbf{y} \in  \mathcal R:~ T_{\mathrm{Cyl}_{ N,a;\mathbf{v}}}(\mathcal{A},\mathbf{y}) \leq  (N^{ \frac{m+1}{M}}-N^{ \frac{m}{M}})\mu(\mathbf{v}) + (2M)^{-1} N^{a}\right\}\right|\leq (1-2^{M(m+1-2M)})|\mathcal R|.
\end{equation}
Then, there exists $b\in \Iintv{\overline{\chi}+\varepsilon,(1\wedge \frac{M a(1-\overline{\chi}-\varepsilon)}{m(1-(Ma/m))})}_L$ such that
\al{
&|\{ \mathbf{z} \in \mathcal{Z}_{N,m}(b):\,\text{ $\mathbf{z}$ is $(b,K)_{\mathbf{v}}$-bad}\}| \geq  N^{-d\varepsilon^2} \left(\frac{N^{a}}{K_m^b}\right)\left(\frac{N^{\frac{a}{2}+\frac{m}{2M}}}{K_m^{\frac{b+1}{2}}}\right)^{d-1}.
}
\end{lem}
\begin{proof}
 First, we consider $m\in \Iintv{1,M-2}$. Without loss of generality, we assume (1), i.e., $\mathcal L=L_{m}$ and $\mathcal{R}=L_{m+1}$, and \eqref{eq: assumption1} hold.   We write $J=J_m$ and $K= K_{m}.$

Given $\mathbf{x}\in \mathcal{L}$, we consider the following linear interpolation: We let $\mathbf{x}'\in \mathrm{Cyl}_{\mathbf{0}, \mathbf{v}}(\{N^{ \frac{m}{M}}\}, N^{\frac{a}{2}+\frac{m}{2M}-\varepsilon^2}])$  such that $\mathbf{x} = \lfloor \mathbf{x}'\rfloor$ with a deterministic rule breaking ties, and set 
$$\mathbf{y}' := N^{ \frac{m+1}{M}}\mathbf{v}+N^{\frac{1}{2M}}(\mathbf{x}'- N^{ \frac{m}{M}}\mathbf{v})\in \mathrm{Cyl}_{\mathbf{0}, \mathbf{v}}(\{N^{ \frac{m+1}{M}}\}, N^{\frac{a}{2}+\frac{m+1}{2M}-\varepsilon^2}). $$ 
 For $i\in \Iintv{0,J}$, set 
$$\mathbf{x}_K(i):= \mathbf{x}' + \frac{i}{J} (\mathbf{y'-x'}).$$ 
Note that $\mathbf{x}_K(0) =\mathbf{x}'$ and $\mathbf{x}_K(J) = \mathbf{y}'$. 
 We further define
\al{
\mathcal{L}^{\rm bad} &:=\left\{ \mathbf{x} \in \mathcal{L}:~
T_{\mathrm{Cyl}_{ N,a;\mathbf{v}}}(\mathbf{x}_K(0),\mathbf{x}_K(J)) > \mu(\mathbf{v})(N^{ \frac{m+1}{M}}-N^{ \frac{m}{M}}) + (4M)^{-1} N^{a}
\right\}.
}
If $\mathbf{x}\in \mathcal{L}\setminus \mathcal{L}^{\rm bad}$, since $N^{\frac{1}{M}}\ll N^{a}$ by $M\geq 4\varepsilon^{-4}$, then for any $\mathbf{z} \in  \mathcal{R}$ with $|\mathbf{z} -\mathbf{x}_K(J)|_\infty\leq 2d N^{\frac{1}{2M}}$, by \eqref{ass: weight distribution}, we have
\aln{
\label{eq: contradicts}
T_{\mathrm{Cyl}_{ N,a;\mathbf{v}}}(\mathbf{x},\mathbf{z})&\leq T_{\mathrm{Cyl}_{ N,a;\mathbf{v}}}(\mathbf{x}_K(0),\mathbf{x}_K(J))+T_{\mathrm{Cyl}_{ N,a;\mathbf{v}}}(\mathbf{x}_K(J),\mathbf{z})\notag\\
&\leq \mu(\mathbf{v})(N^{ \frac{m+1}{M}}-N^{ \frac{m}{M}}) + (4M)^{-1} N^{a}+N^{\frac{1}{M}}\\
&\leq \mu(\mathbf{v})(N^{ \frac{m+1}{M}}-N^{ \frac{m}{M}}) + (2M)^{-1} N^{a}.\notag
}

We claim that we must have
\aln{\label{eq: bad number inequality}
&|\mathcal{L}^{\rm bad}|\geq 2^{-(2M)^2} |\mathcal L|.
}
To this end, let $\mathcal{A}\subset \mathcal{L}$ with $|\mathcal{A}|\geq (1-2^{M(m-2M)})|\mathcal L|$ satisfy \eqref{eq: assumption1}. We define
$$\mathcal{B}:= \{\mathbf{z} \in  \mathcal{R}:~\exists\mathbf{x} \in  (\mathcal{L} \setminus \mathcal{L}^{\rm bad})\cap \mathcal{A},\,|\mathbf{z}-\mathbf{x}_{K}(J)|_\infty \leq 2d N^{\frac{1}{2M}}\}.$$
By \eqref{eq: contradicts},  we have
$$\mathcal{B}\subset \mathcal{B}' : = \left\{\mathbf{z} \in  \mathcal R:~ T(\mathcal{A},\mathbf{z}) \leq  (N^{ \frac{m+1}{M}}-N^{ \frac{m}{M}})\mu(\mathbf{v}) + (2M)^{-1} N^{a}\right\}.$$
Hence, we have $|\mathcal{B}|\leq |\mathcal{B}'| \leq (1-2^{M(m+1-2M)})|\mathcal{R}|$ by \eqref{eq: assumption1}, which implies $|\mathcal{R}\setminus \mathcal{B}|\geq 2^{M(m+1-2M)}|\mathcal{R}|$.
 Note that for any $\mathbf{z} \in \mathcal{R}$ there exists $\mathbf{x} \in  \mathcal{L}$ such that $|\mathbf{z}-\mathbf{x}_{K}(J)|_\infty \leq 2dN^{\frac{1}{2M}}$. Moreover, if, in addition, $\mathbf{z} \in  \mathcal{R}\setminus \mathcal{B}$, then $\mathbf{x} \in  \mathcal{L}^{\rm bad} \cup (\mathcal L \cap \mathcal{A}^c)$. Hence, together with 
 $|\mathcal{R} \cap (\mathbf{y}+ [-2dN^{\frac{1}{2M}},2dN^{\frac{1}{2M}}]^d)|\leq (2d)^{2d} N^{\frac{d-1}{2M}}$ uniformly for $\mathbf{y}\in \R^d$, since $M\geq (2d)^{2d}$, we obtain
 $$|\mathcal{R}\setminus \mathcal{B}|\leq    |\mathcal{L}^{\rm bad} \cup (\mathcal L \cap \mathcal{A}^c)|\times M N^{\frac{d-1}{2M}}\leq MN^{\frac{d-1}{2M}} (|\mathcal{L}^{\rm bad}|+|\mathcal{L}\cap  \mathcal{A}^c|) .$$
Since  $|\mathcal{L}\cap  \mathcal{A}^c|\leq 2^{M(m-2M)} |\mathcal{L}|$ by $|\mathcal{A}|\geq (1-2^{M(m-2M)})|\mathcal L|$,  we conclude
\al{
|\mathcal{L}^{\rm bad}|& \geq  M^{-1} N^{-\frac{d-1}{2M}} |\mathcal{R}\setminus \mathcal{B}| - 2^{M(m-2M)} |\mathcal{L}| \\
&\geq  M^{-1} N^{-\frac{d-1}{2M}} 2^{M(m+1-2M)}|\mathcal{R}|  - 2^{M(m-2M)} |\mathcal{L}|,
}
which is, in particular,  bigger than $2^{-(2M)^2} |\mathcal L|$ due to the simple bound $|\mathcal{R}| \geq (2d)^{-2d} N^{\frac{(d-1)}{2M}} |\mathcal{L}|
$ for $N$ large enough, and $M$ is large enough so that $2^M \geq M (2d)^{2d+1}$.

 Consider
$$b_*(\mathbf{x}_K(i)) := \sup\left\{b\in  [\overline{\chi}+\varepsilon ,\infty) \cap L^{-1} \mathbb{Z}:~T_{\mathrm{Cyl}_{ N,a;\mathbf{v}}}(\mathbf{x}_K(i),\mathbf{x}_K(i+1))\geq \mu(\mathbf{v}) K +2 K^b\right\}, $$
where we let $b_*(\mathbf{x}_K(i)) :=\overline{\chi}+\varepsilon$ if the set above is empty.
Note that we always have $b_*(\mathbf{x}_K(i))\leq 1$ since $T(\mathbf{x}_K(i),\mathbf{x}_K(i+1)) \leq K$ due to \eqref{ass: weight distribution}. Hence, since $L\geq 4\varepsilon^{-4}$, for any $\mathbf{x}\in \mathcal{L}$,  we have
\al{
T_{\mathrm{Cyl}_{ N,a;\mathbf{v}}}(\mathbf{x}_K(i),\mathbf{x}_K(i+1))
&\leq \mu(\mathbf{v}) K + 2K^{b_*(\mathbf{x}_K(i))+L^{-1}}\\
&\leq  \mu(\mathbf{v}) K + 2K^{\overline{\chi}+\varepsilon +\varepsilon^4}  + \sum_{b \in \Iintv{\overline{\chi}+\varepsilon, b_*(\mathbf{x}_K(i)) }_L}2  K^{b +\varepsilon^4}.
}
Thus, since $J =  \Delta_m /K \leq N^{ \frac{m+1}{M}} K^{-1}$, we have
\al{
&T_{\mathrm{Cyl}_{ N,a;\mathbf{v}}}(\mathbf{x}_K(0),\mathbf{x}_K(J)) \leq  \sum_{i\in \Iintv{0,J-1}} T_{\mathrm{Cyl}_{ N,a;\mathbf{v}}}(\mathbf{x}_K(i),\mathbf{x}_K(i+1))\\
& \leq \mu(\mathbf{v}) (N^{ \frac{m+1}{M}}-N^{ \frac{m}{M}}) +2  N^{ \frac{m+1}{M}} K^{ -(1-\overline{\chi}-\varepsilon-\varepsilon^4)} + \sum_{i\in \Iintv{0,J-1}}\sum_{b \in \Iintv{\overline{\chi}+\varepsilon, b_*(\mathbf{x}_K(i)) }_L}2  K^{b +\varepsilon^4}.
}
Hence, interchanging the sums, this can be further bounded from above by 
\begin{align}
&\mu(\mathbf{v}) (N^{ \frac{m+1}{M}}-N^{ \frac{m}{M}}) + 2  N^{ \frac{m+1}{M}} K^{ -(1-\overline{\chi}-\varepsilon-\varepsilon^4)} \nonumber\\
\label{eq: after_interchange}
&+ \sum_{b \in [\overline{\chi} +\varepsilon+ L^{-1}, \infty) \cap L^{-1}\Z} 2 K^{b +\varepsilon^4} |\{i \in \Iintv{0,J-1}:~b\leq b_*(\mathbf{x}_K(i))-L^{-1} \}|.
\end{align}
We define
\aln{\label{def: bstar}
b_*^{\rm max} :=  \min\left\{1,\frac{M a (1-\overline{\chi}-\varepsilon)}{m(1-Ma/m)_+}\right\},
}
where we also set $b_*^{\rm max}:=1$ if $Ma\geq m$. Note that $K^{\frac{M a (1-\overline{\chi}-\varepsilon)}{m(1-Ma/m)}} \gg  N^{a}$ by \eqref{eq: condition for K}, and  $\frac{M a (1-\overline{\chi}-\varepsilon)}{m(1-Ma/m)} \geq  a \geq \overline{\chi}+\varepsilon+\varepsilon^{3/2}$. Therefore, $b_*^{\max} \geq  \overline{\chi}+\varepsilon + L^{-1}$ for $L= \lceil 4\varepsilon^{-4}\rceil$.
We will prove that, if $\mathbf{x}\in \mathcal L^{\rm bad}$, i.e.,
 $$ \mu(\mathbf{v}) (N^{ \frac{m+1}{M}}-N^{ \frac{m}{M}}) + (4M)^{-1} N^{a}\leq T_{\mathrm{Cyl}_{ N,a;\mathbf{v}}}(\mathbf{x}_K(0),\mathbf{x}_K(J)),$$ then we have
\begin{equation}
\label{eq: bad_ineq}
\begin{split}
&\mu(\mathbf{v}) (N^{ \frac{m+1}{M}}-N^{ \frac{m}{M}})+(4M)^{-1} N^{a} \\
&\leq  \mu(\mathbf{v}) (N^{ \frac{m+1}{M}}-N^{ \frac{m}{M}}) +  N^{ \frac{m+1}{M}} K^{ -(1-\overline{\chi}-\varepsilon-\varepsilon^4)}  \\
&+ \sum_{b\in  \Iintv{\overline{\chi}+\varepsilon, b_*^{\rm max} }_L}  K^{b +\varepsilon^4}|\{i\in \Iintv{0,J-1}:~b\leq b_*(\mathbf{x}_K(i))-L^{-1} \}|.
\end{split}
\end{equation}
To show \eqref{eq: bad_ineq}, we consider two cases: either (a) all $b_*(\mathbf{x}_K(i)) \leq b_*^{\rm max}$, or (b) there is $i$ such that $b_*(\mathbf{x}_K(i)) > b_*^{\rm max}$. In case (a), it is clear that we can restrict the sum in \eqref{eq: after_interchange} to $b\in \Iintv{\overline{\chi}+\varepsilon, b_*^{\rm max} }_L$, and hence we have \eqref{eq: bad_ineq}. In case (b), there exists $i_0$ such that, $b_*(\mathbf{x}_K(i_0)) > b_*^{\rm max}$. If $1 \leq \frac{M a (1-\overline{\chi}-\varepsilon)}{m(1-Ma/m)}$, then $b_*^{\rm max} = 1$, and we are back to the first case, since we always have $b_*(\mathbf{x}_K(i)) \leq 1= b_*^{\rm max}$ for all $i$. So we can focus on the case that $1 > \frac{M a (1-\overline{\chi}-\varepsilon)}{m(1-Ma/m)}$, i.e., $b_*^{\rm max} = \frac{M a (1-\overline{\chi}-\varepsilon)}{m(1-Ma/m)}$. In this case, we then have
\[
N^{a} \leq K^{b_*^{\rm max}} \leq \sum_{b\in  \Iintv{\overline{\chi}+\varepsilon, b_*^{\rm max} }_L}  K^{b +\varepsilon^4}|\{i\in \Iintv{0,J-1}:~b\leq b_*(\mathbf{x}_K(i))-L^{-1} \}|.
\]
The second inequality holds since the set on the right side is nonempty. Thus, \eqref{eq: bad_ineq} follows easily.

We note that for $N$ large enough,
$$K^{\frac{\overline{\chi}+\varepsilon+1}{2}} \geq  J^{-1} N^{\frac{a}{2}+\frac{m+1}{2M}-\varepsilon^2}.$$
Indeed,  since $1\leq K \leq N^{\frac{(m/M)-a}{1-\overline{\chi}-\varepsilon}+O(\varepsilon^3)}$ and $\varepsilon$ is small enough,  we estimate
\[
K^{\frac{\overline{\chi}+\varepsilon+1}{2}} = K\, K^{\frac{\overline{\chi}+\varepsilon-1}{2}} \geq K\, N^{\frac{a-(m/M)}{2} - O(\varepsilon^3)}\gg K\,  N^{\frac{a}{2}-\frac{m+1}{2M}-\varepsilon^2} \geq J^{-1} N^{\frac{a}{2}+\frac{m+1}{2M}-\varepsilon^2} .
\]
Hence, for $i\in \Iintv{0,J-1}$, we have
\aln{\label{eq: next point constraint}
\mathbf{x}_K(i+1) \in {\rm Cyl}_{\mathbf{x}_K(i),\mathbf{v}}( \{K\} ,J^{-1} N^{\frac{a}{2}+\frac{m+1}{2M}-\varepsilon^2}) \subset {\rm Cyl}_{\mathbf{x}_K(i),\mathbf{v}}( \{K\} ,K^{\frac{\overline{\chi}+\varepsilon+1}{2}}).
}

Note that for any $b\in \Iintv{\overline{\chi}+\varepsilon,b_*^{\rm max} }_L$, if  $ b\leq b_*(\mathbf{x}_K(i))-L^{-1}$, then letting 
$\mathbf{z}\in \mathcal{Z}_{N,m}(b)$ with $\mathbf{x}_K(i)\in {\rm Cyl}_{\mathbf{z},\mathbf{v}}([-K,2K],K^{\frac{b+1}{2}})$, $\mathbf{z}$ is $(b,K)_{\mathbf{v}}$-bad. {\CO  Indeed, by \eqref{eq: next point constraint}, setting $\mathbf y=\mathbf{x}_K(i)\in {\rm Cyl}_{\mathbf{z},\mathbf{v}}([-K,2K],K^{\frac{b+1}{2}}) $ and $\mathbf y'=\mathbf{x}_K(i+1)\in {\rm Cyl}_{\mathbf{x}_K(i),\mathbf{v}}( \{K\} ,K^{\frac{\overline{\chi}+\varepsilon+1}{2}})$, by $ b\leq b_*(\mathbf{x}_K(i))-L^{-1}$, we have
\al{
    T_{\mathrm{Cyl}_{\mathbf{z}, \mathbf{v}}([-4K, 4K], 4 K^{\frac{b+1}{2}})}(\mathbf{y,y'})\geq T_{\mathrm{Cyl}_{ N,a;\mathbf{v}}}(\mathbf{y,y'})\geq  K\mu(\mathbf{v})+ 2K^b.
}} 

On the other hand, for any $b\in \Iintv{\overline{\chi}+\varepsilon,b_*^{\rm max} }_L,$ and for any $\mathbf{w}\in \R^d$, recalling that $M$ is large enough that $M\geq (2d)^{2d}$, one has
$$|\{(i,\mathbf{x})\in \Iintv{0, J-1}\times \mathcal{L}:~\mathbf{x}_K(i)\in {\rm Cyl}_{\mathbf{w},\mathbf{v}}([-K,2K],K^{\frac{b+1}{2}})\}|\leq M K^{\frac{(b+1)(d-1)}{2}}.$$ 
Thus, we have
\al{
&\sum_{\mathbf{x}\in \mathcal{L}^{\rm bad}} |\{i\in \Iintv{0,J-1}:~b\leq b_*(\mathbf{x}_K(i))-L^{-1} \}| \\
&\leq \sum_{\mathbf{x}\in \mathcal{L}^{\rm bad}}\; \sum_{i=0}^{J-1}\,  \sum_{\mathbf{z} \in\mathcal{Z}_{N,m}(b)} \mathbf{1}\{\text{$\mathbf{z}$ is $(b,K)_{\mathbf{v}}$-bad},\quad \mathbf{x}_K(i)\in {\rm Cyl}_{\mathbf{z},\mathbf{v}}([-K,2K],K^{\frac{b+1}{2}})\} \\
&= \sum_{\mathbf{z} \in\mathcal{Z}_{N,m}(b)}\;\sum_{\mathbf{x}\in \mathcal{L}^{\rm bad}}\; \sum_{i=0}^{J-1} \mathbf{1}\{\text{$\mathbf{z}$ is $(b,K)_{\mathbf{v}}$-bad},\quad \mathbf{x}_K(i)\in {\rm Cyl}_{\mathbf{z},\mathbf{v}}([-K,2K],K^{\frac{b+1}{2}})\} \\
&\leq M K^{\frac{(d-1)(b+1)}{2}}|\{\mathbf{z} \in\mathcal{Z}_{N,m}(b):~\text{$\mathbf{z}$ is $(b,K)_{\mathbf{v}}$-bad}\}|.
}
Summing both sides of \eqref{eq: bad_ineq} over all $\mathbf{x}\in \mathcal{L}^{\rm bad}$, together with the previous inequality, we have 
\begin{align}
&\sum_{\mathbf{x}\in \mathcal{L}^{\rm bad}} \Big(\mu(\mathbf{v}) (N^{ \frac{m+1}{M}}-N^{ \frac{m}{M}}) + (4M)^{-1} N^{a} \Big)\label{eq: middle step impor lem}\\
&\leq \sum_{\mathbf{x}\in \mathcal{L}^{\rm bad}} (\mu(\mathbf{v}) (N^{ \frac{m+1}{M}}-N^{ \frac{m}{M}})  + N^{ \frac{m+1}{M}} K^{ -(1-\overline{\chi}-\varepsilon-\varepsilon^4)}) \notag\\
&\qquad + M \sum_{b\in  \Iintv{\overline{\chi} +\varepsilon, b_*^{\rm max} }_L}  K^{b +\varepsilon^4} \times  K^{\frac{(d-1)(b+1)}{2}}|\{\mathbf{z} \in\mathcal{Z}_{N,m}(b):~\text{$\mathbf{z}$ is $(b,K)_{\mathbf{v}}$-bad}\}|.\notag
\end{align}
Next we prove 
\aln{\label{eq: comparison the fluctuations}
N^{ \frac{m+1}{M}} K^{ -(1-\overline{\chi}-\varepsilon-\varepsilon^4)} <  (8M)^{-1} N^{a}.
}
 Indeed, since $\overline{\chi}+\varepsilon+\varepsilon^{3/2}\leq a\leq (m+2)/M$, $K \geq N^{\frac{((m/M)-a)_+}{1-{\overline{\chi}}-\varepsilon}+\varepsilon^3}$, and $\varepsilon$ is small enough, we estimate
\al{
K^{-(1-\overline{\chi}-\varepsilon-\varepsilon^4)} &\leq  N^{-(\frac{((m/M)-a)_+}{1-{\overline{\chi}}-\varepsilon}+\varepsilon^3)(1-\overline{\chi}-\varepsilon-\varepsilon^4)}\\
&\leq N^{-(\frac{((m/M)-a)_+}{1-\overline{\chi}-\varepsilon})(1-\overline{\chi}-\varepsilon)-\varepsilon^3(1-\overline{\chi}-\varepsilon)+2\varepsilon^4} \leq N^{-\frac{m}{M}+a-\varepsilon^4},
}
which, by $M\geq 4\varepsilon^{-4}$, implies that for $N$ large enough depending on $\varepsilon$, 
\al{
N^{ \frac{m+1}{M}} K^{ -(1-\overline{\chi}-\varepsilon-\varepsilon^4)}&\leq N^{a+M^{-1} -\varepsilon^4} <  (8M)^{-1} N^{a}\,.
}
Plugging \eqref{eq: comparison the fluctuations} into \eqref{eq: middle step impor lem}, we have
\al{
&M \sum_{b\in  \Iintv{\overline{\chi} +\varepsilon, b_*^{\rm max} }_L}  K^{b +\varepsilon^4}\times  K^{\frac{(d-1)(b+1)}{2}}|\{\mathbf{z} \in\mathcal{Z}_{N,m}(b):~\text{$\mathbf{z}$ is $(b,K)_{\mathbf{v}}$-bad}\}|\\
&\geq (8M)^{-1} \, |\mathcal{L}^{\rm bad}|\, N^{a} \\
&\geq 2^{-(4M)^2}\, |\mathcal{L}| \, N^{a}\geq (2d)^{-2d}\,2^{-(4M)^2}\, N^{(d-1)\Big(\frac{a}{2}+\frac{m}{2M}-\varepsilon^2\Big)} N^{a}\,,
}
where we have used \eqref{eq: bad number inequality} and the simple bound $|\mathcal{L}|\geq (2d)^{-2d} N^{(d-1)\Big(\frac{a}{2}+\frac{m}{2M}-\varepsilon^2\Big)}$. Thus, by the pigeonhole principle, there exists $b\in  \Iintv{\overline{\chi} +\varepsilon, b_*^{\rm max} }_L$ such that
\al{
&|\{\mathbf{z} \in \mathcal{Z}_{N,m}(b):~\text{$\mathbf{z}$ is $(b,K)_{\mathbf{v}}$-bad}\}|\\
&\geq  (2d)^{-2d}\,2^{-(4M)^2}\, M^{-1}\,L^{-1}\,N^{(d-1)\Big(\frac{a}{2}+\frac{m}{2M}-\varepsilon^2\Big)}\, N^{a} \,  K^{-b-\varepsilon^4}\,  K^{-\frac{(d-1)(b+1)}{2}},
}
which implies the desired conclusion for the case $m\leq M-2$, using $K\leq N^2$, as long as $N$ is large enough.

Finally, we consider $m=M-1$. Since the argument is similar to the previous one, we only mention the changes. Given $\mathbf{x}\in \mathcal{L}$, we consider the following linear interpolation: We let $\mathbf{x}'\in \mathrm{Cyl}_{\mathbf{0}, \mathbf{v}}(\{N^{ \frac{M-1}{M}}\}, N^{\frac{a}{2}+\frac{M-1}{2M}-\varepsilon^2}])$  such that $\mathbf{x} = \lfloor \mathbf{x}'\rfloor$ with a deterministic rule breaking ties, and set 
$$\mathbf{y}' :=(N- N^{ \frac{M-1}{M}})\mathbf{v}+(\mathbf{x}'- N^{ \frac{M-1}{M}}\mathbf{v})\in \mathrm{Cyl}_{\mathbf{0}, \mathbf{v}}(\{N-N^{ \frac{M-1}{M}}\}, N^{\frac{a}{2}+\frac{M-1}{2M}-\varepsilon^2}]).$$ 
Observe that the vector   $\mathbf{y}'-\mathbf{x}'$ is parallel  to $\mathbf{v}$. For $i\in \Iintv{0,J_{M-1}}$, set 
$$\mathbf{x}_K(i):= \mathbf{x}' + \frac{i}{J_{M-1}} (\mathbf{y'-x'}).$$
We introduce
\al{
   \mathcal{L}^{\rm bad} &:=\left\{ \mathbf{x} \in \mathcal{L}:~
    T(\mathbf{x}',\mathbf{y}') > \mu(\mathbf{v})(N-2N^{\frac{M-1}{M}}) + (4M)^{-1} N^{a}
\right\}.
}
As before, we can prove that
$
\left|\mathcal{L}^{\rm bad} \right|\geq 2^{-(2M)^2} |\mathcal L|,
$
 otherwise it contradicts the assumption. Consider
$$b_*(\mathbf{x}_K(i)) := \inf\left\{b\in  [\overline{\chi}+\varepsilon,\infty) \cap L^{-1} \Z:~T_{\mathrm{Cyl}_{ N,a;\mathbf{v}}}(\mathbf{x}_K(i),\mathbf{x}_K(i+1))\leq \mu(\mathbf{v}) K + 2K^{b}\right\}.$$ 
Following exactly  the same argument before, we have
\al{
&|\{\mathbf{z} \in \mathcal{Z}_{N,M-1}(b):~\text{$\mathbf{z}$ is $(b,K)_{\mathbf{v}}$-bad}\}| \\
&\geq  2^{-(8M)^2}  N^{(d-1)\Big(\frac{a}{2}+\frac{M-1}{2M}-\varepsilon^2\Big)}\, N^{a} \,  K^{-b-\varepsilon^4}\,  K^{-\frac{(d-1)(b+1)}{2}}.
}
\end{proof}

\begin{proof}[Proof of Proposition~\ref{prop: key prop for upper bound}]
Suppose that the event \{$T_{\mathrm{Cyl}_{\mathbf{0}, \mathbf{v}}(\R, N^{\frac{a+1}{2}}) }(\mathbf{0}, N\mathbf{v})> N\mu(\mathbf{v})+ N^{a}\}$ occurs. Recall $b_*^{\rm max}$ from \eqref{def: bstar} and   $M=L=\lceil 4 \varepsilon^{-4}\rceil$. By  Lemma~\ref{prop: bad slabs} and Lemma~\ref{prop: many bad boxes}, there exist $m\in \Iintv{1,M-1}$ with $Ma\leq m+2$ and  $b\in \Iintv{\overline{\chi} + \varepsilon,b^{\rm max}_*}_L$ such that for $N$ large enough,
\al{
&|\{ \mathbf{z} \in \mathcal{Z}_{N,m}(b):\,\text{ $\mathbf{z}$ is $(b,K)_{\mathbf{v}}$-bad}\}| \geq  \left(\frac{N^{a-d\varepsilon^2}}{K_m^b}\right)\left(\frac{N^{\frac{a}{2}+\frac{m}{2M}}}{K_m^{\frac{b+1}{2}}}\right)^{d-1}.
}
    We write $K:=K_m$, $a_m := (1-L^{-1})\wedge (Ma/m)$, and $\overline{\chi}':=\overline{\chi}+\varepsilon$. (Here, for $a_m$, we include $(1-L^{-1})$ to ensure that  $\frac{a_m(1-\overline{\chi}')}{1-a_m}$ remains finite.) By $Ma\leq m+2$, we note that $|(Ma/m)- a_m|\leq O(\varepsilon^4)$ and $b^{\rm max}_*\leq \frac{a_m(1-\overline{\chi}')}{1-a_m}$. 
For any $b\in\Iintv{\overline{\chi}' ,\frac{a_m(1-\overline{\chi}')}{1-a_m}}_L$,  considering the decomposition:
\al{
\mathcal{Z}_{N,m}(b) = \bigcup_{(y_i)_{i=1}^d \in \Iintv{0,7}^d} \Big(y_1\mathbf{v} +\sum_{i=2}^d y_i \tilde{\mathbf{u}}_i\Big) + 8 \mathcal{Z}_{N,m}(b)  ,
}
we have
    \aln{\label{eq: MDP extact pegionhole}
    & \mathbb{P}\left(|\{ \mathbf{z} \in \mathcal{Z}_{N,m}(b):\,\text{ $\mathbf{z}$ is $(b,K)_{\mathbf{v}}$-bad}\}|\geq  \left(\frac{N^{a-d\varepsilon^2}}{K^b}\right)\left(\frac{N^{\frac{a}{2}+\frac{m}{2M}}}{ K^{\frac{b+1}{2}}}\right)^{d-1}\right)\notag\\
    & \leq \sum_{(y_i)_{i=1}^d\in \Iintv{0,7}^d} \mathbb{P}\left(\begin{array}{c}
    \Big|\Big\{ \mathbf{z} \in \Big(y_1\mathbf{v} +\sum_{i=2}^d y_i \tilde{\mathbf{u}}_i \Big)+ 8 \mathcal{Z}_{N,m}(b)  :\,\text{ $\mathbf{z}$ is $(b,K)_{\mathbf{v}}$-bad}\Big\}\Big|\\
    \geq  8^{-d} \left(\frac{N^{a-d\varepsilon^2}}{ K^{b}}\right)\left(\frac{N^{\frac{a}{2}+\frac{m}{2M}}}{ K^{\frac{b+1}{2}}}\right)^{d-1}
    \end{array}\right)\\\notag
        &\leq 8^d
            \mathbb{P}\left(\Big|\Big\{ \mathbf{z} \in 8\mathcal{Z}_{N,m}(b):\,\text{ $\mathbf{z}$ is $(b,K)_{\mathbf{v}}$-bad}\Big\}\Big|\geq  8^{-d} \left(\frac{N^{a-d\varepsilon^2}}{ K^{b}}\right)\left(\frac{N^{\frac{a}{2}+\frac{m}{2M}}}{ K^{\frac{b+1}{2}}}\right)^{d-1}\right),\notag
       }
    where, for {\CO $c > 0$ and $E\subset \mathbb{R}$, $c E:=\{cx: x\in E\}$,} and we have applied the pigeonhole principle together with a union bound in the second line.
 Since the random variables $$(\mathbf{1}\{\text{$\mathbf{z}$ is $(b,K)_{\mathbf{v}}$-bad}\})_{\mathbf{z}\in 8\mathcal{Z}_{N,m}(b)}$$ are independent, rewriting $|\{ \mathbf{z} \in 8\mathcal{Z}_{N,m}(b):\,\text{ $\mathbf{z}$ is $(b,K)_{\mathbf{v}}$-bad}\}|=\sum_{\mathbf{z}\in  8\mathcal{Z}_{N,m}(b)} \mathbf{1}\{\text{$\mathbf{z}$ is $(b,K)_{\mathbf{v}}$-bad}\}$, by \eqref{eq: bad probability} with $|\mathbf{v}-\mathbf{u}|<N^{\frac{\overline{\chi}+\varepsilon-1+\varepsilon^5}{2}}\leq K^{\frac{\overline{\chi}+\varepsilon-1+\varepsilon^5}{2}}$ (since $1\leq K\leq N$), \eqref{eq: condition for K},  and Chernoff's inequality, this is further bounded from above by
    \al{
   & \exp{\Big(-c_0(N^{\frac{a}{2}+\frac{m}{2M}}/K^{\frac{b+1}{2}})^{d-1}\times (N^{a-d\varepsilon^2}/ K^b) \times K^{\theta_0 (b-\overline{\chi})} \Big)}\notag\\
    &\leq \exp{\Big(-c_1 N^{(d-1)\left(\frac{m(a_m+1)}{2M}-\frac{m(1-a_m)(b+1)}{2M(1-\overline{\chi}')}\right)+\left(\frac{m a_m}{M}-\frac{mb(1-a_m)}{M(1-\overline{\chi}')}\right) +\frac{m(1-a_m)\theta_0(b-\overline{\chi}' )}{M(1-\overline{\chi}')}-2d\varepsilon^2}\Big)},\label{eq: box probab}
    }
    for some $c_0,c_1>0$ independent of $N,a,b$. 
    
    Note that the following expression
    \al{
    &(d-1)\left(\frac{m(a_m+1)}{2M}-\frac{m(1-a_m)(b+1)}{2M(1-\overline{\chi}')}\right)+\left(\frac{m a_m}{M}-\frac{mb(1-a_m)}{M(1-\overline{\chi}')}\right) +\frac{m(1-a_m)\theta_0(b-\overline{\chi}')}{M(1-\overline{\chi}')}
    }
    is linear in $b$, and in particular it attains the minimum at $b= \overline{\chi}'$ or $b=\frac{a_m(1-\overline{\chi}')}{1-a_m}$   over $[\overline{\chi}' , \frac{a_m(1-\overline{\chi}')}{1-a_m}]$.     When $b=\overline{\chi}'$, then
    \al{
    &{(d-1)}\left(\frac{m(a_m+1)}{2M}-\frac{m(1-a_m)(\overline{\chi}'+1)}{2M(1-\overline{\chi}')}\right)+\Big(\frac{m a_m}{M}-\frac{m(1-a_m)}{M(1-\overline{\chi}')}\overline{\chi}\Big)\\
    &= \frac{m(d-1)}{M(1-\overline{\chi}')}\left(a_m-\overline{\chi}'\right)+\frac{m(a_m-\overline{\chi}')}{M(1-\overline{\chi}')}\\
    &=\frac{md}{M(1-\overline{\chi}')}(a_m-\overline{\chi}') \geq  \frac{d}{1-\overline{\chi}'}(a-\overline{\chi}')-\varepsilon^2.
    }
When $b= \frac{a_m(1-\overline{\chi}')}{1-a_m}$, then
    \al{
    &(d-1)\left(\frac{m(a_m+1)}{2M}-\frac{m(1-a_m)(b+1)}{2M(1-\overline{\chi}')}\right)+\left(\frac{m a_m}{M}-\frac{m(1-a_m)}{M(1-\overline{\chi}')}b\right) +\frac{m(1-a_m)}{M(1-\overline{\chi}')}\theta_0(b-\overline{\chi}')\\
    &\geq\frac{m(1-a_m)}{M(1-\overline{\chi}')}\theta_0\Big(\frac{a_m(1-\overline{\chi}')}{1-a_m} -\overline{\chi}'\Big)\\
    &=  \frac{m\theta_0}{M(1-\overline{\chi}')}(a_m-\overline{\chi}') \geq \frac{\theta_0}{1-\overline{\chi}}(a-\overline{\chi}')-\varepsilon^2.
    } 
Therefore, for $N$ large enough depending on $\varepsilon$, using $M\geq  4\varepsilon^{-4}$, $a-\overline{\chi}'\geq \varepsilon^{3/2}$ and \eqref{cond: vareps condition}, we have
\al{
c_1 N^{(d-1)\left(\frac{m(a_m+1)}{2M}-\frac{m(1-a_m)(b+1)}{2M(1-\overline{\chi}')}\right)+\left(\frac{m a_m}{M}-\frac{mb(1-a_m)}{M(1-\overline{\chi}')}\right) +\frac{m(1-a_m)\theta_0(b-\overline{\chi}')}{M(1-\overline{\chi}')}-2d\varepsilon^2} > N^{(1-\varepsilon)\Big(\frac{\theta_0}{1-\overline{\chi}'}\wedge \frac{d}{1-\overline{\chi}'}\Big)(a-\overline{\chi}')}.
}
This proves the proposition.
\end{proof}

\subsection{Proof of Theorem~\ref{thm: key prop for upper bound}-(1)}\label{sec:upper-bound-1}
We take $a\in (\overline{\chi},1)$ and $\varepsilon\in (0,a-\overline{\chi})$. We let $R\in\mathbb{N}$ be such that
    \begin{equation}
        \label{eq: choice of R 1.1}
    \frac{\theta_0(1-(\varepsilon^2/R))^R}{(1-\overline{\chi})^{R}} > \frac{d}{1-\overline{\chi}},
        \end{equation}
    where such an $R$ exists due to $\overline{\chi}>0$. Recalling the function $\overline{f}_{\overline{\chi}}(x):=\frac{x}{1-\overline{\chi}}\wedge \frac{d}{1-\overline{\chi}}$, we define 
    $$\theta_k:=  (1-(\varepsilon^2/R)) \overline{f}_{\overline{\chi}}(\theta_{k-1}),$$ for $k\geq 1$. Then, by induction, for any $k\in \Iintv{1,R}$,
    $$\theta_k \geq (1-(\varepsilon^2/R))^k (\overline{f}_{\overline{\chi}}\circ\cdots\circ \overline{f}_{\overline{\chi}})(\theta_{0})\geq  (1-(\varepsilon^2/R))^k \Big( \frac{d}{1-\overline{\chi}} \wedge \frac{\theta_0}{(1-\overline{\chi})^k}\Big),$$
where the above composition is done $k$ times. Using \eqref{eq: choice of R 1.1} and the inequality $(1-(\varepsilon^2/R))^R\geq 1-\varepsilon^2$, we have
$$\theta_R\geq (1-\varepsilon^2) \frac{d}{1-\overline{\chi}}.$$
Inductively, we define 
$$\varepsilon_k:=\Big(8d+\Big(1-\overline{\chi}_{k-1}\Big)^{-1} + c_{k-1}^{-1}+\theta_{k-1}^{-1}+ (\varepsilon^2/R)^{-1}\Big)^{-2d},\quad c_k := \varepsilon_k^{5},\quad \overline{\chi}_k := \overline{\chi} + \sum_{i=1}^k \varepsilon_i .$$
It is straightforward to check that $\overline{\chi}_k < \overline{\chi} + \varepsilon$ for all $k\in \Iintv{1,R}$ due to $\varepsilon_i<\varepsilon/R$. Also, since $\overline{f}_{a}(x)$ is increasing both in $x$ and in $a$, we have
$$\theta_k \leq (1-\varepsilon_k) \overline{f}_{\overline{\chi}_{k-1}+\varepsilon_k}(\theta_{k-1}).$$
Then, we  apply Proposition~\ref{prop: key prop for upper bound} for $R$ many times, with $\overline{\chi}_k$, $\theta_k$ and $c_k$ in place of $\overline{\chi}$, $\theta_0$, and $c_0$ in the $k$-th step, and obtain a sequence of large integers $N_1 < N_2 < \cdots < N_R$, such that Assumption~\ref{assum: initial concentration} holds with $\overline{\chi}_k$ and $\theta_k$ for any $N\geq N_k$ and $\mathbf{v}\in H_\mathbf{u}$ with $|\mathbf{v-u}| \leq N^{\frac{\overline{\chi}_k-1+c_k}{2}}$. In particular, for any $N\geq N_R$, for any $a \in [\overline{\chi}_R, 1]$, and for any $\mathbf{v}\in H_\mathbf{u}$ with $|\mathbf{v-u}| \leq N^{\frac{\overline{\chi}_R-1+c_R}{2}}$, one has
\[
\mathbb{P} \Big(T_{{\rm Cyl}_{\mathbf{v}}\Big(\R,N^{\frac{a+1}{2}}\Big)}(\mathbf{0}, N\mathbf{v})> N\mu(\mathbf{v})+ N^{a}\Big)\leq  e^{-  N^{\theta_R (a-\overline{\chi}_R)}} \leq e^{-N^{\frac{d(1-\varepsilon)}{1-\overline{\chi}}(a - \overline{\chi})}}.
\]
This proves Theorem~\ref{thm: key prop for upper bound}-(1).

    \subsection{Proof of Theorem~\ref{thm: key prop for upper bound}-(2)}\label{sec:upper-bound-2}
Assume that $\zeta\in (0,1/2)$. The proof is similar to the previous one, although we do not need the induction scheme in the proof of Theorem~\ref{thm: key prop for upper bound}-(1). As before, we prepare some notations and lemmas. Let $c'>0$ be chosen later.    Given \(m \in \Iintv{0,M-1}\), we define
\begin{align*}
\widehat{L}_m &= \mathrm{Cyl}_{\mathbf{0},\mathbf{u}}(\{\zeta^{  \frac{M+1-m}{M}} N\}, \zeta^{\frac{2M-m}{2M}+\varepsilon^2} N), \\
\widehat{R}_m &=\mathrm{Cyl}_{\mathbf{0},\mathbf{u}}(\{N-\zeta^{  \frac{M+1-m}{M}} N\}, \zeta^{\frac{2M-m}{2M}+\varepsilon^2} N).  
\end{align*} 
Given $m\in \Iintv{1,M-1}$,  
$$\widehat{\Delta}_m:= 
\begin{cases}
   \zeta^{\frac{M-m}{M}}N- \zeta^{\frac{M+1-m}{M}}N & \text{if \(m\leq M-2\),}\\
   N- 2\zeta^{\frac{1}{M}}N & \text{if \(m=M-1\).}
\end{cases}$$
Given $a>\overline{\chi}$, $L\in\N$ and $m \in \Iintv{0,M-1}$ with $M\in\N$, we define $\widehat{J}_m = \lfloor \widehat{\Delta}_m \zeta^{\frac{(m/M)-1}{1-{\overline{\chi}}-\varepsilon} - \varepsilon^3 } \rfloor $ and 
$$\widehat{K}_{m} := \widehat{\Delta}_m/\widehat{J}_m\in [\zeta^{-\frac{(m/M)-1}{1-{\overline{\chi}}-\varepsilon} - \varepsilon^3 } ,2\zeta^{-\frac{(m/M)-1}{1-{\overline{\chi}}-\varepsilon} - \varepsilon^3 } ].$$
For simplicity of notation, in the proof, we first consider the deviation $\zeta^{{1-\varepsilon^2}} N$, and later on we will replace it by $\zeta N$.
\begin{lem}\label{prop: bad slabs2}
Fix an integer $M \geq 2$. Then, for all sufficiently large $N$, the following holds: If
$
T(\mathbf{0}, N\mathbf{u}) > N \mu + \zeta^{{1-\varepsilon^2}} N,
$
then there exist $ m \in \Iintv{0,M-1}$ with $m\geq Ma - 1$ and hyperplanes $(\mathcal{L}, \mathcal{R})$ from the following choices:
\begin{enumerate}
    \item $m \leq M-2$, $\mathcal{L} = \widehat{L}_m$, $\mathcal{R} = \widehat{L}_{m+1}$,
    \item $m \leq M-2$, $\mathcal{L} = \widehat{R}_m$, $\mathcal{R} = \widehat{R}_{m+1}$,
    \item $m = M-1$, $\mathcal{L} = \widehat{L}_{M-1}$, $\mathcal{R} = \widehat{R}_{M-1}$,
\end{enumerate}
and there exists $A \subset \mathcal{L}$ with
$
|A| \geq (1 - 2^{M(m-2M)}) |\mathcal{L}|
$ such that
\[
\left| \left\{ \mathbf{y} \in  \mathcal{R} : T(A,\mathbf{y}) \leq \widehat{\Delta}_m \mu(\mathbf{u}) + (2M)^{-1} \zeta^{{1-\varepsilon^2}} N \right\} \right| \leq (1 - 2^{M(m+1-2M)}) |\mathcal{R}|.
\]
\end{lem}
We set
\al{
\widehat{\mathcal{Z}}_{N,m}(b) &:=
\left\{\Big\lfloor y_1 \mathbf{u} + \sum_{i=2}^d y_i \tilde{\mathbf{u}}_i\Big \rfloor :
\begin{array}{c}
y_1\in ([\zeta^{\frac{M+1-m}{M}} N,\zeta^{\frac{M-m}{M}} N]\cup [N-\zeta^{\frac{m+1}{M}} N,N-\zeta^{\frac{m}{M}} N])\cap \widehat{K}_m\Z,\\
y_i \in [-\zeta^{\frac{2M-m}{2M}+\varepsilon^2} N,\zeta^{\frac{2M-m}{2M}+\varepsilon^2} N]\cap ( \widehat{K}_m^{\frac{1+b}{2}} \Z),\,i\geq 2
\end{array}
\right\},\\
\widehat{\mathcal{Z}}_{N,M-1}(b) &:=
\left\{\Big\lfloor y_1 \mathbf{u} + \sum_{i=2}^d y_i \tilde{\mathbf{u}}_i\Big \rfloor :
\begin{array}{c}
y_1\in [\zeta^{\frac{1}{M}} N,N-\zeta^{\frac{1}{M}} N]\cap \widehat{K}_{M-1} \Z,\\
y_i \in [-\zeta^{\frac{1}{2M}+\varepsilon^2} N,\zeta^{\frac{1}{2M}+\varepsilon^2} N]\cap ( \widehat{K}_{M-1}^{\frac{1+b}{2}} \Z),\,i\geq 2
\end{array}
\right\}.
}
\begin{lem}\label{prop: many bad boxes2}
Let $\varepsilon>0$, and set \(M:=L:= \lceil 4 \varepsilon^{-4}\rceil \in \mathbb{N}\).   For $N$ large enough, for any $m \in \Iintv{1,M-1}$ with $m\geq Ma - 1$, the following holds: 
Suppose one of the three holds: 
\begin{enumerate}
    \item $m \leq M-2$, $\mathcal L= \widehat{L}_{m}$ and $\mathcal R=\widehat{L}_{m+1}$,
    \item $m \leq M-2$, $\mathcal L= \widehat{R}_{m}$ and $\mathcal R=\widehat{R}_{m+1}$,
    \item $m =M-1 $, $\mathcal L= \widehat{L}_{M-1}$ and $\mathcal R=\widehat{R}_{M-1}$.
\end{enumerate}
Suppose further that  there exists $A \subset  \mathcal L$ with  $|A|\geq (1-2^{M(m-2M)})|\mathcal L|$ such that
\begin{equation}
\label{eq: assumption}
    \left|\left\{\mathbf{y} \in  \mathcal R:~ T(A,\mathbf{y}) \leq \widehat{\Delta}_m \mu(\mathbf{u}) + (2M)^{-1} \zeta^{{1-\varepsilon^2}} N\right\}\right|\leq (1-2^{M(m+1-2M)})|\mathcal R|.
\end{equation}
Then, there exists $b\in \Iintv{\overline{\chi}+\varepsilon ,1}_L$ such that
\al{
&|\{ \mathbf{z} \in \widehat{\mathcal{Z}}_{N,m}(b):\,\text{ $\mathbf{z}$ is $(b,K)_{\mathbf{u}}$-bad}\}| \geq  \left(\frac{\zeta^{1+d\varepsilon^2} N}{\widehat{K}_m^b}\right)\left(\frac{\zeta^{\frac{2M-m}{2M}} N}{\widehat{K}_m^{\frac{b+1}{2}}}\right)^{d-1}.
}
\end{lem}
The proofs of the lemmas above are exactly the same as before, so we omit them.
\begin{proof}[Proof of Theorem~\ref{thm: key prop for upper bound}-(2)]
Let $\varepsilon>0$ be small enough. Suppose that the event $\{T(\mathbf{0}, N\mathbf{u})> N\mu(\mathbf{u})+\zeta^{{1-\varepsilon^2}} N\}$ occurs. By  Lemma~\ref{prop: bad slabs2} and Lemma~\ref{prop: many bad boxes2}, there exist $m\in \Iintv{1,M-1}$ with $m\geq Ma-2$ and  $b\in \Iintv{\overline{\chi} +\varepsilon,1}_L$ such that for $N$ large enough
\al{
&|\{ \mathbf{z} \in \widehat{\mathcal{Z}}_{N,m}(b):\,\text{ $\mathbf{z}$ is $(b,K)_{\mathbf{u}}$-bad}\}| \geq  \left(\frac{\zeta^{1+d\varepsilon^2} N}{\widehat{K}_m^b}\right)\left(\frac{\zeta^{\frac{2M-m}{2M}}N}{\widehat{K}_m^{\frac{b+1}{2}}}\right)^{d-1}.
} For any $b\in[\overline{\chi}+\varepsilon + L^{-1},1-L^{-1}]$, by a union bound, we have
    \aln{
    & \mathbb{P}\left(|\{ \mathbf{z} \in \widehat{\mathcal{Z}}_{N,m}(b):\,\text{ $\mathbf{z}$ is $(b,K)_{\mathbf{u}}$-bad}\}|\geq  \left(\frac{\zeta^{1+d\varepsilon^2} N}{\widehat{K}_m^b}\right)\left(\frac{\zeta^{\frac{2M-m}{2M}}N}{\widehat{K}_m^{\frac{b+1}{2}}}\right)^{d-1}\right)\notag\\
    &\leq 8^d \mathbb{P}\left(\sum_{\mathbf{z}\in  8\widehat{\mathcal{Z}}_{N,m}(b)} \mathbf{1}\{\text{$\mathbf{z}$ is $(b,K)_{\mathbf{u}}$-bad}\}\geq 8^{-d} \left(\frac{\zeta^{1+d\varepsilon^2} N}{ \widehat{K}_m^b}\right)\left(\frac{\zeta^{\frac{2M-m}{2M}}N}{ \widehat{K}_m^{\frac{b+1}{2}}}\right)^{(d-1)}\right).
\notag
    }
Using independence and the Chernoff inequality,   this is further bounded from above by
    \al{
   & \exp{\Big(-c_0\Big(\frac{\zeta^{1+d\varepsilon^2} N}{ \widehat{K}_m^b}\Big)\times \Big(\frac{\zeta^{\frac{2M-m}{2M}}N}{ \widehat{K}_m^{\frac{b+1}{2}}}\Big)^{(d-1)} \times \widehat{K}_m^{\frac{d(1-\varepsilon)}{1-\overline{\chi}} (b-\overline{\chi} )} \Big)}\notag\\
    &\leq \exp{\Big(-c_1
\zeta^{1 + \frac{(2M - m)(d-1)}{2M} + \frac{m}{M(1 - \overline{\chi})}\left(b + \frac{(b + 1)(d - 1)}{2}\right) - \frac{d m (b - \overline{\chi})}{M(1 - \overline{\chi})^2}+O(\varepsilon) } N^d \Big)},
    }
    for some $c_0,c_1>0$. Note that the following expression
    \al{
    &1  + \frac{(2M - m)(d-1)}{2M} + \frac{m}{M(1 - \overline{\chi})}\left(b + \frac{(b + 1)(d - 1)}{2}\right) - \frac{d m (b - \overline{\chi})}{M(1 - \overline{\chi})^2}
    }
    attains the minimum at $b= \overline{\chi}$ or $b=1$ over $[\overline{\chi} , 1]$.   
    When $b=\overline{\chi}$, then
    \al{
    &1  + \frac{(2M - m)(d-1)}{2M} + \frac{m}{M(1 - \overline{\chi})}\left(b + \frac{(b + 1)(d - 1)}{2}\right) - \frac{d m (b - \overline{\chi})}{M(1 - \overline{\chi})^2}\\
    &= 1  + \frac{(2M - m)(d-1)}{2M} + \frac{m}{M(1 - \overline{\chi})}\left(\overline{\chi} + \frac{(\overline{\chi} + 1)(d - 1)}{2}\right)\\
    &\leq 1  + \frac{d-1}{2}+ \frac{1}{(1 - \overline{\chi})}\left(\overline{\chi} + \frac{(\overline{\chi} + 1)(d - 1)}{2}\right) =  \frac{d}{1-\overline{\chi}}.
    }
    When $b= 1$, then
    \al{
    &1  + \frac{(2M - m)(d-1)}{2M} + \frac{m}{M(1 - \overline{\chi})}\left(b + \frac{(b + 1)(d - 1)}{2}\right) - \frac{d m (b - \overline{\chi})}{M(1 - \overline{\chi})^2}\\
    & =1  + \frac{(2M - m)(d-1)}{2M} \leq \frac{d}{1-\overline{\chi}}.
    }
By a union bound over $b$, replacing $\zeta^{{1-\varepsilon^2}}N$ by $\zeta N$, letting $N\to\infty$ and then adjusting $\varepsilon$ if needed,  we finish the proof.
\end{proof}

\section{Refined analysis of the initial concentration condition}\label{sec:refinement}
In this section, we refine Assumption~\ref{assum: initial concentration} and establish Theorem~\ref{thm: refinement of condition} and Theorem~\ref{thm: bound for chiu}. The proof shares some similarity with the one of Theorem~\ref{thm: key prop for upper bound} but is simpler.

\subsection{Proof of Theorem~\ref{thm: refinement of condition}}\label{sec:refinement-1}

Let $\varepsilon>0$, and let $\mathbf{v}\in H_{\mathbf{u}}$ with $|\mathbf{u}-\mathbf{v}|<N^{\frac{\overline{\chi}_u -1 + \varepsilon+\varepsilon^5}{2}}$. Recall $\tilde{\mathbf{u}}_i$ from \eqref{def: utilde}, and $\Delta_m$ from \eqref{Def: Delta}. Define
\al{
\overline{\mathcal{Z}}_{N} &:=\left\{\mathbf{y} :~
\mathbf{y}= \sum_{i=2}^d y_i\tilde{\mathbf{u}}_i\in \R^d,\,y_i \in [-1,1]\cap N^{-\varepsilon^7} \Z\right\}.
}
Given $\mathbf{x}\in \overline{\mathcal{Z}}_{N}$ and $m\in \Iintv{0,M-1}$, we define $$\overline{\mathbf{x}}_{m}:= N^{\frac{m}{M}}\mathbf{v} + N^{\frac{m(1+\overline{\chi}_u +\varepsilon- \varepsilon^2)}{2M}} \mathbf{x}.$$ For $m\in \Iintv{M,2M-1}$, we define 
$$\overline{\mathbf{x}}_{m} := (N-N^{\frac{2M-1-m}{M}})\mathbf{v} + N^{\frac{(2M-1-m)(1+\overline{\chi}_u +\varepsilon- \varepsilon^2)}{2M}} \mathbf{x}.$$
See Figure~\ref{fig:xm} for a depiction. 

\begin{figure}[t]
    \centering
   \includegraphics[width=0.7\linewidth]{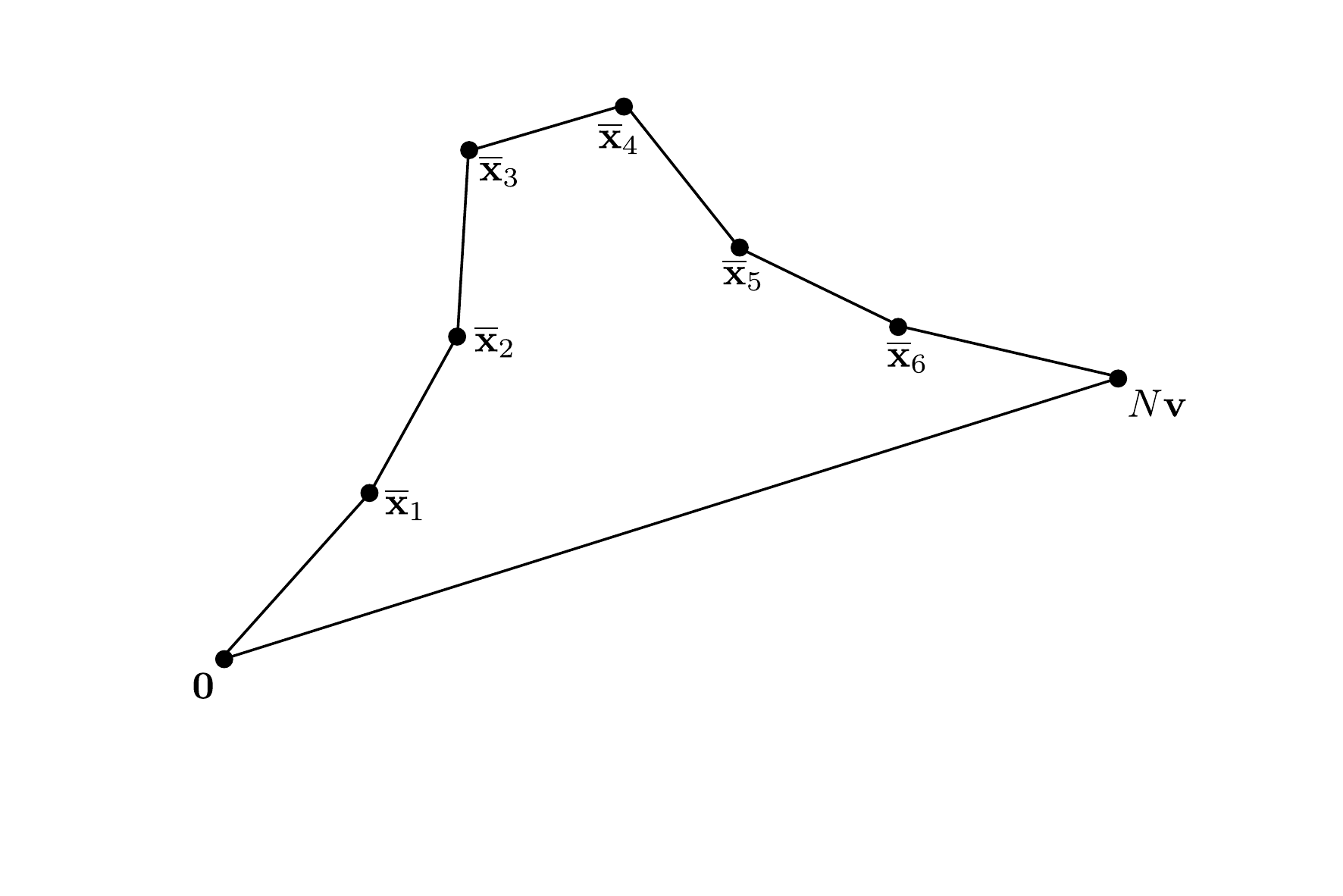} 
\caption{A depiction of the points $(\overline{\mathbf{x}}_m)$ with $M = 4$. $\overline{\mathbf{x}}_1, \ldots, \overline{\mathbf{x}}_{M-1}$ go from $\mathbf{0}$ towards $N\mathbf{v}$ (and ``drifted'' by the direction $\mathbf{x}$), while $\overline{\mathbf{x}}_{2M-1}, \ldots, \overline{\mathbf{x}}_{M}$ go from $N\mathbf{v}$ back to $\mathbf{0}$ (also ``drifted'' by the direction $\mathbf{x}$).}
    \label{fig:xm}
\end{figure}

For $m\in\Iintv{0,M-1}$, we say that $\mathbf{x}\in \overline{\mathcal{Z}}_{N}$ is $m$-black if 
\al{
T_{{\rm Cyl}_{\overline{\mathbf{x}}_{m},\overline{\mathbf{x}}_{m+1}-\overline{\mathbf{x}}_{m}}([-2\Delta_m,2\Delta_m],|\overline{\mathbf{x}}_{m+1}-\overline{\mathbf{x}}_{m}|^{\frac{1+\overline{\chi}_u+(\varepsilon/2)}{2}})}(\overline{\mathbf{x}}_{m},\overline{\mathbf{x}}_{m+1})\geq \Delta_m \mu(\mathbf{v})+ (4M)^{-1} N^{\overline{\chi}_u +\varepsilon}.
} 
For $m\geq M$, we say that $x\in \overline{\mathcal{Z}}_{N}$ is $m$-black if 
\al{
&T_{{\rm Cyl}_{\overline{\mathbf{x}}_{m+1},\overline{\mathbf{x}}_{m+1}-\overline{\mathbf{x}}_{m}}([-2\Delta_m,2\Delta_m], |\overline{\mathbf{x}}_{m+1}-\overline{\mathbf{x}}_{m}|^{\frac{1+\overline{\chi}_u+(\varepsilon/2)}{2}})}(\overline{\mathbf{x}}_{m},\overline{\mathbf{x}}_{m+1})\geq \Delta_{2M-2-m}\, \mu(\mathbf{v})+ (4M)^{-1} N^{\overline{\chi}_u + \varepsilon}.
}  
Finally, set
\al{
\overline{\mathcal{Z}}^{\rm black}_{N,m}& :=\left\{\mathbf{x}\in  \overline{\mathcal{Z}}_{N}:~\mathbf{x}\text{ is $m$-black}
\right\}.
}

\begin{lem}\label{prop: many bad boxes 2}
Let $\varepsilon\in (0,(8d+(1-\overline{\chi}_u)^{-1})^{-2d})$ and $M:=\lceil 4 \varepsilon^{-4}\rceil$.   If  $N$ is large enough, on the event 
\begin{equation}
\label{eq: event_bad_boxes_2}\{T_{\mathrm{Cyl}_{\mathbf{0}, \mathbf{v}}([-2N,2N], N^{\frac{\overline{\chi}_u+1+\varepsilon}{2}}) }(\mathbf{0}, N\mathbf{v}) > N\mu(\mathbf{v})+N^{\overline{\chi}_u + \varepsilon}\},
\end{equation}
there exists $m\in \Iintv{\varepsilon^2 M , 2M-\varepsilon^2 M}$ such that
\al{
&|\overline{\mathcal{Z}}^{\rm black}_{N,m}| \geq  N^{\varepsilon^{8}}.
}
\end{lem}
\begin{proof}
Suppose that \eqref{eq: event_bad_boxes_2} occurs, and fix $\mathbf{x}\in \overline{\mathcal{Z}}_{N}$. First note that if $m\in \Iintv{0, \varepsilon^2M - 1}$ or $m\in \Iintv{2M - \varepsilon^2M + 1, 2M-2}$, by \eqref{ass: weight distribution2}, one has
\[
T_{\mathrm{Cyl}_{\mathbf{0}, \mathbf{v}}([-2N,2N], N^{\frac{\overline{\chi}_u+1+\varepsilon}{2}})}(\overline{\mathbf{x}}_m, \overline{\mathbf{x}}_{m+1}) \leq N^{\varepsilon^2} + N^{\frac{\varepsilon^2(1+\overline{\chi}_u + \varepsilon - \varepsilon^2)}{2}} \leq N^{2\varepsilon^2}.
\]
Therefore, if $\mathbf{x}$ is not $m$-black for any $m\in \Iintv{\varepsilon^2 M , 2M-\varepsilon^2 M}$, then
\al{
&T_{\mathrm{Cyl}_{\mathbf{0}, \mathbf{v}}([-2\Delta_m,2\Delta_m], N^{\frac{\overline{\chi}_u+1+\varepsilon}{2}}) }(\mathbf{0}, N\mathbf{v})\\
&\leq \sum_{m\in \Iintv{\varepsilon^2 M , 2M-\varepsilon^2 M}}T_{\mathrm{Cyl}_{\mathbf{0}, \mathbf{v}}([-2\Delta_m,2\Delta_m], N^{\frac{\overline{\chi}_u+1+\varepsilon}{2}}) }(\overline{\mathbf{x}}_{m},\overline{\mathbf{x}}_{m+1}) \\
&\qquad\qquad + \sum_{m\in  \Iintv{0 , 2M-2}\setminus  \Iintv{\varepsilon^2 M , 2M-\varepsilon^2 M}} T_{\mathrm{Cyl}_{\mathbf{0}, \mathbf{v}}([-2\Delta_m,2\Delta_m], N^{\frac{\overline{\chi}_u+1+\varepsilon}{2}}) }(\overline{\mathbf{x}}_{m},\overline{\mathbf{x}}_{m+1})  \\
&\leq \sum_{m\in \Iintv{\varepsilon^2 M , M-2}}((N^{\frac{m+1}{M}}-N^{\frac{m}{M}})\mu(\mathbf{v})+ (4M)^{-1} N^{\overline{\chi}_u +\varepsilon})+  (N-2N^{\frac{M-1}{M}})\mu(\mathbf{v})+ (4M)^{-1} N^{\overline{\chi}_u + \varepsilon }\\
&\qquad\qquad+ \sum_{m\in \Iintv{M,2M-\varepsilon^2 M}}((N^{\frac{2M-1-m}{M}}-N^{\frac{2M-2-m}{M}})\mu(\mathbf{v})+ (4M)^{-1} N^{\overline{\chi}_u + \varepsilon})  + 4 \varepsilon^2 M N^{2\varepsilon^2}\\
&< N\mu(\mathbf{v})+ N^{\overline{\chi}_u +\varepsilon},
}
contradicting the occurrence of \eqref{eq: event_bad_boxes_2}. Hence, if \eqref{eq: event_bad_boxes_2} occurs, then for any $\mathbf{x}\in \overline{\mathcal{Z}}_{N}$, there exists $m\in \Iintv{\varepsilon^2 M , 2M-\varepsilon^2 M}$ such that  $\mathbf{x}$ is $m$-black. The lemma then follows from the pigeonhole principle and the fact that $|\overline{\mathcal{Z}}_{N}|\geq N^{\varepsilon^7}$ for $N$ large enough. 
\end{proof}

\begin{proof}[Proof of Theorem~\ref{thm: refinement of condition}]
Let $\varepsilon>0$. Without loss of generality, we can assume that $\overline{\chi}_u<1$, since the claim is trivial if $\overline{\chi}_u=1$ for any  $\varepsilon>0$. We take  $\mathbf{v} \in H_{\mathbf{u}}$ with $|\mathbf{u-v}|<N^{\frac{\overline{\chi}_u-1+\varepsilon+\varepsilon^5}{2}}$. Suppose that the event \eqref{eq: event_bad_boxes_2} occurs. By  Lemma~\ref{prop: many bad boxes 2}, there exists $m\in \Iintv{\varepsilon^2 M,2M-\varepsilon^2 M}$ such that 
$| \overline{\mathcal{Z}}_{N,m}^{\rm black}|\geq  N^{\varepsilon^{8}}.$ By symmetry, without loss of generality, we may assume $m\in  \Iintv{\varepsilon^2 M,M-1}$. 

We first consider $m\leq M-2$, where in this case $\Delta_m = N^{\frac{m+1}{M}}-N^{\frac{m}{M}}$. We claim that if $N$ is sufficiently large, then
\begin{equation} \label{eq:mu_ineq}
\Delta_m \mu(\mathbf{v}) + (4M)^{-1} N^{\overline{\chi}_u + \varepsilon}
\geq \mu(\overline{\mathbf{x}}_{m+1} - \overline{\mathbf{x}}_{m}) + |\overline{\mathbf{x}}_{m+1} - \overline{\mathbf{x}}_{m}|^{\overline{\chi}_u + (\varepsilon/2)}.
\end{equation}
By the definition of  $\overline{\mathbf{x}}_{m}$, we have the bound $$|(\overline{\mathbf{x}}_{m+1} - \overline{\mathbf{x}}_{m}) - \Delta_m \mathbf{v}| \leq 2d N^{\frac{(m+1)(\overline{\chi}_u + 1 +\varepsilon -\varepsilon^2)}{2M}}.$$ Hence, together with $| \mathbf{v}-\mathbf{u}|\leq  N^{\frac{\overline{\chi}_u-1+\varepsilon + \varepsilon^5}{2}}$, $v\in H_\mathbf{u}$ and $\Delta_m\leq N^{1-M^{-1}}\leq N^{1-(\varepsilon^{4}/8)}$, we have
\al{
|(\overline{\mathbf{x}}_{m+1} - \overline{\mathbf{x}}_{m}) - \Delta_m \mathbf{u}|&\leq |(\overline{\mathbf{x}}_{m+1} - \overline{\mathbf{x}}_{m}) - \Delta_m \mathbf{v}|+ |\Delta_m \mathbf{v}-\Delta_m \mathbf{u}| \\
&\leq 2 \Delta_m N^{\frac{\overline{\chi}_u-1+\varepsilon + \varepsilon^5}{2}} \leq N^{\frac{\overline{\chi}_u+1+\varepsilon - \varepsilon^5}{2}}.
}
Moreover, Assumption~\ref{assum: finite curvature} gives
\al{
|\mu(\overline{\mathbf{x}}_{m+1} - \overline{\mathbf{x}}_{m}) - \mu(\Delta_m \mathbf{v})|& \leq |\mu(\overline{\mathbf{x}}_{m+1} - \overline{\mathbf{x}}_{m}) - \mu(\Delta_m \mathbf{u})|+|\mu(\Delta_m \mathbf{u})- \mu(\Delta_m \mathbf{v})|\\
& \leq 2 C N^{{\overline{\chi}_u+\varepsilon - \varepsilon^5}}\leq (8M)^{-1} N^{{\overline{\chi}_u+\varepsilon}}.
}
Together with $|\overline{\mathbf{x}}_{m+1} - \overline{\mathbf{x}}_{m}|\leq N$, we obtain \eqref{eq:mu_ineq}.

Hence, using Markov's inequality, \eqref{lem: Theorem 5.2 and Proposition 5.8}, and the definition of $\overline{\chi}_u$, and recalling that $M=\lceil 4\varepsilon^{-4}\rceil$, for $\varepsilon>0$ small enough and $N$ large enough, we have
\aln{
&\mathbb{P}(\text{$\mathbf{x}\in \overline{\mathcal{Z}}_{N}$ is $m$-black})\notag\\
&\leq \mathbb{P}(T_{{\rm Cyl}_{\mathbf{0},\overline{\mathbf{x}}_{m+1}-\overline{\mathbf{x}}_{m}}([-2\Delta_m,2\Delta_m],|\overline{\mathbf{x}}_{m+1}-\overline{\mathbf{x}}_{m}|^{\frac{1+\overline{\chi}_u+(\varepsilon/2)}{2}})}(\mathbf{0},\overline{\mathbf{x}}_{m+1}-\overline{\mathbf{x}}_{m})\geq  \mu(\overline{\mathbf{x}}_{m+1}-\overline{\mathbf{x}}_{m})+  \Delta_m^{\overline{\chi}_u +(\varepsilon/2)})\notag\\
&\leq  \Delta_m^{-\varepsilon/4}\leq  N^{-\varepsilon^6}.\notag
}

Next, we consider $m=M-1$. In this case, since $\overline{\mathbf{x}}_{M}-\overline{\mathbf{x}}_{M-1} = \Delta_{M-1}\mathbf{v}$, similar to above we have
\aln{
&\mathbb{P}(\text{$\mathbf{x}\in \overline{\mathcal{Z}}_{N}$ is $(M-1)$-black})\notag\\
&\leq \mathbb{P}(T_{{\rm Cyl}_{\mathbf{0},\overline{\mathbf{x}}_{M}-\overline{\mathbf{x}}_{M-1}}([-2\Delta_m,2\Delta_m],|\overline{\mathbf{x}}_{M}-\overline{\mathbf{x}}_{M-1}|^{\frac{1+\overline{\chi}_u+(\varepsilon/2)}{2}})}(\mathbf{0},\overline{\mathbf{x}}_{M}-\overline{\mathbf{x}}_{M-1})\geq  \mu(\overline{\mathbf{x}}_{M}-\overline{\mathbf{x}}_{M-1})+  \Delta_{M-1}^{\overline{\chi}_u +(\varepsilon/2)})\notag\\
&\leq  \Delta_{M-1}^{-\varepsilon/4}\leq  N^{-\varepsilon^6}.\notag
}

For any $m\in \Iintv{\varepsilon^2 M,M-1}$, since the random variables $$(\mathbf{1}\{\text{$\mathbf{x}$ is $m$-black}\})_{\mathbf{x}\in \overline{\mathcal{Z}}_{N}}$$  are independent (because the cylinders in the definition of $m$-black are disjoint), using a Chernoff bound, 
    \aln{\label{eq: union bound2}
  \mathbb{P}(| \overline{\mathcal{Z}}_{N,m}^{\rm black}|\geq  N^{\varepsilon^{8}}) &\leq  |\overline{\mathcal{Z}}_{N}|^{ N^{\varepsilon^{8}}} (N^{-\varepsilon^6})^{ N^{\varepsilon^{8}}}\leq e^{-N^{\varepsilon^{8}}}
}
    for $N$ large enough, where we have used $|\overline{\mathcal{Z}}_{N}|= O(N^{d\varepsilon^7})$ and the choice of $\varepsilon$ from \eqref{cond: vareps condition}.

Putting things together with a union bound over $m\in \Iintv{\varepsilon^2 M,M-\varepsilon^2 M}$, taking $\theta_0= \varepsilon^{9}$ and $c_0=\varepsilon^5$, we prove the theorem.
\end{proof}
\subsection{Proof of Theorem~\ref{thm: bound for chiu}}\label{sec:refinement-2}
 Let $\delta\in (0,1/5)$ be a fixed small constant. We set
\[
\alpha := \frac{4}{5} +\delta
,\,J := \lfloor N^{1-\alpha}\rfloor,\,K := N/J, \quad \overline{\chi} := 1 - \frac{\alpha}{2} + 2\delta.
\]
Note that $\overline{\chi}\leq \frac{3}{5} + 2\delta$. 
By the Cauchy--Schwarz inequality and the fact that $\E\Big[T_{{\rm Cyl}_{\mathbf{0},\mathbf{v}}^{\mathbf{v}}(\mathbb{R},N^{\frac{\overline{\chi}+1}{2}})}(\mathbf{0}, N\mathbf{v})^2\Big]= O(N^2)$, it suffices to show that there exists $c>0$ such that for any $\mathbf{v}\in S^{d-1},$
\aln{\label{eq: prob_bound_chi_u}
\mathbb{P}\Bigl(T_{{\rm Cyl}_{\mathbf{0},\mathbf{v}}^{\mathbf{v}}(\mathbb{R},N^{\frac{\overline{\chi}+1}{2}})}(\mathbf{0}, N\mathbf{v}) > N\mu(\mathbf{v}) + N^{\overline{\chi}}\Bigr)\leq e^{-N^c}.
}
Indeed, if \eqref{eq: prob_bound_chi_u} holds, then using \eqref{ass: weight distribution}, for some constant $C>0$,
\begin{align*}
    &\E (T_{{\rm Cyl}_{\mathbf{0},\mathbf{v}}^{\mathbf{v}}(\mathbb{R},N^{\frac{\overline{\chi}+1}{2}})}(\mathbf{0}, N\mathbf{v}) - N\mu(\mathbf{v}))_+ \\
    \leq&\; CNe^{-\frac{N^c}{2}} + \E [(T_{{\rm Cyl}_{\mathbf{0},\mathbf{v}}^{\mathbf{v}}(\mathbb{R},N^{\frac{\overline{\chi}+1}{2}})}(\mathbf{0}, N\mathbf{v}) - N\mu(\mathbf{v}))_+ \mathbf{1}_{\{(T_{{\rm Cyl}_{\mathbf{0},\mathbf{v}}^{\mathbf{v}}(\mathbb{R},N^{\frac{\overline{\chi}+1}{2}})}(\mathbf{0}, N\mathbf{v}) - N\mu(\mathbf{v}))_+ \leq N^{\overline{\chi}}\}}]\\
    \leq &\; CNe^{-\frac{N^c}{2}} + N^{\overline{\chi}},
\end{align*}
which will imply that $\overline{\chi}_u \leq \overline{\chi}$. Letting $\delta \to 0$, we obtain  $\overline{\chi}_u \le 3/5$.

Thus, it remains to show \eqref{eq: prob_bound_chi_u}. We define
\[
T_i := T_{{\rm Cyl}^{\mathbf{v}}_{iK\mathbf{v},\mathbf{v}}(\mathbb{R},N^{\alpha+\delta})}(iK\mathbf{v}, (i+1)K\mathbf{v}).
\]
Then, by a standard concentration result (see, e.g., \cite[Theorem 3.10]{50years}), there exists a constant $c=c(\delta)>0$ such that
\[
\mathbb{P}\Bigl(T_i > K\mu(\mathbf{v}) + K^{\frac{1}{2} + \delta}\Bigr) \leq e^{-N^{c}}.
\]
Since
\[
T_{{\rm Cyl}^{\mathbf{v}}_{\mathbf{0},\mathbf{v}}(\mathbb{R},N^{\frac{\overline{\chi}+1}{2}})}(\mathbf{0}, N\mathbf{v}) \leq \sum_{i=0}^{J-1} T_i,
\]
we deduce that
\al{
&\mathbb{P}\Bigl(T_{{\rm Cyl}^{\mathbf{v}}_{\mathbf{0},\mathbf{v}}(\mathbb{R},N^{\frac{\overline{\chi}+1}{2}})}(\mathbf{0}, N\mathbf{v}) > N\mu(\mathbf{v}) + N^{\overline{\chi}}\Bigr)\\
&\leq \mathbb{P}\Bigl(\exists\, i\in\Iintv{0,J-1}:\; T_i > J^{-1}\Bigl(N\mu(\mathbf{v}) + N^{\overline{\chi}}\Bigr)\Bigr).
}
A straightforward computation with $\alpha\leq (4/5)+\delta$ shows that, for the chosen parameters,
\[
J^{-1}\Bigl(N\mu(\mathbf{v}) + N^{\overline{\chi}}\Bigr) \ge K\mu(\mathbf{v}) + K^{\frac{1}{2} + \delta}.
\]
Thus, by a union bound we obtain
\[
\mathbb{P}\Bigl(\exists\, i\in\Iintv{0,J-1}:\; T_i > K\mu(\mathbf{v}) + K^{\frac{1}{2} + \delta}\Bigr) \le J\, e^{-N^c}.
\]
which decays stretched-exponentially in $N$. This completes the proof.
\section{Upper bound for lower tail moderate deviations}\label{sec:upper bound lower tail}
We will prove the lower tail moderate deviation in this section. The proof follows an idea similar to that of the proof of Theorem~\ref{thm: key prop for upper bound}.
\subsection{Proof of Theorem~\ref{thm: upper bound for  lower tail MD}-(1)}
Let us define the function 
\al{
\underline{f}^{\overline{\chi}}(x) := \min\left\{\frac{x}{1-{\overline{\chi}}}, \frac{1}{1-{\overline{\chi}}}\right\}.
}
The following plays a key role for the inductive scheme.
\begin{prop}\label{prop: key prop for upper bound2}
 Assume Assumption~\ref{ass: lower tail concentration} with ${\overline{\chi}}\in (0,1)$, $\theta_0>0$.    For any $\varepsilon>0 $,  if $N\in\N$ is large enough, then for any  $a\in [{\overline{\chi}}+\varepsilon,1]$ and any $\mathbf{v}\in H_{\mathbf{u}}$,
 \al{
\mathbb{P} \Big(T(\mathbf{0}, N\mathbf{v})< N\mu(\mathbf{u})- N^{a}\Big)\leq  \exp{\left(-N^{(1-\varepsilon)\underline{f}^{{\overline{\chi}}+\varepsilon}(\theta_0) (a-{\overline{\chi}}-\varepsilon)}\right)}.
}
\end{prop}
\begin{proof}
Let $\varepsilon>0$. Without loss of generality, we may assume $\varepsilon<(8d+(1-\overline{\chi})^{-1}+\theta_0^{-1})^{-2d}$. If $a\in [{\overline{\chi}} +\varepsilon,{\overline{\chi}} +\varepsilon+\varepsilon^{3/2}]$, Assumption~\ref{ass: lower tail concentration} directly implies the proposition, similar to the proof of Proposition~\ref{prop: key prop for upper bound} (see the discussion right below the statement of Proposition~\ref{prop: key prop for upper bound}). Hence, assume  $a\geq {\overline{\chi}} +\varepsilon+\varepsilon^{3/2}$ and  define
\[
J := \lfloor N^{1-\frac{1-a}{1-{\overline{\chi}}-\varepsilon} - \varepsilon^3}\rfloor \quad \text{and} \quad K := \frac{N}{J}\geq N^{\frac{1-a}{1-{\overline{\chi}}-\varepsilon}+\varepsilon^3}.
\]
We also let $L:=M:=\lceil 4\varepsilon^{-4}\rceil$. Define
\al{
F_i : =( iK \mathbf{u}+ H_{\mathbf{u}}) \cap [-N^2,N^2]^d.
}
By a union bound over possible pairs $(\mathbf{x},\mathbf{y})\in F_i\times F_{i+1}$ (the number of such pairs is polynomial in $N$), the condition~\eqref{Assum:initial concentration for lower tail} implies that  for any $b\in [{\overline{\chi}},1]$, for $K$ large enough,
    \aln{\label{eq: conclusion of assumption}
    \mathbb{P}(T(F_i,F_{i+1})<K\mu(\mathbf{u})- K^b) \leq N^{8d} e^{ -K^{\theta_0(b-{\overline{\chi}})}}.
    }
    We have
    \aln{\label{eq: exit bot}
  &  \mathbb{P}(T(\mathbf{0}, N\mathbf{v})<N\mu(\mathbf{u}) - N^a)\notag \\
  & \leq \mathbb{P}(T_{[-N^2,N^2]^{d}}(\mathbf{0}, N\mathbf{v})<N\mu(\mathbf{u} )- N^a) + \mathbb{P}(\exists \gamma\not\subset[-N^2,N^2]^{d}, \mathbf{0}\in \gamma, \,T(\gamma)<N\mu(\mathbf{u} )- N^a).
    }
    Note that the second term in \eqref{eq: exit bot} is at most $e^{-\Omega(N^2)}$ by \eqref{eq: Kesten Propositin 5.8}, which is negligible.
    
    Let 
    $$b_\sharp^{\max} :=\min\left\{1, \frac{1-{\overline{\chi}}-\varepsilon}{1-a}a\right\}.$$ 
    We will prove that if 
    $$T_{[-N^2,N^2]^{d}}(\mathbf{0}, N\mathbf{v})<N\mu(\mathbf{u}) - N^a,$$
    then there exists $b\in \Iintv{{\overline{\chi}}+\varepsilon,b_\sharp^{\rm max}}_L$ such that 
    \aln{\label{eq: middle goal}
    |\{i\in \Iintv{0,J-1}:~T(F_i,F_{i+1})<K\mu(\mathbf{u}) - K^{b}\}|\geq N^a K^{-b-2\varepsilon^4}.
    }
    Indeed, if there exists $i_0\in \Iintv{0,J-1}$ such that $T(F_{i_0},F_{i_0+1})<K\mu(\mathbf{u}) - K^{b_\sharp^{\rm max}}$, then we must have $b_\sharp^{\rm max} = \frac{1-{\overline{\chi}}-\varepsilon}{1-a}a$. Otherwise, if $b_\sharp^{\rm max} = 1$, then by \eqref{ass: weight distribution}, 
    \[
    T(F_i,F_{i+1})<K\mu(\mathbf{u}) - K^{b_\sharp^{\rm max}} = K\mu(\mathbf{u}) - K<0,
    \]
    which is impossible. In this case, \eqref{eq: middle goal} trivially holds for any $b\in \Iintv{b_\sharp^{\rm max}-\varepsilon^4,b_\sharp^{\rm max}}_L$, since $N^a K^{-b-2\varepsilon^4}<1.$ We suppose that for any $i \in \Iintv{0,J-1}$, $T(F_i,F_{i+1})\geq K\mu(\mathbf{u}) - K^{b_\sharp^{\rm max}}$. Then, by $N=KJ$, using a similar argument leading to \eqref{eq: bad_ineq}, we have 
    \al{
    N\mu(\mathbf{u}) - N^a&>T_{[-N^2,N^2]^{d}}(\mathbf{0}, N\mathbf{v})\\
    &\geq \sum_{i\in\Iintv{0,J-1}} T(F_i,F_{i+1})\\
    &\geq N \mu(\mathbf{u}) - J K^{{\overline{\chi}} +\varepsilon + \varepsilon^4}\\
    &\qquad-\sum_{b\in \Iintv{{\overline{\chi}}+\varepsilon, b_\sharp^{\rm max}}_L} K^{b+\varepsilon^4} |\{i\in \Iintv{0,J-1}:~T(F_i,F_{i+1})<K\mu(\mathbf{u}) - K^{b}\}|.
    }
    By the pigeonhole principle and $J K^{{\overline{\chi}} +\varepsilon + \varepsilon^4} = N K^{-(1-{\overline{\chi}} -\varepsilon) + \varepsilon^4} \ll N^a$, there exists  $b\in \Iintv{{\overline{\chi}}+\varepsilon, b_\sharp^{\rm max}}_L $ such that
    \begin{equation}
    \label{eq: some_eq_in_4.1}
    K^{b+\varepsilon^4} |\{i\in \Iintv{0,J-1}:~T(F_i,F_{i+1})<K\mu(\mathbf{u}) - K^{b}\}|\geq \frac{1}{2L}N^a,
    \end{equation}
    which proves \eqref{eq: middle goal} for $N$ large enough depending on $\varepsilon$. 
    
    Therefore, we arrive at
    \al{
    &\mathbb{P}(T_{[-N^2,N^2]^{d}}(\mathbf{0}, N\mathbf{v})<N\mu(\mathbf{u} )- N^a)\\
    &\leq \sum_{b\in \Iintv{{\overline{\chi}}+\varepsilon, b_\sharp^{\rm max}}_L }\P\Big(  |\{i\in \Iintv{0,J-1}:~T(F_i,F_{i+1})<K\mu(\mathbf{u}) - K^{b}\}|\geq N^a K^{-b-2\varepsilon^4}\Big).
    }
    By a Chernoff bound and \eqref{eq: conclusion of assumption}, this is further bounded from above by 
    \aln{
    &\sup_{b\in \Big[{\overline{\chi}}+\varepsilon, b_\sharp^{\rm max}\Big]}  L \left(N^{8d} e^{-K^{\theta_0(b-{\overline{\chi}})}}\right)^{N^a K^{-b-2\varepsilon^4}} J^{N^a K^{-b-2\varepsilon^4}}\notag\\
    &\leq \sup_{b\in\Big[{\overline{\chi}}+\varepsilon, \frac{1-{\overline{\chi}}-\varepsilon}{1-a}a\Big]} \exp{\Big( -N^{\theta_0\frac{(b-\overline{\chi}-\varepsilon)(1-a)}{1-\overline{\chi}-\varepsilon}+a - b\frac{1-a}{1-{\overline{\chi}}-\varepsilon} - \varepsilon^2} \Big)},\label{eq: numba}
    }
where the terms involving $\varepsilon^3$ and $\varepsilon^4$ are absorbed into $\varepsilon^2$. By a simple computation, we have 
    \al{
   & \inf\left\{\theta_0\frac{(b-\overline{\chi}-\varepsilon)(1-a)}{1-\overline{\chi}-\varepsilon}+a - b\frac{1-a}{1-{\overline{\chi}}-\varepsilon}:~ b\in\Big[{\overline{\chi}}+\varepsilon, \frac{1-{\overline{\chi}}-\varepsilon}{1-a}a\Big]\right\}\\
   &= \min\left\{ \frac{\theta_0(a-{\overline{\chi}}-\varepsilon)}{1-\overline{\chi}-\varepsilon},\frac{a-{\overline{\chi}}-\varepsilon}{1-{\overline{\chi}}-\varepsilon}\right\},
    }
since the above infimum  is attained at the endpoints $b={\overline{\chi}}+\varepsilon$ or $b= \frac{1-{\overline{\chi}-\varepsilon}}{1-a}a$.    Hence,  recalling that $a\geq  \overline{\chi}+\varepsilon +\varepsilon^{3/2}$, \eqref{eq: numba} is further bounded from above by 
    \al{
    \exp{\Big( -N^{(1-\varepsilon)\underline{f}^{{\overline{\chi}}+\varepsilon}(\theta_0)(a-{\overline{\chi}}-\varepsilon)}\Big)}. 
    }
    This proves the proposition.
    \end{proof}
    \begin{proof}[Proof of Theorem~\ref{thm: upper bound for  lower tail MD}-(1)]
We take $a\in ({\overline{\chi}},1)$ and $\varepsilon\in (0,a-{\overline{\chi}})$. We let $R\in\mathbb{N}$ be such that
    \begin{equation}
        \label{eq: choice of R}
    \frac{\theta_0(1-(\varepsilon^2/R))^R}{(1-{\overline{\chi}})^{R}} > \frac{1}{1-{\overline{\chi}}},
        \end{equation}
    where such a $R$ exists due to $\overline{\chi}>0$. Recalling the function $\underline{f}^{\overline{\chi}}(x):=\frac{x}{1-\overline{\chi}}\wedge \frac{1}{1-\overline{\chi}}$, we define 
    $$\theta_k:=  (1-(\varepsilon^2/R)) \underline{f}^{\overline{\chi}}(\theta_{k-1}),$$ for $k\geq 1$.  Inductively, we define 
$$\varepsilon_k:=(2d + (1-\overline{\chi}_{k-1})^{-1} +\theta_{k-1}^{-1}+ (\varepsilon^2/R)^{-1} )^{-2d},\quad  \overline{\chi}_k := \overline{\chi} + \sum_{i=1}^k \varepsilon_i .$$
Then, as in the proof of Theorem~\ref{thm: key prop for upper bound}-(1), we have
$$\theta_R\geq (1-\varepsilon^2) \frac{1}{1-\overline{\chi}},\quad \theta_k \leq (1-\varepsilon_k) \underline{f}^{\overline{\chi}_{k-1}+\varepsilon_k}(\theta_{k-1}).$$
Then, we apply Proposition~\ref{prop: key prop for upper bound2} to obtain integers $N_1 < N_2 < \cdots < N_R$ with $\overline{\chi}_k$ and $\theta_k$ in place of $\overline{\chi}$ and $\theta_0$ in the $k$-th step, ensuring that Assumption~\ref{ass: lower tail concentration}  holds with $\overline{\chi}_k$ and $\theta_k$ for any $N\geq N_k$. Arguing similarly to the proof of Theorem~\ref{thm: key prop for upper bound}-(1), we obtain Theorem~\ref{thm: upper bound for  lower tail MD}-(1).
    \end{proof}
\subsection{Proof of Theorem~\ref{thm: upper bound for  lower tail MD}-(2)}
Let  $\zeta>0$. For simplicity of notation, we first consider the deviation $\zeta^{{1-\varepsilon^2}} N$, and replace it by $\zeta N$ afterwards as before. Set
\[
\widehat{J} = \lfloor \zeta^{\frac{1}{1-{\overline{\chi}}-\varepsilon} - \varepsilon^3} N\rfloor \quad \text{and} \quad \widehat{K} := \frac{N}{\widehat{J}}\geq \zeta^{-\frac{1}{1-{\overline{\chi}}-\varepsilon}-\varepsilon^3}.
\]
We define
\al{
\widehat{F}_i : = i\widehat{K} \mathbf{u}+ H_{\mathbf{u}}.
}
Let $A > 4d$ be a large constant to be determined later. Given $\mathbf{x}\in \widehat{K}\Z^d$, we say that $\mathbf{x}$ is $b$-dark if there exists $
 \mathbf{y}$ with $|\mathbf{y}-\mathbf{x}|_\infty\leq 2\widehat{K}$ such that either:
\al{
T_{\mathbf{x}+[-A\widehat{K},A\widehat{K}]^d}(\mathbf{y},\mathbf{y}+\widehat{K} H_{\mathbf{u}})<\widehat{K}\mu(\mathbf{u})- 2\widehat{K}^b \quad\text{or}\quad \exists \gamma'\in \mathbb{G}(\mathbf{y},\mathbf{y}+\widehat{K} H_{\mathbf{u}})\not\subset \mathbf{x}+[-A\widehat{K},A\widehat{K}]^d.
}

\begin{lem}\label{lem: dark probab}
For any $\varepsilon>0$ and $b\in \Iintv{\overline{\chi}+\varepsilon,1}_L$ with $L:=\lceil 4\varepsilon^{-4}\rceil$, if $\widehat{K}$ is large enough, then  for any $\mathbf{x}\in \R^d$, 
    \al{
    \mathbb{P}(\textnormal{$\mathbf{x}$ is $b$-dark})\leq \widehat{K}^{4d} e^{-\widehat{K}^{(1-\varepsilon^4)\frac{b-\overline{\chi}}{1-\overline{\chi}}}}.
    }
\end{lem}
\begin{proof}
    Let $\mathbf{y}$ with $|\mathbf{y}-\mathbf{x}|_\infty\leq 2\widehat{K}$. For the points close to $\mathbf{y}$, we have
    \al{
    &\mathbb{P}(T_{\mathbf{x}+[-A\widehat{K},A\widehat{K}]^d}(\mathbf{y},\mathbf{y}+\widehat{K} H_{\mathbf{u}})<\widehat{K}\mu(\mathbf{u})- 2\widehat{K}^b)\\
    &\leq \sum_{\mathbf{y}'}\mathbb{P}(T(\mathbf{y},\mathbf{y}')<\widehat{K}\mu(\mathbf{u})- \widehat{K}^b)\leq (2A\widehat{K}+1)^d e^{-\widehat{K}^{(1-\varepsilon^4)\frac{b-\overline{\chi}}{1-\overline{\chi}}}},
    }
    where $\mathbf{y}'$ runs over $\{ \lfloor \mathbf{z}\rfloor:~\mathbf{z}\in  (\mathbf{x}+([-A\widehat{K},A\widehat{K}]^d\cap \widehat{K} H_{\mathbf{u}}))\}$, and we have used Theorem~\ref{thm: upper bound for  lower tail MD}-(1) in the last inequality. For the second condition, by Lemma~\ref{lem: Theorem 5.2 and Proposition 5.8}, if $A$ is taken to be sufficiently large, we have
    \begin{equation}
    \label{eq: geo1}
    \mathbb{P}(\exists \gamma'\in \mathbb{G}(\mathbf{y},\mathbf{y}+\widehat{K} H_{\mathbf{u}})\not\subset \mathbf{x}+[-A\widehat{K},A\widehat{K}]^d)\leq e^{-c \widehat{K}},
    \end{equation}
    for some $c>0$. A union bound over $\mathbf{y}$ (the number of such $\mathbf{y}$ is polynomial in $K$) finishes the proof.
\end{proof}
We take $\gamma\in \mathbb{G}(\mathbf{0}, N\mathbf{u})$ with a deterministic rule breaking ties. We set $$\Lambda_{\widehat{K}}(\gamma):= \left\{\mathbf{x}\in \Z^d:~\gamma\cap (\widehat{K} \mathbf{x}+[0,\widehat{K})^d)\neq\emptyset\right\}.$$
Note that $ \Lambda_{\widehat{K}}(\gamma)$ is  a connected set in $\Z^d$ and 
$$\widehat{K} |\Lambda_{\widehat{K}}(\gamma)|\leq \sum_{\mathbf{x}\in \Z^d} \sum_{\mathbf{y}\in \gamma}\mathbf{1}_{\mathbf{y}\in (\widehat{K}\mathbf{x}+[-\widehat{K},2\widehat{K})^d)} = \sum_{\mathbf{y}\in \gamma} \sum_{\mathbf{x}\in \Z^d} \mathbf{1}_{\mathbf{y}\in (\widehat{K}\mathbf{x}+[-\widehat{K},2\widehat{K})^d)}\leq  3^d |\gamma|,$$ where $|\gamma|$ is the number of vertices in $\gamma$. Moreover, by \eqref{eq: Kesten Propositin 5.8}, by increasing $A$ if necessary, we have
\begin{equation}
\label{eq: geo2}
\mathbb{P}(\exists \gamma\in \mathbb{G}(\mathbf{0}, N\mathbf{u}),\,|\gamma|\geq AN)\leq e^{-N}.
\end{equation}
We will fix the constant $A$ such that both \eqref{eq: geo1} and \eqref{eq: geo2} hold. Hence, since $\widehat{K}^{-1} N= \widehat{J}$, we have
\begin{equation}
\label{eq: confinement}
\begin{split}
    &\mathbb{P}(T(\mathbf{0}, N\mathbf{u})<N\mu(\mathbf{u}) - \zeta^{{1-\varepsilon^2}} N)\\
&\leq \mathbb{P}(T(\mathbf{0}, N\mathbf{u})<N\mu(\mathbf{u}) - \zeta^{{1-\varepsilon^2}} N,\,|\Lambda_{\widehat{K}}(\gamma)|\leq 3^d A \widehat{J}) + e^{-N}.
\end{split}
\end{equation}
Write $\gamma = (\mathbf{x}_i)_{i=1}^{\ell}$ as a sequence in $\Z^d$. Given $j\in \Iintv{0,\widehat{J}-1}$, we set
\begin{figure}[t]
    \centering
   \includegraphics[width=0.7\linewidth]{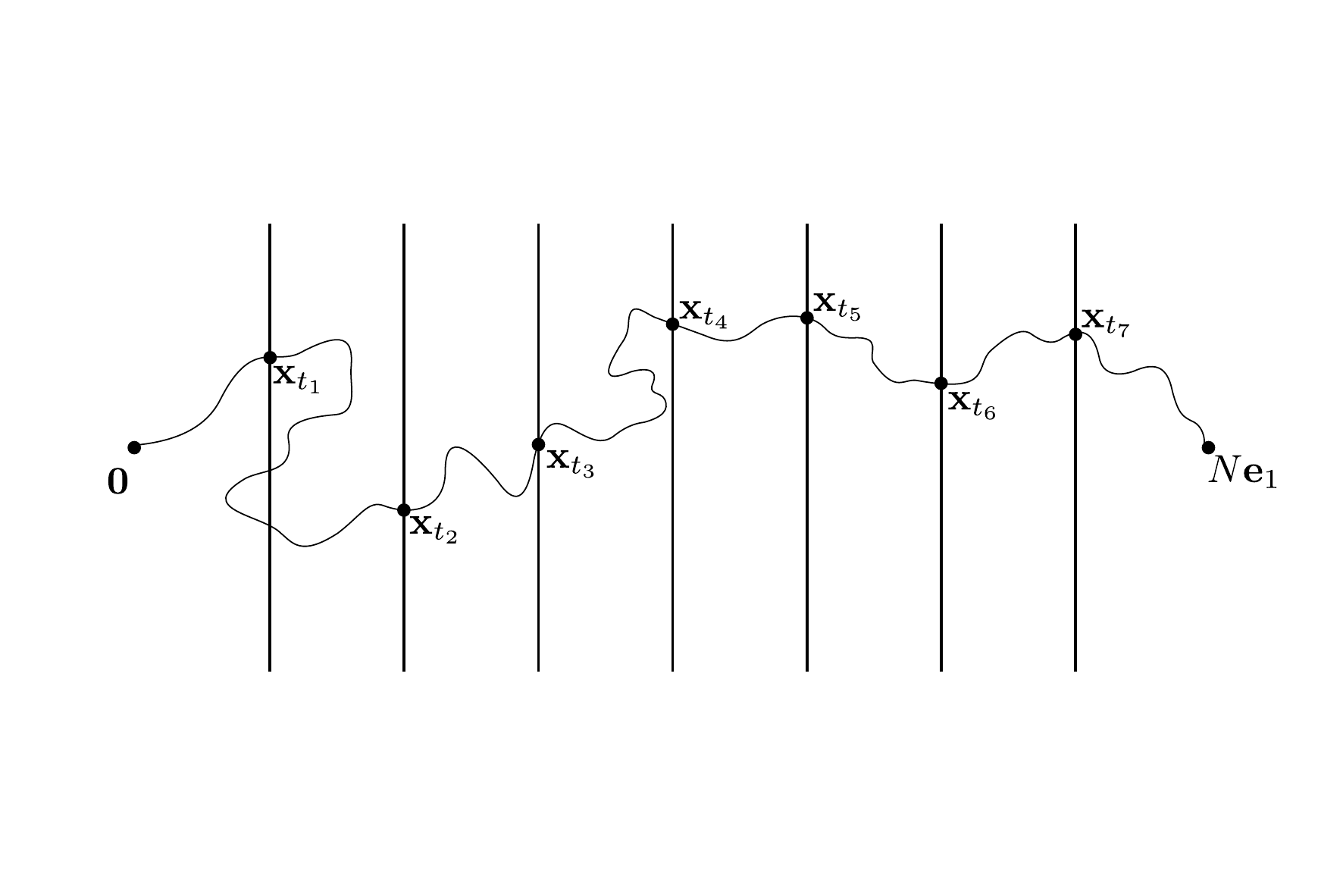} 
    \caption{The path from $\mathbf{0}$ to $N\mathbf{e}_1$ is a geodesic. It intersects the faces $j\widehat{K}H_{\mathbf{e}_1}$, $j=0, 1, \ldots, \widehat{J} - 1$. $\mathbf{x}_{t_j}$ is the first time the path touches $j\widehat{K}H_{\mathbf{e}_1}$, considering the path starting from $\mathbf{0}$.}
    \label{fig:xtj}
\end{figure}
\al{
t_j := \inf\{i\in \Iintv{1,\ell}:~ \exists \mathbf{y}\in j \widehat{K} H_{\mathbf{u}},\,\mathbf{x}_i=\lfloor \mathbf{y}\rfloor  \}.
}
(See also Figure~\ref{fig:xtj}.) If 
    $$T(\mathbf{0}, N\mathbf{u})<N\mu(\mathbf{u}) - \zeta^{{1-\varepsilon^2}} N,$$
    then since $T(\mathbf{0}, N\mathbf{u})\geq \sum_{j\in \Iintv{0,\widehat{J}-1}}T(\mathbf{x}_{t_j},\mathbf{x}_{t_{j}}+\widehat{K}H_{\mathbf{u}})$ and $\zeta^{{1-\varepsilon^2}} N \gg \widehat{J} \widehat{K}^{\overline{\chi}+\varepsilon+\varepsilon^4}$, by a similar argument leading to \eqref{eq: some_eq_in_4.1}, there exists $b\in \Iintv{{\overline{\chi}}+\varepsilon,1}_L$ such that 
    \aln{\label{eq: middle goal2}
    |\{q\in \Iintv{0,\widehat{J}-1}:~T(\mathbf{x}_{t_j},\mathbf{x}_{t_{j}}+\widehat{K}H_{\mathbf{u}})<\widehat{K}\mu(\mathbf{u}) - 2\widehat{K}^{b}\}|\geq \zeta^{{1-\varepsilon^2}} \widehat{K}^{-b-\varepsilon^4} N.
    }
    This implies
    \al{
     |\{\mathbf{x}\in \Lambda_{\widehat{K}}(\gamma):~\text{$\widehat{K}\mathbf{x}$ is $b$-dark}\}|\geq\zeta^{{1-\varepsilon^2}} \widehat{K}^{-b-\varepsilon^4} N.
    }
    Hence,  letting ${\rm Conn}(R):= \{U\subset \Z^d:~\mathbf{0}\in U,\,|U|\leq R,\,\text{$U$ is connected in $\Z^d$}\}$, we have
    \al{
     &\mathbb{P}(T(\mathbf{0}, N\mathbf{u})<N\mu(\mathbf{u} )- \zeta^{{1-\varepsilon^2}} N,\,|\Lambda_{\widehat{K}}(\gamma)|\leq 3^d A \widehat{J})\\
    &\leq \sum_{b\in \Iintv{{\overline{\chi}}+\varepsilon,1}_L}\mathbb{P}( |\{\mathbf{x}\in \Lambda_{\widehat{K}}(\gamma):~\text{$\widehat{K}\mathbf{x}$ is $b$-dark}\}|\geq \zeta^{{1-\varepsilon^2}} \widehat{K}^{-b-\varepsilon^4} N,\,|\Lambda_{\widehat{K}}(\gamma)|\leq 3^d A \widehat{J}) \\
    &\leq \sum_{b\in \Iintv{{\overline{\chi}}+\varepsilon,1}_L}\sum_{\mathbf{y}\in \Iintv{-2A,2A}^d} \sum_{U\in {\rm Conn}(3^d A  \widehat{J})}\mathbb{P}( |\{\mathbf{x}\in U\cap (\widehat{K}\mathbf{y}+2A\widehat{K} \mathbb{Z}^d):\,\text{$\widehat{K}\mathbf{x}$ is $b$-dark}\}|\geq (5A)^{-d} \zeta^{{1-\varepsilon^2}} \widehat{K}^{-b-\varepsilon^4} N ).
    }
    Therefore, since $|{\rm Conn}(R)|\leq e^{C'R}$ for some $C'=C'(d)>0$ by \cite[(4.24)]{grimmett2010percolation}, and 
    $$|\{\mathbf{x}\in U\cap (\widehat{K}\mathbf{y}+2A\widehat{K} \mathbb{Z}^d):\,\text{$\widehat{K}\mathbf{x}$ is $b$-dark}\}|= \sum_{\mathbf{x}\in U\cap (\widehat{K}\mathbf{y}+2A\widehat{K} \mathbb{Z}^d)}\mathbf{1}\{\text{$\widehat{K}\mathbf{x}$ is $b$-dark}\}$$ forms an i.i.d.\ sum,   by the Chernoff inequality together with Lemma~\ref{lem: dark probab},  this is further bounded from above by
    \al{
    \sup_{b\in[{\overline{\chi}}+\varepsilon, 1]} (5A)^d\,L\,  e^{ 3^d C' A \widehat{J}} \, \Big(\widehat{K}^{4d} e^{-\widehat{K}^{(1-\varepsilon^4)\frac{b-\overline{\chi}}{1-\overline{\chi}}}}\Big)^{(10A)^{-d} \zeta^{{1-\varepsilon^2}} N\,\widehat{K}^{-b-\varepsilon^4}}\leq \exp{\Big( -\zeta^{\frac{1}{1-{\overline{\chi}}} +  O(\varepsilon)}\,N \Big)},
    }
where we evaluated the supremum at the endpoints $b=\overline{\chi}+\varepsilon$ or $b=1$ to show that for $\varepsilon$ small enough, 
\al{
\widehat{J} =\frac{N}{\widehat{K}}\ll \big(\widehat{K}^{\frac{b-\overline{\chi}}{1-\overline{\chi}}} \big) ({\zeta N\,\widehat{K}^{-b}}).  
}
Finally, replacing $\zeta^{{1-\varepsilon^2}} N$ by $\zeta N$, together with \eqref{eq: confinement} proves the theorem.

\section{Lower bound for upper tail moderate deviations}
\label{sec: lower_bound_upper_tail}
Throughout this section, we assume Assumptions~\ref{assum: finite curvature}, \ref{cond: full concentration} and \ref{ass: positive curvature} with $\chi\in [0,1)$.  To simplify the notation, we primarily focus on the case that $\mathbf{u} = \mathbf{e}_1$. We will outline how the general case can be proved in Section~\ref{section: lower bound general direction}.
\subsection{Sketch of the proof}\label{sec:key-lower-bound}
Since the proof is quite different from those in the previous sections, before going into the details, we first give a sketch of the proof of Theorem~\ref{thm: key prop for lower bound}-(2). (Theorem~\ref{thm: key prop for lower bound}-(1) can be proved in a similar manner.) 

Let $\tilde{K} = \zeta^{-\frac{1}{1-\chi}}$, where $\zeta > 0$ is small. We put faces at $\mathbf{0}$, $\widetilde{K}\mathbf{e}_1$, $2\widetilde{K}\mathbf{e}_1$, and so on, up to $N\mathbf{e}_1$. Note that the passage time from $\mathbf{0}$ to $N\mathbf{e}_1$ is bounded below by the sum of the passage times between these faces. By using the curvature assumption (Assumption~\ref{ass: positive curvature}), we know that any geodesic between two faces cannot deviate too much, in the sense that it is more or less ``horizontal'', and lies within a box of height $\widetilde{K}^{\frac{\chi+1}{2}}$. In particular, with a reasonable probability the starting point and the ending point of a geodesic cannot be too far away in the vertical distance. This is handled in Lemma~\ref{lem: S probab} (this corresponds to the event $B_K'$ in the proof). On the other hand, if the starting point and the ending point are close in the vertical distance, then by Assumption~\ref{cond: full concentration}, with a reasonable probability, the passage time is at least $\widetilde{K}\mu(\mathbf{e}_1) + K^{\chi - \varepsilon}$ (see Lemma~\ref{lem:modified} and also the event $A_K'$ in the proof of Lemma~\ref{lem: S probab}).

On the other hand, by Assumption~\ref{ass: positive curvature} again, we can show that with high probability, if $T(\mathbf{0}, N\mathbf{e}_1) \geq N\mu(\mathbf{e}_1) + \zeta N$, then any geodesic from $\mathbf{0}$ to $N\mathbf{e}_1$ lies within $\mathbb{R} \times [-A\sqrt{\zeta}N, A\sqrt{\zeta}N]^{d-1}$ for some large constant $A$ (see Lemma~\ref{lem:exit-geodesics}).

Now, we subdivide the box $[0, N]\times [-A\sqrt{\zeta}N, A\sqrt{\zeta}N]^{d-1}$ into boxes of width $\widetilde{K}$ and height $\widetilde{K}^{\frac{\chi+1}{2}}$. If the passage time from the left side to the right side of each box is at least $\widetilde{K}\mu(\mathbf{e}_1) + K^{\chi - \varepsilon}$, then the total face-to-face passage times will be approximately
\[
\frac{N}{\widetilde{K}} \left(\widetilde{K}\mu(\mathbf{e}_1) + \widetilde{K}^{\chi - \varepsilon}\right) = N\mu(\mathbf{e}_1) + N\widetilde{K}^{\chi - 1 - \varepsilon} \approx N\mu(\mathbf{e}_1) + \zeta N.
\]
Therefore, in order to bound the rate function, it suffices to lower bound the probability that all boxes have passage time at least $\widetilde{K}\mu(\mathbf{e}_1) + K^{\chi - \varepsilon}$. But since this happens with a reasonable probability, say $c$, on each of the boxes, by the FKG inequality, the required probability is at least
\[
c^{\frac{N\, (A\sqrt{\zeta}N)^{d-1}}{\widetilde{K}\, (\widetilde{K}^{\frac{\chi+1}{2}})^{d-1}}} = \exp\left(-c' N^d \zeta^{\frac{d}{1-\chi}}\right),
\]
giving the desired conclusion.

\subsection{Deviation estimates of geodesics and passage times}
We first give some preliminary estimates on geodesics and passage times, which will be useful in the proof of Theorem~\ref{thm: key prop for lower bound}.
\begin{lem}\label{lem: exit geodesics with a}
Assume Assumptions~\ref{cond: full concentration} and \ref{ass: positive curvature}. There exists $A>1$ such that for any $a\in (\chi,1)$, there exists $c=c(a)>0$ such that for $N$ large enough,
\al{
\mathbb P(\exists \gamma:\mathbf{0}\to N\mathbf{u},\,\gamma\not\subset \mathrm{Cyl}_{\mathbf{0},\mathbf{u}}(\R, A N^{\frac{a+1}{2}}),\,T(\gamma)\leq \mu(\mathbf{u}) N + N^a)\leq e^{-N^c}.
}
\end{lem}
\begin{proof}
    By a union bound, the left side is bounded above by
    \begin{align}
    \sum_{\mathbf{z}}\,\mathbb{P}\left(T(\mathbf{0},\mathbf{z}) + T(\mathbf{z}, N\mathbf{u}) \;\le\; \mu(\mathbf{u}) N + N^a\right),\label{eq: second_term_cyl}
    \end{align}
    where the sum in \eqref{eq: second_term_cyl} is taken over all
 points    $\mathbf{z}$  in the outer boundary   of 
   $ \mathrm{Cyl}_{\mathbf{0},\mathbf{u}}([-AN,AN], A N^{\frac{a+1}{2}}) \,\cap\, \Z^d$   in  $\Z^d.$
    To bound \eqref{eq: second_term_cyl}, fix $\mathbf{z} = m\mathbf{u} + \ell\hat{\mathbf{e}}$, where $\hat{\mathbf{e}}\in {\rm span}\{\tilde{\mathbf{u}}_2,\ldots \tilde{\mathbf{u}}_d\}$ is a unit vector, so that either $|m|>AN$ or $|\ell|> A N^{\frac{a+1}{2}}$ holds.
     If $|m|>AN$, then, since $\mu(\mathbf{u})\leq 1$ by \eqref{ass: weight distribution}, using \eqref{eq: Kesten Propositin 5.8}, if $A$ is taken to be sufficiently large,    there exists $c > 0$ such that
\[
\mathbb{P}(T(\mathbf{0},\mathbf{z})+T(\mathbf{z},N\mathbf{u}) \leq \mu(\mathbf{u}) N + N^a)   \leq  \mathbb{P}(T(\mathbf{0},\mathbf{z}) \leq 2N) \leq e^{-cN}.
\]
     Next, we consider the case that $|\ell|>A N^{\frac{a+1}{2}}$ but $|m|\leq AN$.  By Assumption~\ref{ass: positive curvature}, if $A$ is large enough, then when $m\neq 0$,
    \begin{align*}
    \mu(m\mathbf{u} + \ell\hat{\mathbf{e}}) -\mu(m\mathbf{u})
    &= |m|(\mu(\mathbf{u} + (\ell/m)\hat{\mathbf{e}}) - \mu(\mathbf{u}))\\
    &\ge \frac{|m|}{C}\left(\frac{\ell}{m}\right)^{2}
    \ge 2N^a.
    \end{align*}
    We also have the same inequality for $m = 0$. Similarly we also have $\mu((N-m)\mathbf{u} - \ell \hat{\mathbf{e}}) - \mu((N-m)\mathbf{u}) \ge 2N^a$. Thus,
\begin{align*}
    \mu(\mathbf{z}) + \mu(N\mathbf{u} -\mathbf{z}) - \mu(N\mathbf{u}) &= \mu(m\mathbf{u} + \ell\hat{\mathbf{e}}) - \mu(m\mathbf{u}) + \mu((N-m)\mathbf{u} - \ell \hat{\mathbf{e}}) - \mu((N-m)\mathbf{u})\\
    &\ge 4N^a.
\end{align*}
    Then, by Assumption~\ref{cond: full concentration}, there exists $c'=c'(a)>0$ such that
\al{
    \mathbb{P}(T(\mathbf{0},\mathbf{z}) + T(\mathbf{z},N\mathbf{u})\leq \mu(\mathbf{u}) N + N^a)
    & \leq \mathbb{P}(T(\mathbf{0},\mathbf{z}) + T(\mathbf{z},N\mathbf{u})\leq \mu(\mathbf{z}) + \mu(N\mathbf{u} -\mathbf{z}) - 2N^a)\\
    &\leq \mathbb{P}(T(\mathbf{0},\mathbf{z}) \leq \mu(\mathbf{z}) - N^a) + \mathbb{P}(T(\mathbf{z},N\mathbf{u})\leq  \mu(N\mathbf{u} -\mathbf{z}) - N^a)\\
    &\leq e^{-N^{c'}}.
}
    (Here, we applied Assumption~\ref{cond: full concentration}-(2) to Assumption~\ref{cond: full concentration}-(1), so that we can replace the expectation by $\mu$ and still have the probability bound.) Summing over all relevant $\mathbf{z}$ (polynomially many) and decreasing $c'$ slightly, we are done.
\end{proof}

\begin{lem}\label{lem:exit-geodesics}
Assume Assumption~\ref{ass: positive curvature}. Then there exists a constant \(A>1\) such that for every \(\zeta>0\) there exists a constant \(c=c(\zeta)>0\) satisfying, for all sufficiently large \(N\in\mathbb{N}\),
\[
\mathbb{P}\Bigl(\exists\,\gamma:\mathbf{0}\to N\mathbf{u}, \, \gamma\not\subset \mathrm{Cyl}_{\mathbf{0},\mathbf{u}}(\mathbb{R},A\sqrt{\zeta}\,N) \textnormal{ and } T(\gamma)\le (\mu(\mathbf{u})+\zeta)N\Bigr)\le e^{-cN}\,.
\]
\end{lem}

\begin{proof}
The proof follows the same general strategy as in Lemma~\ref{lem: exit geodesics with a}, with the modification that the deviation cost is now given by \(\zeta N\).  By Assumption~\ref{ass: positive curvature}, there exists a constant \(A>1\)  such that for any \(\mathbf{z}\in \R^d\setminus 
\mathrm{Cyl}_{\mathbf{0},\mathbf{u}}([-AN,AN],A\sqrt{\zeta}\,N),\)   
one has
\[
\mu(\mathbf{z}) + \mu(N\mathbf{u}- \mathbf{z}) > (\mu(\mathbf{u})+2\zeta)N\,.
\]
An application of the lower tail large deviation \eqref{eq: lower tail LDP} shows that for each such \(\mathbf{z}\) one has
\[
\mathbb{P}\Bigl(T(\mathbf{0},\mathbf{z})+T(\mathbf{z},N\mathbf{u}) \le (\mu(\mathbf{u})+\zeta)N\Bigr)\le \mathbb{P}\Bigl(T(\mathbf{0},\mathbf{z})+T(\mathbf{z},N\mathbf{u}) \le \mu(\mathbf{z}) + \mu(N\mathbf{u}- \mathbf{z}) -\zeta N\Bigr) \le e^{-c N}\,,
\]
for some constants \(c>0\) depending on $\zeta$. A union bound over the (polynomially many) possible choices of \(\mathbf{z}\) finishes the proof.
\end{proof}

\begin{lem}\label{lem:modified}
Assume Assumption~\ref{cond: full concentration} holds. Then for any $\epsilon'>0$,  for all $N\ge 1$ large enough and for any $\mathbf{v}\in H_{\mathbf{u}}$, 
\[
\mathbb{P}\Bigl(T(\mathbf{0},N \mathbf{v})>N\mu(\mathbf{u})+N^{\chi-\epsilon'}\Bigr)\ge N^{-2}.
\]
\end{lem}

\begin{proof}
Without loss of generality, we may assume $\epsilon' < 1$. 
We decompose the variance into three terms:
\[
\begin{split}
&\mathbb{E}\Bigl[\Bigl(T(\mathbf{0}, N\mathbf{v})-\mathbb{E} T(\mathbf{0}, N\mathbf{v})\Bigr)^2\Bigr]\\
&=\mathbb{E}\Bigl[\Bigl(T(\mathbf{0}, N\mathbf{v})-\mathbb{E} T(\mathbf{0}, N\mathbf{v})\Bigr)^2\,\mathbf{1}_{\{|T(\mathbf{0}, N\mathbf{v})-\mathbb{E} T(\mathbf{0}, N\mathbf{v})|\le N^{\chi-\epsilon'}\}}\Bigr] \\
&\quad+\mathbb{E}\Bigl[\Bigl(T(\mathbf{0}, N\mathbf{v})-\mathbb{E} T(\mathbf{0}, N\mathbf{v})\Bigr)^2\,\mathbf{1}_{\{N^{\chi-\epsilon'}< |T(\mathbf{0}, N\mathbf{v})-\mathbb{E} T(\mathbf{0}, N\mathbf{v})|\le N^{\chi+\frac{\epsilon'}{2}}\}}\Bigr] \\
&\quad+\mathbb{E}\Bigl[\Bigl(T(\mathbf{0}, N\mathbf{v})-\mathbb{E} T(\mathbf{0}, N\mathbf{v})\Bigr)^2\,\mathbf{1}_{\{|T(\mathbf{0}, N\mathbf{v})-\mathbb{E} T(\mathbf{0}, N\mathbf{v})|> N^{\chi+\frac{\epsilon'}{2}}\}}\Bigr].
\end{split}
\]
On the event \(\{|T(\mathbf{0}, N\mathbf{v})-\mathbb{E} T(\mathbf{0}, N\mathbf{v})|\le N^{\chi-\epsilon'}\}\) the squared deviation is at most \(N^{2(\chi-\epsilon')}\), while on the event
\(\{N^{\chi-\epsilon'}\le| T(\mathbf{0}, N\mathbf{v})-\mathbb{E} T(\mathbf{0}, N\mathbf{v})|\le N^{\chi+\frac{\epsilon'}{2}}\}\) it is at most \(N^{2(\chi+\frac{\epsilon'}{2})}\). Finally, by Assumption~\ref{cond: full concentration}, there exists \(\theta=\theta(\epsilon')>0\) such that
\[
\mathbb{P}\Bigl(|T(\mathbf{0}, N\mathbf{v})-\mathbb{E} T(\mathbf{0}, N\mathbf{v})|> N^{\chi+\frac{\epsilon'}{2}}\Bigr)\le e^{-N^\theta},
\]
so that, by the Cauchy--Schwarz inequality and \eqref{ass: weight distribution}, the contribution of the last term is at most \(4N^2 e^{-N^\theta/2}\leq N^{2(\chi-\epsilon')}\). Thus, by Assumption~\ref{cond: full concentration},  we have
\[
\begin{split}
N^{2\chi-\epsilon'} &\leq\mathbb{E}\Bigl[\Bigl(T(\mathbf{0}, N\mathbf{v})-\mathbb{E} T(\mathbf{0}, N\mathbf{v})\Bigr)^2\Bigr]\\
&\le 2 N^{2(\chi-\epsilon')}
+N^{2(\chi+\frac{\epsilon'}{2})}\,\mathbb{P}\Bigl(N^{\chi-\epsilon'}\le |T(\mathbf{0}, N\mathbf{v})-\mathbb{E} T(\mathbf{0}, N\mathbf{v})|\le N^{\chi+\frac{\epsilon'}{2}}\Bigr).
\end{split}
\]
Subtracting \(N^{2(\chi-\epsilon')}\) from both sides and rearranging yields
\[
\mathbb{P}\Bigl(|T(\mathbf{0}, N\mathbf{v})-\mathbb{E} T(\mathbf{0}, N\mathbf{v})|\ge N^{\chi-\epsilon'}\Bigr)
\ge \frac{N^{2\chi-\epsilon'}-2N^{2(\chi-\epsilon')}}{N^{2(\chi+\frac{\epsilon'}{2})}}\geq 2^{-1} N^{-2\epsilon'}.
\]
In particular, for $N$ large enough, we have
\[
\mathbb{E}\Bigl[|T(\mathbf{0}, N\mathbf{v})-\mathbb{E} T(\mathbf{0}, N\mathbf{v})|\Bigr]
\ge N^{\chi-\epsilon'} \mathbb{P}\Bigl(|T(\mathbf{0}, N\mathbf{v})-\mathbb{E} T(\mathbf{0}, N\mathbf{v})|\ge N^{\chi-\epsilon'}\Bigr)\geq 2^{-1} N^{\chi-3\epsilon'}.
\]
Using the fact that any mean-zero random variable $X$ has $\E |X| = 2\E X_+$, we have
\al{
4^{-1} N^{\chi-3\epsilon'}&\leq \mathbb{E}\Bigl[(T(\mathbf{0}, N\mathbf{v})-\mathbb{E} T(\mathbf{0}, N\mathbf{v}))_+\Bigr]\\
&=\mathbb{E}\Bigl[(T(\mathbf{0}, N\mathbf{v})-\mathbb{E} T(\mathbf{0}, N\mathbf{v}))_+ \mathbf{1}_{\{T(\mathbf{0}, N\mathbf{v})-\mathbb{E} T(\mathbf{0}, N\mathbf{v})\leq N^{\chi-4\epsilon'}\}}\Bigr] \\
&\qquad +\mathbb{E}\Bigl[(T(\mathbf{0}, N\mathbf{v})-\mathbb{E} T(\mathbf{0}, N\mathbf{v})) \mathbf{1}_{\{T(\mathbf{0}, N\mathbf{v})-\mathbb{E} T(\mathbf{0}, N\mathbf{v})\geq N^{\chi+\epsilon'}\}}\Bigr] \\
&\qquad\qquad + \mathbb{E}\Bigl[(T(\mathbf{0}, N\mathbf{v})-\mathbb{E} T(\mathbf{0}, N\mathbf{v})) \mathbf{1}_{\{N^{\chi-4\epsilon'}<T(\mathbf{0}, N\mathbf{v})-\mathbb{E} T(\mathbf{0}, N\mathbf{v})< N^{\chi+\epsilon'}\}}\Bigr] \\
&\leq N^{\chi-4\epsilon'}+ \mathbb{E}\Bigl[(T(\mathbf{0}, N\mathbf{v})-\mathbb{E} T(\mathbf{0}, N\mathbf{v})) \mathbf{1}_{\{T(\mathbf{0}, N\mathbf{v})-\mathbb{E} T(\mathbf{0}, N\mathbf{v})\geq N^{\chi+\epsilon'}\}}\Bigr] \\
&\qquad + N^{\chi+\epsilon'}\mathbb{P}\Bigl(T(\mathbf{0}, N\mathbf{v})-\mathbb{E} T(\mathbf{0}, N\mathbf{v})\ge N^{\chi-4\epsilon'}\Bigr). 
}
The second term (the expectation) on the right-hand side decays super polynomially due to Assumption~\ref{cond: full concentration} together with the Cauchy--Schwarz inequality as before. Finally, using $\E[T(\mathbf{0}, N\mathbf{v})]\geq N \mu(\mathbf{v})-1\geq N \mu(\mathbf{u})-1$, we conclude that
\al{
\mathbb{P}\Bigl(T(\mathbf{0}, N\mathbf{v})-\mu(N\mathbf{u})\ge N^{\chi-5\epsilon'}\Bigr) &\geq \mathbb{P}\Bigl(T(\mathbf{0}, N\mathbf{v})-\mathbb{E} T(\mathbf{0}, N\mathbf{v})\ge N^{\chi-4\epsilon'}\Bigr)\geq 8^{-1} N^{-3\epsilon'},
}
which gives the desired conclusion by replacing $5\epsilon'$ by $\epsilon'$
\end{proof}

Given $K\in\N$, $\alpha\in \R$, and $\mathbf{x}\in  \R^d$, define
\al{
S^\alpha_K(\mathbf{x}):=\mathbf{x}+ \{0\}\times [-K^{\frac{\chi+1}{2}(1-\alpha)},K^{\frac{\chi+1}{2}(1-\alpha)}]^{d-1}.
}
(See Figure~\ref{fig:L_i} for a depiction of $S^\alpha_K(\mathbf{x})$ in two dimensions.)

\begin{lem}
\label{lem: lower_bound_lemma_2}
For any $\epsilon>0$, for all $K\geq 1$ sufficiently large and for any $\hat{\mathbf{y}}\in  \{K\}\times \Z^{d-1}$, we have
    \[
    \mathbb{P}\Bigl(T\bigl(S^{3\epsilon}_K(\mathbf{0}),  S^{3\epsilon}_K(\hat{\mathbf{y}})\bigr)> K\mu(\mathbf{e}_1) + K^{\chi(1-\epsilon)}\Bigr)\geq \frac{1}{2}K^{-2}\,.
    \]
\end{lem}
\begin{proof}
Fix $\epsilon > 0$ small. Define
$$A_K:=\Bigl\{T\bigl(-K^{1-2\epsilon}\mathbf{e}_1,\hat{\mathbf{y}}+K^{1-2\epsilon}\mathbf{e}_1\bigr)> \bigl(K+2K^{1-2\epsilon}\bigr)\mu(\mathbf{e}_1)+ 3K^{\chi(1-\epsilon)}\Bigr\}\,.$$
By Lemma~\ref{lem:modified}, since $\mu(\hat{\mathbf{y}})\geq K \mu(\mathbf{e}_1)$ by symmetry, for all $K$ sufficiently large we have
$$\mathbb{P}(A_K)\ge K^{-2}\,.$$

We claim that  for sufficiently large $K$ and for any $\mathbf{x} \in S_K^{3\epsilon}(\mathbf{0})$ one has
\[
\mathbb{P}\Bigl(T\bigl(-K^{1-2\epsilon}\mathbf{e}_1, \mathbf{x}\bigr)\leq K^{1-2\epsilon}\mu(\mathbf{e}_1)+K^{\chi(1-\epsilon)}\Bigr) \geq 1 - \frac{1}{4}K^{-2}\,.
\]
First, by Assumption~\ref{cond: full concentration} and a union bound, for $K$ sufficiently large, one has
\begin{equation}
\label{eq: event_B_prob_mod}
\mathbb{P}\Bigl(\forall \mathbf{x} \in S_K^{3\epsilon}(\mathbf{0}),\,T\bigl(-K^{1-2\epsilon}\mathbf{e}_1, \mathbf{x}\bigr) \leq \mu\bigl(\mathbf{x} + K^{1-2\epsilon}\mathbf{e}_1\bigr) + \bigl|\mathbf{x} + K^{1-2\epsilon}\mathbf{e}_1\bigr|^{\chi(1+\epsilon)}\Bigr) \geq 1 - \frac{1}{4}K^{-2}\,.
\end{equation}
Next, by Assumption~\ref{assum: finite curvature}, whenever $K$ is sufficiently large, for $\mathbf{x} \in S_{K}^{3\epsilon}(\mathbf{0})$ we have
\begin{align}
|    \mu\bigl(\mathbf{x} + K^{1-2\epsilon}\mathbf{e}_1\bigr) - K^{1-2\epsilon} \mu(\mathbf{e}_1)|
    &= K^{1-2\epsilon} \Bigl|\mu\Bigl(\mathbf{e}_1 + \frac{\mathbf{x}}{K^{1-2\epsilon}}\Bigr) - \mu(\mathbf{e}_1)\Bigr|\nonumber\\[1mm]
    &\leq C \frac{|\mathbf{x}|^2}{K^{1-2\epsilon}}\nonumber\\[1mm]
    \label{eq: curvature_bound_B_mod}
    &\leq C(d-1) \frac{K^{(\chi + 1)(1-3\epsilon)}}{K^{1-2\epsilon}} \leq \frac{1}{2} K^{\chi(1 - \epsilon)}\,.
\end{align}
Also, using the fact that $\chi \leq 1$ (so that $(\chi + 1) / 2\leq 1$), we have
\begin{equation}
\label{eq: x_norm_mod}
\bigl|\mathbf{x} + K^{1-2\epsilon}\mathbf{e}_1\bigr|^{\chi(1+\epsilon)} \leq d^{\chi(1+\epsilon)} K^{\chi (1+\epsilon)(1-2\epsilon)} \leq \frac{1}{2} K^{\chi(1-\epsilon)},
\end{equation}
if $K$ is sufficiently large. Combining \eqref{eq: event_B_prob_mod}, \eqref{eq: curvature_bound_B_mod} and \eqref{eq: x_norm_mod}, we have for all large $K$ that
\[
\mathbb{P}\Bigl(\forall \mathbf{x} \in S_K^{3\epsilon}(\mathbf{0}),\,T\bigl(-K^{1-2\epsilon}\mathbf{e}_1, \mathbf{x}\bigr) \leq K^{1-2\epsilon}\mu(\mathbf{e}_1)+K^{\chi(1-\epsilon)}\Bigr) \geq 1 - \frac{1}{4}K^{-2}\,,
\]
and this proves the claim.

By symmetry, for $K$ sufficiently large and for any $\hat{\mathbf{y}}\in  \{K\}\times \Z^{d-1}$, we also have
\[
\mathbb{P}\Bigl(\forall \mathbf{y} \in  S_K^{3\epsilon}(\hat{\mathbf{y}}),\,T\bigl(\mathbf{y},\hat{\mathbf{y}} + K^{1-2\epsilon}\mathbf{e}_1\bigr)\leq K^{1-2\epsilon}\mu(\mathbf{e}_1)+K^{\chi(1-\epsilon)}\Bigr) \geq 1 - \frac{1}{4}K^{-2}\,.
\]

For any $\hat{\mathbf{y}}\in  \{K\}\times \Z^{d-1}$, define the event
\[
B_K :=
\left\{
\begin{array}{l}
    \forall \mathbf{x} \in S_K^{3\epsilon}(\mathbf{0}),\, T\bigl(-K^{1-2\epsilon}\mathbf{e}_1,\mathbf{x}\bigr)\leq K^{1-2\epsilon}\mu(\mathbf{e}_1)+K^{\chi(1-\epsilon)},\\[1mm]
  \forall \mathbf{y} \in  S_K^{3\epsilon}(\hat{\mathbf{y}}),\,  T\bigl(\mathbf{y},\hat{\mathbf{y}} + K^{1-2\epsilon}\mathbf{e}_1\bigr)\leq K^{1-2\epsilon}\mu(\mathbf{e}_1)+K^{\chi(1-\epsilon)}
  \end{array}
\right\}\,.
\]
By the claim, we can choose $K$ sufficiently large so that
\[
\mathbb{P}(A_K\cap B_K) \geq \mathbb{P}(A_K) -\mathbb{P}(B_K^c)\geq  K^{-2} - \frac{K^{-2}}{4}-\frac{K^{-2}}{4}= \frac{K^{-2}}{2}\,.
\]
On the event \(A_K\cap B_K\), by triangle inequality, for any \(\mathbf{x} \in S^{3\epsilon}_K(\mathbf{0})\) and \(\mathbf{y} \in  S^{3\epsilon}_K(\hat{\mathbf{y}})\) we have
\begin{align*}
    T(\mathbf{x},\mathbf{y})&\geq T\bigl(-K^{1-2\epsilon}\mathbf{e}_1,\hat{\mathbf{y}}+K^{1-2\epsilon}\mathbf{e}_1\bigr)
    -T\bigl(-K^{1-2\epsilon}\mathbf{e}_1, \mathbf{x}\bigr)
    -T\bigl(\mathbf{y},\hat{\mathbf{y}}+K^{1-2\epsilon}\mathbf{e}_1\bigr)\\[1mm]
    &> \Bigl(K+2K^{1-2\epsilon}\Bigr)\mu(\mathbf{e}_1)+ 3K^{\chi(1-\epsilon)}
    -2\Bigl(N^{1-2\epsilon}\mu(\mathbf{e}_1)+K^{\chi(1-\epsilon)}\Bigr)\\[1mm]
    &\geq K\mu(\mathbf{e}_1)+ K^{\chi(1-\epsilon)}\,.
\end{align*}
In other words,
\[
\mathbb{P}\Bigl(T\bigl(S^{3\epsilon}_K(\mathbf{0}),  S^{3\epsilon}_K(\hat{\mathbf{y}})\bigr)> K\mu(\mathbf{e}_1) + K^{\chi(1-\epsilon)}\Bigr) \geq \prob(A_K\cap B_K) \geq \frac{1}{2K^2}.
\]
This completes the proof of the lemma.
\end{proof}

\begin{lem}\label{lem: S probab}
    For any $\epsilon>0$, for all $K$ sufficiently large, we have
    \[
    \mathbb{P}\Bigl(T\bigl(S^{3\epsilon}_{K}(\mathbf{0}),  \{K\}\times \Z^{d-1}\bigr)> K\mu(\mathbf{e}_1) + K^{\chi(1-\epsilon)}\Bigr)\geq e^{-K^{6d\epsilon}}.
    \]
\end{lem}
\begin{proof}
We consider two events:
    \al{
    A'_K &:=\{T\bigl(S^{3\epsilon}_{K}(\mathbf{0}),S^{-\epsilon}_{K}(K\mathbf{e}_1)\bigr)> K\mu(\mathbf{e}_1) + K^{\chi(1-\epsilon)}\}, \\
    B'_K& := \{T\bigl(S^{3\epsilon}_{K}(\mathbf{0}),(\{K\}\times \Z^{d-1})\setminus S^{-\epsilon}_{K}(K\mathbf{e}_1)\bigr)> K\mu(\mathbf{e}_1) + K^{\chi(1-\epsilon)}\}.
    }
    so that $A'_K\cap B'_K$ implies the event inside the probability in the statement of the lemma. Since $\mathbb{P}(A'_K \cap B'_K)\geq \mathbb{P}(A'_K)\mathbb{P}(B'_K)$ by the FKG inequality, it suffices to show
    \aln{
    &\mathbb{P}(A'_K)\geq 2 e^{-K^{6d\epsilon}}\,,\label{eq: estimate A'}\\
    &\mathbb{P}(B'_K)\geq \frac{1}{2}\,.  \label{eq: estimate B'}
    }

We first estimate $\mathbb{P}(A'_K)$. Since $\chi\leq 1$, for $K$ large enough, we have
\[
S^{-\epsilon}_{K}(K\mathbf{e}_1)\subset \bigcup_{\mathbf{i}\in \{0\}\times\Iintv{- K^{5\epsilon}, K^{5\epsilon}}^{d-1}}  S^{3\epsilon}_{K}(K\mathbf{e}_1 + \mathbf{i} K^{\frac{\chi+1}{2}(1-4\epsilon)}).
\]
Hence, by the FKG inequality and Lemma~\ref{lem: lower_bound_lemma_2}, if $K$ is large enough, we have
\al{
\mathbb{P}(A_K'))&\geq \prod_{\mathbf{i}\in \{0\}\times\Iintv{- K^{5\epsilon}, K^{5\epsilon}}^{d-1}} \mathbb{P}(T\bigl(S^{3\epsilon}_{K}(\mathbf{0}),S^{3\epsilon}_{K}(K\mathbf{e}_1 + \mathbf{i} K^{\frac{\chi+1}{2}(1-4\epsilon)})\bigr)> K\mu(\mathbf{e}_1) + K^{\chi(1-\epsilon)})\\
&\geq \Big(\frac{1}{2}K^{-2}\Big)^{|\Iintv{- K^{5\epsilon}, K^{5\epsilon}}^{d-1}|}\geq 2 e^{-K^{6d\epsilon}}.
}

Next, we estimate $\mathbb{P}(B'_K)$. By a union bound,
\al{
  \mathbb{P}((B_K')^c)&= \mathbb{P}(\exists \mathbf{x}\in S^{3\epsilon}_{K}(\mathbf{0}) , \exists \mathbf{y}\in (\{K\}\times \Z^{d-1})\setminus S^{-\epsilon}_{K}(K\mathbf{e}_1),\,T\bigl(\mathbf{x},\mathbf{y}\bigr)\leq K\mu(\mathbf{e}_1) + K^{\chi(1-\epsilon)})\\
  &\leq \sum_{\mathbf{x}\in S^{3\epsilon}_{K}(\mathbf{0})} \sum_{\mathbf{y}\in (\{K\}\times \Z^{d-1})\setminus S^{-\epsilon}_{K}(\K\mathbf{e}_1)}\mathbb{P}(T\bigl(\mathbf{x},\mathbf{y}\bigr)\leq K\mu(\mathbf{e}_1) + K^{\chi(1-\epsilon)})\,.
}
We fix $\mathbf{x}\in S^{3\epsilon}_{K}(\mathbf{0})$. Then, for $\zeta$ small enough, 
\al{
&\sum_{\mathbf{y}\in (\{K\}\times \Z^{d-1})\setminus S^{-\epsilon}_{K}(K\mathbf{e}_1)}\mathbb{P}(T\bigl(\mathbf{x},\mathbf{y}\bigr)\leq K\mu(\mathbf{e}_1) + K^{\chi(1-\epsilon)}))\\
&\leq \sum_{\substack{\mathbf{y}\in \{K\}\times \Z^{d-1}:\\
|\mathbf{y}-K\mathbf{e}_1|_\infty \in \Iintv{\frac{1}{2}K^{\frac{\chi+1}{2} (1+\epsilon)},K^2}}}\mathbb{P}(T\bigl(\mathbf{x},\mathbf{y}\bigr)\leq  K\mu(\mathbf{e}_1) + K^{\chi(1-\epsilon)})\\
&\qquad + \sum_{k= K^{2}}^\infty \sum_{\substack{\mathbf{y} \in \{K\}\times \Z^{d-1}:\\
|\mathbf{y}-K\mathbf{e}_1|_\infty = k}}\mathbb{P}(T\bigl(\mathbf{x},\mathbf{y}\bigr)\leq K\mu(\mathbf{e}_1) + K^{\chi(1-\epsilon)}).
}
By Assumption~\ref{ass: positive curvature},  there exists $c'=c'(\epsilon)>0$ such that  $\mu(\mathbf{x}-\mathbf{y})\geq K\mu(\mathbf{e}_1) + \widetilde{K}^{\chi(1+c')}$ for any $\mathbf{y}\in \{K\}\times \Z^{d-1}$ satisfying $|\mathbf{y}|_\infty \geq \frac{1}{2} K^{\frac{\chi+1}{2} (1+\epsilon)}$. Hence, using Assumption~\ref{cond: full concentration}, the first term on the right-hand side is bounded from above by $\widetilde{K}^{2d} e^{-\widetilde{K}^c}$ for some $c>0$. Using \eqref{eq: Kesten Propositin 5.8}, the second term on the right-hand side is bounded from above by 
\al{
&\sum_{k= K^{2}}^\infty e^{-c' k} \leq K^{-2d}.
} 
Summing over $\mathbf{x}$ gives the desired bound.
\end{proof}

\subsection{Proof of Theorem~\ref{thm: key prop for lower bound}-(2)}\label{sec:key-lower-bound-1}
For simplicity, we will assume $\mathbf{u}=\mathbf{e}_1$. We take $A>1$ as in Lemma~\ref{lem:exit-geodesics}. We also take
\[	\widetilde{K} := \widetilde{K}(\zeta) :=\lfloor \zeta^{-\frac{1}{1-\chi}}\rfloor.
\]
We will only consider the case that $\zeta > 0$ being small (so that $\tilde{K}$ is a large integer).
\begin{lem}\label{cor: Li estimate}
Let \(\zeta > 0\) be small. For \(i = 0, 1, \ldots, \lfloor N/\widetilde{K} \rfloor - 1\), define
\[
L_i := \{i\widetilde{K}\} \times [-A\sqrt{\zeta} N, A\sqrt{\zeta}N]^{d-1}\,.
\]
      For any $\varepsilon>0$, for all $\zeta$ sufficiently small depending on $\varepsilon$, if $N$ is large enough, then  we have for any $i\in \Iintv{0,N/\widetilde{K} - 1}$, 
    \al{
   & \mathbb{P}\Bigl(T\bigl(L_i,  \{(i+1)\widetilde{K}\}\times \Z^{d-1})\bigr)> \widetilde{K}\mu(\mathbf{e}_1) + \widetilde{K}^{\chi(1-\epsilon)}\Bigr)\geq e^{-(4AN)^{d-1}\zeta^{\frac{d-1-9d\varepsilon}{1-\chi}}}\,.
    }
 \end{lem}
\begin{proof}

By translation invariance, we may assume $i=0.$ For
$
\mathbf{i} \in \Z^{d-1}\,,
$ we  
write
\begin{align*}
\tilde{L}_0(\mathbf{i}) &:= \Bigl\{0\Bigr\} \times \Bigl(2\mathbf{i}  \widetilde{K}^{\frac{\chi+1}{2}(1-3\varepsilon)}
+\Bigl[- \widetilde{K}^{\frac{\chi+1}{2}(1-3\varepsilon)},  \widetilde{K}^{\frac{\chi+1}{2}(1-3\varepsilon)}\Bigr]^{d-1}\Bigr).
\end{align*}
By the FKG inequality, \begin{align}
 & \mathbb{P}\Bigl(T(L_0, \{\widetilde{K}\}\times \Z^{d-1}) \geq \widetilde{K}\mu(\mathbf{e}_1) + \widetilde{K}^{\chi(1-\varepsilon)}\Bigr)\nonumber\\[1mm]
&\ge \mathbb{P}\Bigl(\forall\, \mathbf{i}\in \Iintv{-\frac{A\sqrt{\zeta}N}{ \widetilde{K}^{\frac{\chi+1}{2}(1-3\varepsilon)}},\,\frac{A\sqrt{\zeta}N}{ \widetilde{K}^{\frac{\chi+1}{2}(1-3\varepsilon)}}}^{d-1},\; T\bigl(\tilde{L}_0(\mathbf{i}), \{\widetilde{K}\}\times \Z^{d-1}\bigr) \geq \widetilde{K}\mu(\mathbf{e}_1) + \widetilde{K}^{\chi(1-\varepsilon)}\Bigr)\nonumber\\[1mm]
\label{eq: FKG_last_line_mod}
&\ge \prod_{\mathbf{i}\in \Iintv{-\frac{A\sqrt{\zeta}N}{ \widetilde{K}^{\frac{\chi+1}{2}(1-3\varepsilon)}},\,\frac{A\sqrt{\zeta}N}{ \widetilde{K}^{\frac{\chi+1}{2}(1-3\varepsilon)}}}^{d-1}}
\mathbb{P}\Bigl(T\bigl(\tilde{L}_0(\mathbf{i}),\{\widetilde{K}\}\times \Z^{d-1}\bigr) \geq \widetilde{K}\mu(\mathbf{e}_1) + \widetilde{K}^{\chi(1-\varepsilon)}\Bigr)\,.
\end{align}
Recall that \(\widetilde{K}=\lfloor\zeta^{-\frac{1}{1-\chi}}\rfloor\). By Lemma~\ref{lem: S probab}, \eqref{eq: FKG_last_line_mod} is further bounded from below by 
\al{
\prod_{\mathbf{i}\in \Iintv{-\frac{A\sqrt{\zeta}N}{ \widetilde{K}^{\frac{\chi+1}{2}(1-3\varepsilon)}},\,\frac{A\sqrt{\zeta}N}{ \widetilde{K}^{\frac{\chi+1}{2}(1-3\varepsilon)}}}^{d-1}}
\Bigl(e^{-\widetilde{K}^{6d\varepsilon}}\Bigr)
&\geq e^{-(4A\sqrt{\zeta}N)^{d-1}\,\widetilde{K}^{6d\varepsilon - \frac{(\chi+1)(d-1)(1-3\varepsilon)}{2}}}\\
&\geq e^{-(4AN)^{d-1}\zeta^{\frac{d-1-9d\varepsilon}{1-\chi}}}.
}
\end{proof}

\begin{figure}[t]
    \centering
    \includegraphics[width=0.7\linewidth]{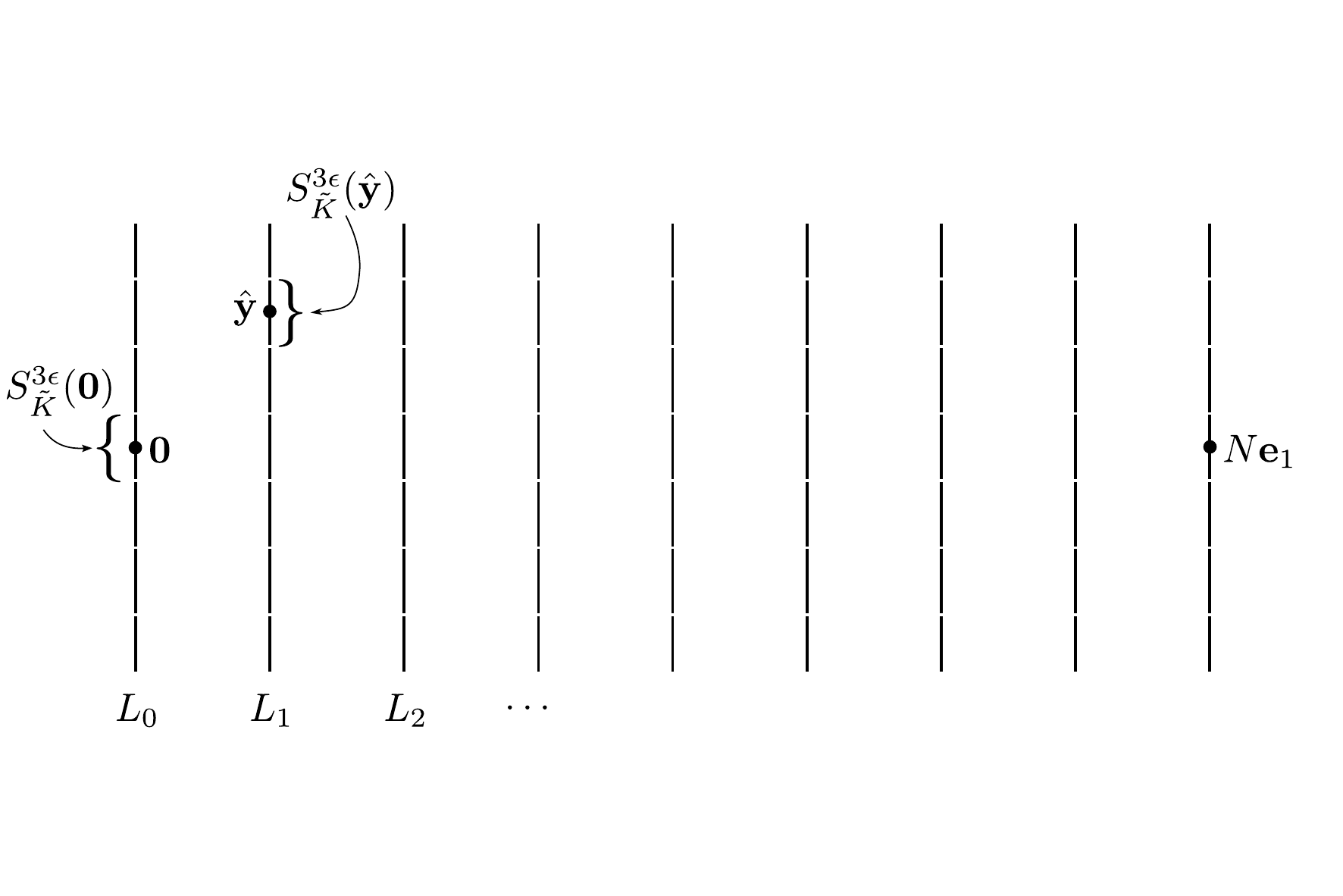}
    \caption{Depiction of the sets $L_i$ and $S^{3\epsilon}_{\widetilde{K}}(\mathbf{x})$. Each dashed line is $L_i$, and each line segment in each $L_i$ is some $S^{3\epsilon}_{\widetilde{K}}(\mathbf{x})$. By definition, $S^{3\epsilon}_{\widetilde{K}}(\mathbf{x})$ is a $(d-1)$-dimensional box of side length ${\widetilde{K}}^{\frac{\chi+1}{2}(1-\epsilon)}$ centered at $\mathbf{x}$. In order to bound the passage time from $L_0$ to $\{\widetilde{K}\} \times \mathbb{Z}^{d-1}$, it suffices to control the passage time between those $S^{3\epsilon}_{\widetilde{K}}(\hat{\mathbf{x}})$ in $L_0$ (which corresponds to $\tilde{L}_0(\mathbf{i})$ in the proof of Lemma~\ref{cor: Li estimate}) and those $S^{3\epsilon}_{\widetilde{K}}(\mathbf{\hat{\mathbf{y}}})$ in $L_1$, and this is handled in Lemma~\ref{lem: lower_bound_lemma_2}.}
    \label{fig:L_i}
\end{figure}

\begin{proof}[Proof of Theorem~\ref{thm: key prop for lower bound}-(2)]
    We are now ready to prove Theorem~\ref{thm: key prop for lower bound}-(2). Consider the events
\[
E_{N, i} := \Bigl\{T(L_{i-1}, \{i\widetilde{K}\}\times \Z^{d-1}) \geq \widetilde{K}\mu(\mathbf{e}_1) + \widetilde{K}^{\chi(1-\varepsilon)}\Bigr\} \quad \text{for } i = 1, 2, \ldots, \left\lceil{N}/{\widetilde{K}}\right\rceil - 1,
\]
and
\[
F_N := \Bigl\{\text{any path } \gamma:\mathbf{0}\to  N\mathbf{e}_1 \text{ with } \gamma\not\subseteq \R\times [-A\sqrt{\zeta} N, A\sqrt{\zeta}N]^{d-1} \text{ satisfies } T(\gamma) > (\mu(\mathbf{e}_1) + \zeta)N\Bigr\}\,.
\]
Suppose that the event
\[
\Bigl(\bigcap_{i=1}^{\lceil N/\widetilde{K}\rceil-1} E_{N, i}\Bigr) \cap F_N
\]
occurs. Let \(\gamma_N \in \mathbb{G}(\mathbf{0},N\mathbf{e}_1)\) be an optimal path. If \(\gamma_N \not\subseteq \R\times [-A\sqrt{\zeta} N, A\sqrt{\zeta}N]^{d-1}\), then the occurrence of \(F_N\) implies that
\[
T(\mathbf{0}, N\mathbf{e}_1) \ge (\mu(\mathbf{e}_1) + \zeta^{1+\frac{\epsilon}{1-\chi}})N\,.
\]
Otherwise, if all \(\gamma_N\subseteq \R\times [-A\sqrt{\zeta} N, A\sqrt{\zeta}N]^{d-1}\), then for \(N\) large the event \(\bigcap_{i=1}^{\lceil N/\widetilde{K}\rceil-1} E_{N, i}\) implies
\begin{align*}
T(\mathbf{0}, N\mathbf{e}_1) &\ge \sum_{i=1}^{\lceil N/\widetilde{K}\rceil - 1} T(L_{i-1},  \{i\widetilde{K}\}\times \Z^{d-1})\\[1mm]
&\ge (\lceil N/\widetilde{K}\rceil  -1)\Bigl(\widetilde{K}\mu(\mathbf{e}_1) + \widetilde{K}^{\chi(1-\epsilon)}\Bigr) \\
&\ge N\mu(\mathbf{e}_1) + \lceil N/\widetilde{K}\rceil\,\widetilde{K}^{\chi(1-\epsilon)}  -\Bigl(\widetilde{K}\mu(\mathbf{e}_1) + \widetilde{K}^{\chi(1-\epsilon)}\Bigr)\\
&\ge N\mu(\mathbf{e}_1) + \zeta^{\frac{1-\chi(1-\epsilon)}{1-\chi}}N  -\Bigl(\widetilde{K}\mu(\mathbf{e}_1) + \widetilde{K}^{\chi(1-\epsilon)}\Bigr)\ge  N\mu(\mathbf{e}_1) + \zeta^{1 + \frac{\epsilon}{1-\chi}}N\,,
\end{align*}
where we have used $\chi<1$ and $N\gg \widetilde{K}$ in the last line. Thus, by the FKG inequality and translation invariance,
\[
\mathbb{P}\Bigl(T(\mathbf{0}, N\mathbf{e}_1) \geq N\mu(\mathbf{e}_1) + \zeta^{1 + \frac{\epsilon}{1-\chi}}N\Bigr)
\ge \mathbb{P}(E_{N, 1})^{\lceil N/\widetilde{K}\rceil}\,\mathbb{P}(F_N)\,.
\]
By Lemma~\ref{lem:exit-geodesics}, for fixed \(\zeta>0\) one has \(\mathbb{P}(F_N)\ge 1/2\) for all \(N\) large. On the other hand, by Lemma~\ref{cor: Li estimate}, for $N$ large enough depending on $\zeta$, we estimate
\al{
 \mathbb{P}(E_{N, 1})^{\lceil N/\widetilde{K}\rceil}
 &\geq e^{-(\lceil N/\widetilde{K}\rceil)(4AN)^{d-1}\zeta^{\frac{d-1-9d\epsilon}{1-\chi}}}\\
 &\geq e^{- 4^d  A^d  N^{d} \zeta^{\frac{d-9d\epsilon}{1-\chi}}}.
}

Replacing \(\zeta^{1+\frac{\epsilon}{1-\chi}}\) by \(\zeta\),  one obtains
\[
\mathbb{P}\Bigl(T(\mathbf{0}, N\mathbf{e}_1) \ge N\mu(\mathbf{e}_1) + \zeta N\Bigr)
\ge \frac{1}{2}\exp\Bigl(-4^d A^d \,N^d\,\zeta^{\left(\frac{d-9d\epsilon}{1-\chi}\right)
\left(1+ \frac{\epsilon}{1-\chi}\right)^{-1}}\Bigr)\,.
\]
Since \(\epsilon>0\) was arbitrary, taking logarithms on both sides, dividing both sides by $N$ and taking $N\to\infty$ finishes the proof.
\end{proof}

\subsection{Proof of Theorem~\ref{thm: key prop for lower bound}-(1)}\label{sec:key-lower-bound-2}
Now we give a lower bound for the upper tail moderate deviation
\[
\mathbb{P}\Bigl(T(\mathbf{0},N\mathbf{e}_1)-\mu(\mathbf{e}_1) N > N^a\Bigr)
\]
instead of the rate function. The proof is essentially the same as that of Theorem~\ref{thm: key prop for lower bound}-(2) and requires only the following modifications: replace every occurrence of \(\zeta\) by \(N^{a-1}\), set \(\widetilde{K} = \lfloor N^{\frac{1-a}{1-\chi}}\rfloor\), and apply Lemma~\ref{lem: exit geodesics with a} instead of Lemma~\ref{lem:exit-geodesics}. Then, given any $\varepsilon > 0$, there exists a constant \(c_1>0\) (possibly depending on \(\epsilon\)) and for all large \(N\),
\[
\mathbb{P}\Bigl(T(\mathbf{0}, N\mathbf{e}_1) \ge N\mu(\mathbf{e}_1) + N^a\Bigr)
\ge \exp\Bigl(-c_1\,N^{d+(a-1)\left(\frac{d-6\epsilon}{1-\chi}\right)
\left(1+ \frac{\epsilon}{1-\chi}\right)^{-1}}\Bigr)\,.
\]
Since $d+(a-1)\left(\frac{d}{1-\chi}\right) = \frac{d(a-\chi)}{1-\chi}$, this implies the desired result by adjusting \(\epsilon>0\) appropriately.\qed

\subsection{Remark on general directions for lower bound}\label{section: lower bound general direction}
The lower bound established above for the direction \(\mathbf{e}_1\) extends to an arbitrary direction \(\mathbf{u}\in S^{d-1}\). We will focus on the rate function, since the upper tail probability can be handled in a similar way.

In the general case one also partitions the region from \(\mathbf{0}\) to \(N\mathbf{u}\) into mesoscopic blocks. For a base point \(\mathbf{x}\in\R^d\), an interval \(I\subset\R\), and a radius \(h>0\) (to be chosen appropriately), recall the definition of the tilted cylinder $\mathrm{Cyl}_{\mathbf{x},\mathbf{u}}(I,h)$. Let $\widetilde{K}:=\widetilde{K}(\zeta):=\lfloor \zeta^{-\frac{1}{1-\chi}}\rfloor $ as before. Subdivide the segment \(I = [0,N]\) into \(\lceil N/\widetilde{K}\rceil\) blocks of length \(\widetilde{K}\) in the \(\mathbf{u}\)-direction. In each block, we replace each $S_{\widetilde{K}}^{3\varepsilon}(\mathbf{x})$ by a tilted cylinder
$
\mathrm{Cyl}_{\mathbf{x},\mathbf{u}}([0,\widetilde{K}],\,\widetilde{K}^{\frac{\chi+1}{2}(1-3\epsilon)}).
$ 
Essentially the same argument as before shows that
\[
\mathbb{P}\Bigl(T(\mathbf{0}, N\mathbf{u})\ge N\mu(\mathbf{u})+\zeta N\Bigr)
\ge \exp\Bigl(-\,N^d\,\zeta^{\gamma}\Bigr)
\]
for some constant \(\gamma\geq \frac{d}{1-\chi}(1-o(1))\) as $\zeta \to 0$.

\section{Lower bound for lower tail moderate deviations}\label{sec:final-proof}
Finally, we consider lower bound for the lower tail moderate deviations. The proof is very straightforward and the results follow directly from the FKG inequality.
\subsection{Proof of Theorem~\ref{thm: lower bound for  lower tail MD}-(1)}
Let 
\[
\widetilde{J}_* := \lceil N^{\frac{a-\underline{\chi}}{1-\underline{\chi}}}\rceil \quad \text{and} \quad \widetilde{K}_* := \frac{N}{\widetilde{J}_*}\leq N^{\frac{1-a}{1-\underline{\chi}}}.
\]
For $i\in\Iintv{0,\widetilde{J}_*-1}$, define
\[
E_i = \Bigl\{T\bigl(i\widetilde{K}\mathbf{u},(i+1)\widetilde{K}\mathbf{u}\bigr) < \widetilde{K}_*\,\mu(\mathbf{u}) - \widetilde{K}_*^{\underline{\chi}}\Bigr\}.
\]
Since $\widetilde{K}_*^{\underline{\chi}} \widetilde{J}_* = N^{\underline{\chi}} \widetilde{J}_*^{1-\underline{\chi}} \geq N^{a}$, for any $\epsilon>0$, if $\delta>0$ is small enough and $N$ is large enough, then we have
\al{
\P\Bigl(T(\mathbf{0}, N\mathbf{u}) < N\,\mu(\mathbf{u}) - N^{a}\Bigr)
&\ge \P\Bigl(\bigcap_{i=0}^{\widetilde{J}_*-1} E_i\Bigr)\\
&\ge \prod_{i=0}^{\widetilde{J}_*-1} \mathbb{P}(E_i)\\
&\ge e^{-\widetilde{J}_* \widetilde{K}_*^\delta}
\ge e^{-N^{(1+\epsilon)\frac{a-\underline{\chi}}{1-\underline{\chi}}}},
}
where the second inequality follows from the FKG inequality, and the third one uses Assumption~\ref{ass: lower bound for lower tail deviations}. \qed
\subsection{Proof of Theorem~\ref{thm: lower bound for  lower tail MD}-(2)}
The proof is similar to that before.  Let 
\[
\widetilde{J}_\sharp = \lceil \zeta^{-\frac{1}{1-\underline{\chi}}} N\rceil \quad \text{and} \quad \widetilde{K}_\sharp := \frac{N}{\widetilde{J}_\sharp}\leq \zeta^{\frac{1}{1-\underline{\chi}}}.
\]
For $i\in\Iintv{0,\widetilde{J}_\sharp-1}$, define
\[
E_i = \Bigl\{T\bigl(i\widetilde{K}_\sharp\mathbf{u},(i+1)\widetilde{K}_\sharp\mathbf{u}\bigr) < \widetilde{K}_\sharp\,\mu(\mathbf{u}) - \widetilde{K}_\sharp^{\underline{\chi}}\Bigr\}.
\]
Since $\widetilde{K}_\sharp^{\underline{\chi}} \widetilde{J}_\sharp = N^{\underline{\chi}}\, \widetilde{J}_\sharp^{1-\underline{\chi}} \geq \zeta N$, for any $\epsilon>0$, if $\delta$ is small enough and $N$ is large enough, then we have
\al{
\P\Bigl(T(\mathbf{0}, N\mathbf{u}) < N\,\mu(\mathbf{u}) - \zeta N\Bigr)
&\ge \P\Bigl(\bigcap_{i=0}^{\widetilde{J}_\sharp-1} E_i\Bigr)\\
&\ge \prod_{i=0}^{\widetilde{J}_\sharp-1} \mathbb{P}(E_i)\\
&\ge e^{-\widetilde{J}_\sharp \widetilde{K}_\sharp^\delta}
\ge e^{-\zeta^{(1-\epsilon)\frac{1}{1-\underline{\chi}}}\,N},
}
where the second inequality follows from the FKG inequality.\qed

\vspace{10pt}
\noindent {\bf Acknowledgments.} The authors would like to thank the organizers of the program {\it Random Interacting Systems, Scaling Limits, and Universality} held at the Institute for Mathematical Sciences, National University of Singapore, in December 2023, where the authors met and eventually initiated this project. They would also like to thank Barbara Dembin for her valuable comments on the introduction. W.-K.L.\ is supported by the National Science and Technology Council in Taiwan grant
number 113-2115-M-002-009-MY3. S.N.\ is supported by JSPS KAKENHI 22K20344 and 24K16937.

\appendix
    \bibliographystyle{plain}
    \bibliography{ref}

\end{document}